\documentclass[1 [leqno,11pt]{amsart}
\usepackage{amssymb}
\usepackage{amssymb, amsmath}
\def\inte#1{
\displaystyle\mathop{#1\kern0pt}^\circ }

\let\pa=\partial
\let\al=\alpha

\let\d=\delta
\let\e=\varepsilon

\let\lam=\lambda
\let\r=\rho

\let\f=\frac

\let\p=\psi

\let\D=\Delta

\let\wt=\widetilde
\let\wh=\widehat

\def\cA{{\mathcal A}}
\def\cB{{\mathcal B}}
\def\cC{{\mathcal C}}

\def\cE{{\mathcal E}}
\def\cF{{\mathcal F}}

\def\cR{{\mathcal R}}
\def\cS{{\mathcal S}}

\def\pa{\partial}

\def\dH{\dot{H}}
\def\dB{\dot{B}}

\def\virgp{\raise 2pt\hbox{,}}
\def\cdotpv{\raise 2pt\hbox{;}}

\def\eqdefa{\buildrel\hbox{\footnotesize def}\over =}

\def\C{\mathop{\mathbb C\kern 0pt}\nolimits}
\def\DD{\mathop{\mathbb D\kern 0pt}\nolimits}
\def\EE{\mathop{{\mathbb E \kern 0pt}}\nolimits}
\def\K{\mathop{\mathbb K\kern 0pt}\nolimits}
\def\N{\mathop{\mathbb N\kern 0pt}\nolimits}
\def\Q{\mathop{\mathbb Q\kern 0pt}\nolimits}
\def\R{\mathop{\mathbb R\kern 0pt}\nolimits}
\def\SS{\mathop{\mathbb S\kern 0pt}\nolimits}
\def\ZZ{\mathop{\mathbb Z\kern 0pt}\nolimits}
\def\TT{\mathop{\mathbb T\kern 0pt}\nolimits}
\def\P{\mathop{\mathbb P\kern 0pt}\nolimits}

\newcommand{\la}{\lambda}

\newcommand{\Z}{{\ZZ}}

\def\dv{\mbox{div}}

\def\dive{\mathop{\rm div}\nolimits}

\def\Supp{\mathop{\rm Supp}\nolimits\ }

\def\no{\noindent}
\def\na{\nabla}
\def\p{\partial}

\newcommand{\w}[1]{\langle {#1} \rangle}
\newcommand{\beq}{\begin{equation}}
\newcommand{\eeq}{\end{equation}}
\newcommand{\ben}{\begin{eqnarray}}
\newcommand{\een}{\end{eqnarray}}
\newcommand{\beno}{\begin{eqnarray*}}
\newcommand{\eeno}{\end{eqnarray*}}
\newcommand{\andf}{\quad\hbox{and}\quad}

\newtheorem{defi}{Definition}[section]
\newtheorem{thm}{Theorem}[section]
\newtheorem{lem}{Lemma}[section]
\newtheorem{rmk}{Remark}[section]
\newtheorem{col}{Corollary}[section]
\newtheorem{prop}{Proposition}[section]

\newcommand{\vv}[1]{\boldsymbol{#1}}

\begin{document}

\title[Global small solutions to 3-D incompressible MHD type system]
{ Global small solutions  to three-dimensional incompressible MHD
system }
\author[L. Xu]{Li Xu}
\address[L. Xu]%
{LSEC, Institute of Computational Mathematics, Academy of Mathematics and Systems Science, CAS\\
Beijing 100190, CHINA}
\email{ xuliice@lsec.cc.ac.cn}
\author[P. Zhang]{Ping Zhang}%
\address[P. Zhang]
 {Academy of Mathematics and Systems Science and  Hua Loo-Keng Key Laboratory of Mathematics,
  Chinese Academy of Sciences, Beijing 100190, CHINA} \email{zp@amss.ac.cn}
\date{9/May/2013}
\maketitle
\begin{abstract}  In this paper, we consider the global wellposedness of
 3-D incompressible magneto-hydrodynamical  system with small and smooth initial
data.  The main difficulty of the proof lies in establishing the
global in time $L^1$  estimate for gradient of the velocity field
due to the strong degeneracy  and anisotropic spectral properties of
the linearized system. To achieve this and to avoid the difficulty
of propagating anisotropic regularity for the transport equation, we
first write our system \eqref{B1} in the Lagrangian formulation
\eqref{B11}. Then we employ anisotropic Littlewood-Paley analysis to
establish the key $L^1$ in time estimates to the  velocity and the
gradient of the pressure in the Lagrangian coordinate. With those
estimates, we prove the global wellposedness of \eqref{B11} with
smooth and small initial data by using the energy method. Toward
this,  we will have to use the algebraic structure of \eqref{B11} in
a rather crucial way. The global wellposedness of the original
system \eqref{B1} then follows by a suitable change of variables
together with a continuous argument. We should point out that
compared with the linearized systems of 2-D MHD equations in
\cite{XLZMHD1} and that of the 3-D modified MHD equations in
\cite{LZ}, our linearized system \eqref{B19} here is much more
degenerate, moreover, the formulation of the initial data for
\eqref{B11} is more subtle than that in \cite{XLZMHD1}.
\end{abstract}


 \qquad\qquad
estimates, Lagrangian coordinates\

\vskip 0.2cm


\setcounter{equation}{0}
\section{Introduction}
In this paper, we investigate the  global wellposedness of the
following three-dimensional incompressible magnetic hydrodynamical
 system (or MHD in short) with initial data being sufficiently close to the equilibrium state:
\begin{equation}\label{B1}
 \left\{\begin{array}{l}
\displaystyle \pa_t\vv b+\vv u\cdot\na\vv b=\vv b\cdot\na\vv u,\qquad (t,x)\in\R^+\times\R^3, \\
\displaystyle \pa_t\vv u +\vv u\cdot\na\vv u -\D\vv u+\na p=-\f12\na|\vv b|^2+\vv b\cdot\na\vv b, \\
\displaystyle \dv\,\vv u =\dv\,\vv b=0, \\
\displaystyle \vv b|_{t=0}=\vv b_0,\quad \vv u|_{t=0}=\vv u_0,
\end{array}\right.
\end{equation}
where $\vv b=(b^1,b^2,b^3)^T$ denotes the magnetic field, and  $\vv
u=(u^1,u^2,u^3)^T, p$ the velocity and scalar pressure of the fluid
respectively.  This MHD system \eqref{B1} with zero diffusivity in
the equation for the magnetic field can be applied to model plasmas
when the plasmas are strongly collisional, or the resistivity due to
these collisions are extremely small. One may check the references
\cite{CP, LL, Ca} for more detailed explanations to this system.

\medbreak

It has been a long-standing open problem that  whether or not
classical solutions of \eqref{B1} can develop finite time
singularities even in the two-dimensional case.  In the case when
there is full magnetic diffusion in \eqref{B1}, Duvaut and Lions
\cite{DL} established the local existence and uniqueness of solution
in the classical Sobolev space $H^s(\R^d)$, $s\geq d,$ they also
proved the global existence of solutions to this system with small
initial data; Sermange and Temam \cite{ST} proved the global unique
solution in the two space dimensions;
 Abidi and Paicu \cite{AP} proved similar result as in \cite{DL} for the so-called inhomogeneous MHD system
  with  initial data in the critical spaces.
With mixed partial dissipation and additional magnetic diffusion in
the two-dimensional MHD system, Cao and Wu \cite{CW} (see also
\cite{CRW}) proved that such a system is globally wellposed for any
data in $H^2(\R^2).$ Lin and the second author \cite{LZ} proved the
global wellposedness to a modified three-dimensional MHD system (3-D
version of \eqref{1.1} below) with initial data sufficiently close
to the equilibrium state. Lin and the authors \cite{XLZMHD1}
established the global existence of small solutions to the
two-dimensional MHD equations \eqref{B1}.

\medbreak

For the incompressible MHD equations \eqref{B1}, whether there is a
dissipation or not for the magnetic field is a very important
problem also from physics of plasmas. The heating of high
temperature plasmas by MHD waves is one of the most interesting and
challenging problems of plasma physics especially when the energy is
injected into the system at the length scales much larger than the
dissipative ones. It has been conjectured that in the
three-dimensional MHD system, energy is dissipated at a rate that is
independent of the ohmic resistivity \cite{ChCa}. In other words,
the viscosity (diffusivity) for the magnetic field equation can be
zero yet the whole system may still be dissipative. We shall justify
this conjecture for \eqref{B1} with initial data close enough to the
equilibrium state.

\medbreak

Notice that in two space dimensions,  $\dv\,\vv b=0$ implies the
existence of a scalar function $\phi$ so that $\vv
b=(\partial_2\phi, -\partial_1\phi)^T,$ and the  system \eqref{B1}
becomes
\begin{equation}\label{1.1}
 \left\{\begin{array}{l}
\displaystyle \pa_t \phi+\vv u\cdot\na\phi=0,\qquad (t,x)\in\R^+\times\R^2, \\
\displaystyle \pa_t \vv u +\vv u\cdot\na\vv u -\D\vv u+\na p=-\f12\na|\na\phi|^2-\dv\bigl[\na\phi\otimes\na\phi\bigr], \\
\displaystyle \dv\,\vv u = 0, \\
\displaystyle \phi|_{t=0}=\phi_0(x)=x_2+\psi_0(x),\quad  \vv
u|_{t=0}=\vv u_0,
\end{array}\right.
\end{equation}
The main idea in \cite{XLZMHD1} is first to seek another scalar
function $\wt{\phi}(x)=-x_1+\wt{\psi}_0$ so that \beq\label{a2}
\det\,U_0=1\quad \mbox{for}\quad
U_0=\begin{pmatrix} 1+\p_{x_2}\psi_0& \p_{x_2}\tilde\psi_0 \\
-\p_{x_1}\psi_0&1-\p_{x_1}\tilde\psi_0
\end{pmatrix}, \eeq
provided that $\psi_0$ is sufficiently small in some sense.  Then
the authors of  \cite{XLZMHD1} looked for a volume preserving
diffeomorphism in $\R^2,$ $X_0(y)=y+Y_0(y),$ so that
\beq\label{app0} U_0\circ X_0(y)=\na_y X_0(y)=I+\na_y Y_0(y). \eeq
Let $(Y(t,y), q(t,y))$ be determined by \beq\label{a6} \begin{split}
&X(t,y)=X_0(y)+\int_0^t\vv u(s,X(s,y))\,ds\eqdefa y+Y(t,y), \\
&q(t,y)\eqdefa(p+|\na\phi|^2)\circ X(t,y). \end{split} \eeq
\eqref{1.1} can be equivalently reformulated as
\beq\label{a14}\left\{\begin{aligned}
&Y_{tt}-\na_Y\cdot\na_Y Y_t-\p_{y_1}^2Y+\na_Yq=\vv 0, \qquad (t,y)\in\R^+\times\R^2,\\
&\na_Y\cdot Y_t=0,\\
&Y|_{t=0}=Y_0,\quad Y_t|_{t=0}=\vv u_0\circ X_0(y)\eqdefa Y_1,
\end{aligned}\right.\eeq
where $\na_Y\eqdefa\mathcal{A}_Y^T\na_y$ and \beq\label{a14a}
\mathcal{A}_Y\eqdefa\begin{pmatrix}
1+\p_{y_2}Y^2&-\p_{y_2}Y^1\\
-\p_{y_1}Y^2&1+\p_{y_1}Y^1
\end{pmatrix}.
\eeq In particular, the linearized system of \eqref{a14} reads
\beq\label{a12}\left\{\begin{aligned}
&Y_{tt}-\Delta_y Y_t-\p_{y_1}^2 Y=\vv f(Y,q),\qquad (t,y)\in\R^+\times\R^2,\\
&\na_y\cdot Y=\r(Y),\\
&Y|_{t=0}=Y_0,\quad Y_t|_{t=0}=Y_1.
\end{aligned}\right.\eeq
By using anisotropic Littlewood-Paley theory, the authors
\cite{XLZMHD1} first established the global wellposedness of
\eqref{a14} with small and smooth initial data, then they proved the
global wellposedness of \eqref{1.1} with sufficiently small data
$(\psi_0,\vv u_0)$ through a suitable changes of variables.

\medbreak

However, in the three-dimensional case, we can not find such an
equivalent formulation of \eqref{B1} as \eqref{1.1}. Instead, for
$\vv b_0-\vv e_3$ being sufficient small, we can find a
$\vv\Psi=(\psi_1,\psi_2,\psi_3)^T$ so that there holds \eqref{B3p1}.
Compared with \eqref{a2}, \eqref{B3p1} is a nonlinear system. With
this $\vv \Psi,$ we can define $ \bar{\vv b}_0$ and $\tilde{\vv
b}_0$ via \eqref{B3p2} so that the $3\times 3$ matrix $U_0\eqdefa
(\bar{\vv b}_0,\tilde{\vv b}_0,\vv b_0)$  satisfies \beq\label{B3}
\dv\bar{\vv b}=\dv\tilde{\vv b}=0,\quad\mbox{and}\quad \det\,U_0=1.
\eeq With thus obtained $U_0,$ we can find a 3-D volume preserving
diffeomorphism $X_0(y)=y+Y_0(y),$ and reformulate \eqref{B1} in the
Lagrangian coordinate \eqref{B11} with its linearized system
\eqref{B12}. We point out that one crucial idea in \cite{XLZMHD1} is
to use $\p_{y_1}Y^1+\p_{y_2}Y^2=\r(Y)$ to propagate the time
dissipative estimate of $\p_{y_1}Y^1$ to that of $\p_{y_2}Y^2.$
Notice that in the linearized system \eqref{B12}, one  only has time
dissipative estimate for $\p_{y_3}Y,$ and we can not use $\na_y\cdot
Y=\r(Y)$ to propagate the time dissipative estimate from
$\p_{y_3}Y^3$ to that of  $\p_{y_1}Y^1, \p_{y_2}Y^2$, which gives
rise to another difficulty in the analysis of three-dimensional MHD
system. And we will have to use the nonlinear structure of
\eqref{B11} in a rather crucial way so that the source term in
\eqref{B12} is still globally integrable in time. As in
\cite{XLZMHD1}, we shall first establish the global wellposedness of
\eqref{B11} with small and smooth initial data,  we then prove the
global existence of small solution to \eqref{B1} by a suitable
changes of variables along with a continuous argument.

\medbreak

We should remark that the system \eqref{1.1} is of interest not only
because it models the incompressible MHD equations, but also because
it arises in many other important applications. Moreover, its
nonlinear coupling structure is universal, see the recent survey
article \cite{Lin}. Indeed, the system \eqref{1.1}  resembles the
2-D viscoelastic fluid system:
\begin{eqnarray}
\left\{ \begin{array}{l} U_t+\vv u\cdot\nabla U=\nabla\vv u U,
\\
\vv u_t+\vv u\cdot\nabla\vv u+\nabla p=\D\vv u+\nabla\cdot(UU^T),\\
\dv\,\vv u=0,\\
U|_{t=0}=U_0,\quad\quad\vv u|_{t=0}=\vv u_0,
\end{array}\right.\label{eq1}
\end{eqnarray}
where $U$ denotes the deformation tensor, $\vv u$ is the fluid
velocity and $p$ represents the hydrodynamic pressure (we refer to
\cite{LLZ} and the references therein for more details).

In two space dimensions, when $\na \cdot U_0=0,$ it follows from
(\ref{eq1})  that  $\na\cdot U(t,x)=0$ for all $t>0.$ Therefore, one
can find a vector $\vv\phi=(\phi_1,\phi_2)^T$ such that
\begin{equation*}
U = \left( \begin{array}{rc} -\partial_2\phi_1 & -\partial_2\phi_2\\
\partial_1\phi_1 &  \partial_1\phi_2
\end{array}  \right).
\end{equation*}
Then (\ref{eq1}) can be equivalently reformulated as
\begin{eqnarray}\label{phi}
\left\{ \begin{array}{l} \vv\phi_t +\vv u\cdot \nabla\vv\phi =\vv0, \\
\vv u_t + \vv u\cdot\nabla\vv u + \nabla p = \Delta\vv u
-\sum_{i=1}^2\dive\big[\na\phi_i\otimes\na\phi_i\bigr],\\
 \dv\,\vv u
= 0,\\
\vv\phi|_{t=0}=\vv\phi_0,\quad\quad\vv u|_{t=0}=\vv u_0.
\end{array}\right.
\end{eqnarray}
The authors (\cite{LLZ}) established the global existence of smooth
solutions to the Cauchy problem in the entire space or on a periodic
domain for (\ref{phi}) in general space dimensions provided that the
initial data is sufficiently close to the equilibrium state (one may
check \cite{Ch-Zh, LLZhen} for the 3-D result). One sees the only
difference between \eqref{1.1} and \eqref{phi} lying in the fact
that $\phi$ is a scalar function in \eqref{1.1}, while
$\vv\phi=(\phi_1,\phi_2)^T$ is a vector-valued function with the
unit Jacobian in \eqref{phi}. However, it gives rise to an essential
difficulty in the analysis. In fact, there is a damping mechanism of
the system \eqref{phi} that can be seen  from the linearization of
the system $\p_t$ \eqref{phi}:
\begin{eqnarray}\label{dphi}
\left\{ \begin{array}{l} \vv\phi_{tt}-\D\vv\phi-\D\vv\phi_t+\na q  =\vv f, \\
\vv u_{tt}-\D\vv u -\D\vv u_t+\na p=\vv F,\\
\, \dv\,\vv u = 0.
\end{array}\right.
\end{eqnarray}
We also remark that the linearized system of \eqref{1.1} in 3-D
reads \beq \label{intro1}
\p_t^2\psi-(\p_{x_1}^2+\p_{x_2}^2)\psi-\D\p_t\psi=f. \eeq One may
check Remark 1.4 of \cite{LZ} for details. It is easy to observe
that our linearized system in \eqref{B12} is much more degenerate
than \eqref{dphi} and \eqref{intro1}.

\medbreak

As in \cite{XLZMHD1}, to describe  the initial data $\vv b_0$ in
\eqref{B1}, we need the following definition:

\begin{defi}\label{def1.1ad}
Let $\vv b_0=(b^1_0,b^2_0, b^3_0)^T$ be a smooth enough vector
field. We define its trajectory $X(t,x)$ by \beq\label{apa6}
\left\{\begin{array}{l}
\displaystyle \f{d X(t,x)}{dt}=\vv b_0(X(t,x)), \\
\displaystyle X(t,x)|_{t=0}=x.
\end{array}\right. \eeq
We call that $f$ and $\vv b_0$ are admissible on a domain $D$ of
$\R^3$ if there holds \beno \int_{\R}f(X(t,x))\,dt=0\quad\mbox{for
\, all }\quad x\in D. \eeno
\end{defi}

\begin{rmk}
As in \cite{XLZMHD1}, the condition that $f$ and $\vv b$ are
admissible on some set of $\R^3$ is to guarantee that
\beq\label{apa4}
b_0^1\p_{x_1}\psi+b_0^2\p_{x_2}\psi+b_0^3\p_{x_3}\psi=f \eeq has a
solution $\psi$ so that $\lim_{|x|\to \infty}\psi(x)=0.$ Let us take
$\vv b=(0,0,1)^T$ for example. In this case, \eqref{apa4} becomes
$\p_{x_3}\psi=f,$ which together with the condition
$\lim_{|x_3|\to\infty}\psi(x)=0$ ensures that \beno
\psi(x_h,x_3)=-\int_{x_3}^\infty
f(x_h,t)\,dt=\int_{-\infty}^{x_3}f(x_h,t)\,dt. \eeno We thus obtain
that $\int_{\R}f(x_h,t)\,dt=0,$ that is, $f$ and $(0,0,1)^T$ are
admissible on $\R^2\times\{0\}.$
\end{rmk}

\no{\bf Notations:} Let $X_1, X_2$ be Banach spaces, the norms
$\|\cdot\|_{X_1\cap X_2}\eqdefa\|\cdot\|_{X_1}+\|\cdot\|_{X_2}$ and
$\|\cdot\|_{L^p(\R^+;X_1\cap
X_2)}\eqdefa\|\cdot\|_{L^p(\R^+;X_1)}+\|\cdot\|_{L^p(\R^+;X_2)}$ for
$p\in[1,\infty]$.

\medbreak

We now state the main result of this paper:

\begin{thm}\label{th2}
{\sl Let $s_1>\f54,$ $s_2\in (-\f12,-\f14),$ and $p\in (\f32,2)$.
Let $s\geq s_1+2,$  let $(\vv b_0,\vv u_0)$ satisfy $\vv b_0-\vv
e_3\in B^{s_1+\f3p+\f12}_{p,1}\cap H^s(\R^3)$ for $\vv
e_3=(0,0,1)^T$, and $\vv u_0\in
\dH^{s_2}\cap\dB^{\f3p-1}_{p,1}(\R^3)$ with $\na\vv u_0\in
B^{s_1+\f3p-\f32}_{p,1}\cap H^{s-1}(\R^3)$ and
 \beq\label{A1} \|\vv b_0-\vv e_3\|_{B^{s_1+\f3p+\f12}_{p,1}}+\|\vv
u_0\|_{\dH^{s_2}\cap\dB^{\f3p-1}_{p,1}}+\|\na\vv
u_0\|_{B^{s_1+\f3p-\f32}_{p,1}}\leq c_0\eeq for some $c_0$
sufficiently small. We assume moreover that $\vv b_0-\vv e_3$ and
$\vv b_0$ are admissible on $\R^2\times\{0\}$ in the sense of
Definition \ref{def1.1ad} and $\Supp\bigl(\vv b_0-\vv
e_3\bigr)(x_1,x_2,\cdot)\subset [-K, K]$ for some positive constant
$K.$ Then
  \eqref{B1} has a unique global solution
$(\vv b, \vv u, p)$ (up to a constant for $p$) so that
\beq\label{th2wq}
\begin{split}
& \vv b-\vv e_3\in C([0,\infty); H^{s}(\R^3))\cap L^2(\R^+;\dH^{s_1+1}\cap\dH^{s_2+1}(\R^3)),\\
& \na p\in C([0,\infty); H^{s-1}(\R^3))\cap
L^2(\R^+;\dH^{s_1}\cap \dH^{s_2}(\R^3)),\\
&\vv u\in  C([0,\infty); H^{s}(\R^3)) \cap
L^1(\R^+;\dH^{s_1+2}\cap\dB^{\f52}_{2,1}(\R^3))\cap
L^2_{\mbox{loc}}(\R^+;\dot{H}^{s+1}(\R^3)).
\end{split}
\eeq
 Furthermore, there holds
 \beq\label{th2wr}
\begin{split}
&\|\vv b-\vv e_3\|_{L^\infty(\R^+;\dH^{s_1+1}\cap\dH^{s_2})} +\|\vv
u\|_{L^\infty(\R^+; \dH^{s_1+1}\cap\dH^{s_2})}
\\
&\quad+\|\vv b-\vv e_3\|_{L^2(\R^+;\dH^{s_1+1}\cap\dH^{s_2+1})}+\|
\vv
u\|_{L^2(\R^+;\dH^{s_1+2}\cap\dH^{s_2+1})}\\
&\quad+\|\vv u\|_{L^1(\R^+; \dH^{s_1+2}\cap\dB^{\f52}_{2,1})}+\|\na p\|_{L^2(\R^+;\dH^{s_1}\cap\dH^{s_2})}\\
&\leq C\bigl(\|\vv b_0-\vv e_3\|_{B^{s_1+\f3p+\f12}_{p,1}}+\|\vv
u_0\|_{\dH^{s_2}\cap\dB^{\f3p-1}_{p,1}}+\|\na\vv
u_0\|_{B^{s_1+\f3p-\f32}_{p,1}}\bigr).
\end{split}
\eeq }
\end{thm}

\begin{rmk} (1)
One may find the definitions of Besov spaces in Subsection
\ref{subsect3.2}. We remark that those technical assumptions on $\vv
b_0$ and $\vv u_0$ will be used to deal with the low frequency part
of $\vv b$ and $\vv u.$ For simplicity, we do not provide result on
the propagation of regularities for $\vv b_0-\vv e_3\in
B^{s_1+\f3p+\f12}_{p,1}(\R^3)$ and $\na\vv u_0\in
B^{s_1+\f3p-\f32}_{p,1}(\R^3).$

(2) Here we  point out that the estimate of $ \|\vv b-\vv
e_3\|_{L^2(\R^+;\dH^{s_1+1}\cap\dH^{s_2+1})}$ in \eqref{th2wr} is
not standard for the solutions of the transport equation in
\eqref{B1}. It is purely due to the coupling structure in
\eqref{B1}. And this estimate in some sense explains that the
magnetic field is indeed time dissipative even without resistivity
for the magnetic field. We shall go back to this point in our future
work.

(3) We can improve the condition that: $\Supp\bigl(\vv b_0-\vv
e_3\bigr)(x_1,x_2,\cdot)\subset [-K,K]$ for some positive number
$K$, in Theorem \ref{th2} by assuming appropriate decay of $\vv
b_0-\vv e_3$ with respect to $x_3$ variable. For a clear
presentation, we prefer not to present this technical part here.
\end{rmk}


Let us complete this section by the notation we shall use in this context.\\

\no{\bf Notation.} For any $s\in\R$, we denote by $H^s(\R^3)$ the
classical  $L^2$ based Sobolev spaces with the norm
$\|\cdot\|_{H^s},$ while $\dot{H}^s(\R^3)$ the classical homogenous
Sobolev spaces with the norm $\|\cdot\|_{\dot{H}^s}$. Let $A, B$ be
two operators, we denote $[A;B]=AB-BA,$ the commutator between $A$
and $B$. For $a\lesssim b$, we mean that there is a uniform constant
$C,$ which may be different on different lines, such that  $a\leq
Cb,$ and $a\sim b$ means that both $a\lesssim b$ and $b\lesssim a$.
We shall denote by $(a|b)$  the $L^2(\R^3)$ inner product of $a$ and
$b.$ $(d_{j,k})_{j,k\in\Z}$ (resp. $(c_j)_{j\in\Z}$) will be a
generic element of $\ell^1(\Z^2)$ (resp. $\ell^2(\Z))$ so that
$\sum_{j,k\in\Z}d_{j,k}=1$ (resp. $\sum_{j\in\Z}c_j^2=1).$ Finally,
we denote by $L^p_T(L^q_h(L^r_v))$ the space $L^p([0,T];
L^q(\R_{x_h}^2;L^r(\R_{x_3})))$ with $x_h=(x_1,x_2)$.

\setcounter{equation}{0}
\section{Lagrangain formulation of  \eqref{B1}}\label{sect2}

Motivated by \cite{XLZMHD1}, we are going to construct two vector
fields $\bar{\vv b}_0=(\bar{b}_0^1,\bar{b}_0^2,\bar{b}_0^3)^T$ and $
\tilde{\vv b}_0=(\tilde{b}_0^1,\tilde{b}_0^2,\tilde{b}_0^3)^T$ so
that the $3\times 3$ matrix $U_0\eqdefa (\bar{\vv b}_0,\tilde{\vv
b}_0,\vv b_0)$  satisfies \eqref{B3}.

\begin{prop}\label{LL1}
{\sl Let $s>2+\f3p$ and $p\in (\f32,2).$ Let $\vv b_0-\vv
e_3=(b_0^1,b_0^2,b_0^3-1)^T\in B^s_{p,1}(\R^3)$ with
\beq\label{LL1pq} \dv\,\vv b_0=0\quad\mbox{and}\quad \|(b_0^1,b_0^2,
b_0^3-1)\|_{B^s_{p,1}}\leq \e_0.\eeq  We assume moreover that $\vv
b_0-\vv e_3$ and $\vv b_0$ are admissible on $\R^2\times\{0\}$ in
the sense of Definition \ref{def1.1ad} and $\Supp\bigl(\vv b_0-\vv
e_3\bigr)(x_1,x_2,\cdot)\subset [-K, K]$ for some positive constant
$K.$  Then  for $\e_0$ sufficiently small, there exists a
$\vv\Psi=(\psi_1,\psi_2,\psi_3)^T$ which satisfies \beq\label{B3p0}
\|(\psi_1,\psi_2,\psi_3\|_{B^s_{p,1}}\leq C(K,\e_0)\|(b_0^1,b_0^2,
b_0^3-1)\|_{B^s_{p,1}}, \eeq and \beq\label{B3p1}\begin{split}
&b_0^1=\p_{x_2}\psi_1\p_{x_3}\psi_2+\p_{x_3}\psi_1(1-\p_{x_2}\psi_2),\quad
b_0^2=
\p_{x_3}\psi_1\p_{x_1}\psi_2+\p_{x_3}\psi_2(1-\p_{x_1}\psi_1),\\
&b_0^3=(1-\p_{x_1}\psi_1)(1-\p_{x_2}\psi_2)-\p_{x_2}\psi_1\p_{x_1}\psi_2,\quad\mbox{and}\quad
\det\,(I-\na_x\vv\Psi)=1. \end{split} \eeq Moreover, we define
\beq\label{B3p2}\begin{aligned}
\bar{\vv b}_0\eqdefa\bigl(&(1-\p_{x_2}\psi_2)(1-\p_{x_3}\psi_3)-\p_{x_3}\psi_2\p_{x_2}\psi_3, \p_{x_3}\psi_2\p_{x_1}\psi_3+\p_{x_1}\psi_2(1-\p_{x_3}\psi_3), \\
& \p_{x_1}\psi_2\p_{x_2}\psi_3+\p_{x_1}\psi_3(1-\p_{x_2}\psi_2)\bigr)^T\quad\mbox{and}\\
\tilde{\vv b}_0\eqdefa \bigl(&\p_{x_3}\psi_1\p_{x_2}\psi_3+\p_{x_2}\psi_1(1-\p_{x_3}\psi_3), (1-\p_{x_1}\psi_1)(1-\p_{x_3}\psi_3)-\p_{x_3}\psi_1\p_{x_1}\psi_3, \\
&\p_{x_2}\psi_1\p_{x_1}\psi_3+\p_{x_2}\psi_3(1-\p_{x_1}\psi_1)\bigr)^T,
\end{aligned}\eeq then $U_0\eqdefa (\bar{\vv b}_0,\tilde{\vv
b}_0,\vv b_0)$ satisfies  \eqref{B3}, and for $\vv e_1=(1,0,0)^T,
\vv e_2=(0,1,0)^T,$ \beq \label{B3p3} \|\bar{\vv b}_0-\vv
e_1\|_{B^{s-1}_{p,1}}+\|\tilde{\vv b}_0-\vv
e_2\|_{B^{s-1}_{p,1}}\leq C(K,\e_0)\|\bigl(b_0^1,b_0^2,
b_0^3-1\bigr)\|_{B^{s}_{p,1}}. \eeq }
\end{prop}

The proof of this proposition is postponed in Appendix
\ref{appenda}.

With $U_0$ obtained in Proposition \ref{LL1}, we shall first
investigate the global wellposedness to the following system with
sufficiently small $u_0:$
\begin{equation}\label{B4}
 \left\{\begin{array}{l}
\displaystyle \p_tU+\vv u\cdot\na U=\na\vv u U,
\qquad (t,x)\in\R^+\times\R^3, \\
\displaystyle \pa_t\vv u +\vv u\cdot\na\vv u -\D\vv u+\na p=-\f12\na|\vv b|^2+\vv b\cdot\na\vv b, \\
\displaystyle \dv\,\vv u =0\quad\mbox{and}\quad \quad\dv\, U=0,, \\
\displaystyle U|_{t=0}=U_0,\quad \vv u|_{t=0}=\vv u_0,
\end{array}\right.
\end{equation}
where the $3\times 3$ matrix $U=\bigl(\bar{\vv b}, \tilde{\vv b},
\vv b\bigr),$ and $\bar{\vv
b}=\bigl(\bar{b}^1,\bar{b}^2,\bar{b}^3\bigr)^T,$  $\tilde{\vv
b}=\bigl(\tilde{b}^1,\tilde{b}^2,\tilde{b}^3\bigr)^T.$ In
particular, for any smooth enough solution $(U,\vv u)$ of
\eqref{B4}, $(\vv b,\vv u)$ must be a smooth enough solution of
\eqref{B1}.

The main result concerning the wellposedness of the system
\eqref{B4} can be stated as follows:

\begin{thm}\label{th1}
{\sl Let $s_1>\f54,$ $s_2\in (-\f12,-\f14)$ and $p\in (1,2).$ Let
$\vv u_0\in \dH^{s_2}\cap\dB^{\f3p-1}_{p,1}(\R^3)$ with $\na\vv
u_0\in B^{s_1+\f3p-\f32}_{p,1}(\R^3)$ and
$U_0=\bigl(I-\na_x{\vv\Psi}\bigr)^{-1}$ with
$\vv\Psi=(\psi_1,\psi_2,\psi_3)^T$ satisfying
$\na\vv\Psi\in\dB^{s_2+\f3p-\f32}_{p,2}\cap
B^{s_1+\f3p-\f12}_{p,1}(\R^3)$ and $\det(I-\na_x{\vv\Psi})=1.$ We
assume that \beq\label{B2}
\|\na\vv\Psi\|_{\dB^{s_2+\f3p-\f32}_{p,2}\cap
B^{s_1+\f3p-\f12}_{p,1}}+\|\vv
u_0\|_{\dH^{s_2}\cap\dB^{\f3p-1}_{p,1}}+\|\na\vv
u_0\|_{B^{s_1+\f3p-\f32}_{p,1}}\leq\e_0\eeq for some $\e_0$
sufficiently small. Then
  \eqref{B4} has a unique global solution
$(U, \vv u, p)$  (up to a constant for $p$), with $U=\bigl(\bar{\vv
b}, \tilde{\vv b}, \vv b\bigr),$ so that \beq\label{th1wq}
\begin{split}
&\bar{\vv b}-\vv e_1,\, \tilde{\vv b}-\vv e_2\in  C([0,\infty); \dH^{s_1+1}\cap
\dH^{s_2+1}(\R^3)),\\
& \vv b-\vv e_3\in C([0,\infty); \dH^{s_1+1}\cap
\dH^{s_2}(\R^3))\cap
L^2(\R^+;\dH^{s_1+1}\cap \dH^{s_2+1}(\R^3)),\\
&\vv u\in  C([0,\infty); \dH^{s_1+1}\cap \dH^{s_2}(\R^3))\cap
L^2(\R^+;\dH^{s_1+2}\cap \dH^{s_2+1}(\R^3))\\
&\qquad\cap
L^1(\R^+; \dH^{s_1+2}\cap\dB^{\f52}_{2,1}(\R^3)),\\
&\na p\in  L^2(\R^+;\dH^{s_1}\cap \dH^{s_2}(\R^3)).
\end{split}
\eeq
 Furthermore, there holds
 \beq\label{th1wr}
\begin{split}
&\|(\bar{\vv b}-\vv e_1,\tilde{\vv b}-\vv e_2)\|_{L^\infty(\R^+;\dH^{s_1+1}\cap\dH^{s_2+1})}+\|\vv b-\vv e_3\|_{L^\infty(\R^+;\dH^{s_1+1}\cap\dH^{s_2})}
\\
&\quad+\|\vv u\|_{L^\infty(\R^+; \dH^{s_1+1}\cap\dH^{s_2})}+\|\vv
b-\vv e_3\|_{L^2(\R^+;\dH^{s_1+1}\cap\dH^{s_2+1})}
\\
&\quad+\| \vv
u\|_{L^2(\R^+;\dH^{s_1+2}\cap\dH^{s_2+1})}+\|\vv u\|_{L^1(\R^+; \dH^{s_1+2}\cap\dB^{\f52}_{2,1})}+\|\na p\|_{L^2(\R^+;\dH^{s_1}\cap\dH^{s_2})}\\
&\leq C\bigl(\|\na\vv\Psi\|_{\dB^{s_2+\f3p-\f32}_{p,2}\cap
B_{p,1}^{s_1+\f3p-\f12}}+\|\vv
u_0\|_{\dH^{s_2}\cap\dB^{\f3p-1}_{p,1}}+\|\na\vv
u_0\|_{B^{s_1+\f3p-\f32}_{p,1}}\bigr).
\end{split}
\eeq }
\end{thm}

In order to avoid the difficulty of propagating anisotropic
regularity for the transport equation in the system \eqref{B4}, we
shall  reformulate \eqref{B4} in the Lagrangian coordinates. Toward
this, we need first to find a volume preserving diffeomorphism
$X_0(y)$ on $\R^3$ so that there holds \eqref{app0}.

\begin{lem}\label{LL2} {\sl Let $p\in (1,2),s>1+\f3p.$ Let
$\vv\Psi=(\psi_1,\psi_2,\psi_3)^T$ satisfy $\na\vv\Psi\in
B^{s-1}_{p,1}(\R^3)$, $det\bigl(I-\na\vv \Psi\bigr)=1$ and
$\|\na\vv\Psi\|_{B^{s-1}_{p,1}}\leq \e_0$ for some $\e_0$
sufficiently small.  Then for $U_0=\bigl(I-\na\vv\Psi\bigr)^{-1},$
there exists $Y_0(y)=(Y_0^1(y),Y_0^2(y),Y_0^3(y))^T$ so that
$X_0(y)=y+Y_0(y)$ satisfies \beq\label{BB2} U_0\circ
X_0(y)=\na_yX_0(y)=I+\na_yY_0(y)\quad \mbox{and}\quad \|\na_y
Y_0\|_{B^{s-1}_{p,1}}\leq C\|\na_x\vv\Psi\|_{B^{s-1}_{p,1}}. \eeq }
\end{lem}

\begin{proof} Let $Y=(Y^1, Y^2, Y^3)^T,$ we denote
\beq\label{app3} \vv F(y,Y)\eqdefa Y-\vv\Psi(y+Y), \eeq with $\vv
F(y,Y)=(F^1(y,Y), F^2(y,Y), F^3(y,Y))^T$.

It is easy to observe from the assumption:
$\mathrm{det}\,(I-\na_x\vv\Psi)=\det\, U_0=1,$ that \beno
\mathrm{det}\f{\p(F^1,F^2,F^3)}{\p(Y^1,Y^2,Y^3)}=\det\,(I-\na_x\vv\Psi)|_{x=y+Y}=1,
\eeno from which, $\|\na\vv \Psi\|_{L^\infty}\leq C\e_0$ for some
$\e_0$ sufficiently small,  and the classical
 implicit function
theorem, we deduce that around every point $y,$ the function
$\vv F(y,Y)=\vv 0$ determines a unique function
$Y_0(y)=(Y_0^1(y),Y_0^2(y),Y_0^3(y))^T$ so that \beno\vv F(y,
Y_0(y))=\vv 0,\eeno or equivalently \beq \label{app3ad}
Y_0(y)=\vv\Psi(y+Y_0(y)).
 \eeq
Then denoting by $X_0(y)=y+Y_0(y)$, we have \beq
\label{app3cd}\begin{aligned}
\p_{y_1}Y^j_0(y)=&\p_{x_1}\psi_j\circ X_0(y)(1+\p_{y_1}Y^1_0(y))+\p_{x_2}\psi_j\circ X_0(y)\p_{y_1}Y^2_0(y)\\
&+\p_{x_3}\psi_j\circ X_0(y)\p_{y_1}Y^3_0(y),\quad\mbox{for}\quad
j=1,2,3.
\end{aligned}\eeq
Due to the fact that $\det\,(I-\na_x\vv\Psi)=\mathrm{det}\,U_0=1,$
we conclude that $I-\na_x\vv\Psi$ equals the adjoint matrix of
$U_0\eqdefa \bigl(b_{ij}\bigr)_{i,j=1,2,3}$,  which along with
\eqref{app3cd} ensures that \beq \label{app3bd}\begin{aligned}
\p_{y_1}Y^1_0(y)=b_{11}\circ X_0(y)-1,\quad
\p_{y_1}Y^2_0(y)=b_{21}\circ X_0(y),\quad
\p_{y_1}Y^3_0(y)=b_{31}\circ X_0(y).
\end{aligned}\eeq
 Along the same line, one has \beno \p_{y_j}Y^i_0(y)=b_{ij}\circ
X_0(y)-\d_{ij}. \eeno which implies the first part of \eqref{BB2}.
This in particular leads to \beno
\na_x\bigl(X_0^{-1}(x)\bigr)=\bigl((\na_yX_0)\circ
X_0^{-1}(x)\bigr)^{-1}=U_0^{-1}(x)=I-\na\vv\Psi(x),\eeno from which,
\eqref{app3ad}, $\|\na\vv\Psi\|_{B^{s-1}_{p,1}}\leq \e_0$ for
$s>1+\f3p$ and Lemma \ref{funct}, we achieve the second part of
\eqref{BB2}.
\end{proof}

 With $X_0(y)=y+Y_0(y)$ obtained in  Lemma \ref{LL2}, we now
define the flow map $X(t,y)$ by \beno\left\{\begin{aligned}
&\f{dX(t,y)}{dt}=\vv u(t,X(t,y)),\\
&X(t,y)|_{t=0}=X_0(y),
\end{aligned}\right.\eeno
and $Y(t,y)$ through \beq\label{B6} X(t,y)=X_0(y)+\int_0^t\vv
u(s,X(s,y))\,ds\eqdefa y+Y(t,y). \eeq Then by virtue of Proposition
1.8 of \cite{Majda} and \eqref{BB2}, we deduce from \eqref{B4} that
\beq\label{B7} U(t, X(t,y))=\na_y X(t,y)=I+\na_y
Y(t,y)\quad\mbox{and}\quad\mathrm{det}\, \bigl(I+\na_y
Y(t,y)\bigr)=1. \eeq Denoting $ U(t,
X(t,y))\eqdefa(a_{ij})_{i,j=1,2,3}$ and
$\cA_Y\eqdefa(b_{ij})_{i,j=1,2,3}$ with
\beq\label{B8}\begin{aligned} &b_{11}=
(1+\p_2Y^2)(1+\p_3Y^3)-\p_3Y^2\p_2Y^3,
\quad b_{12}=\p_3Y^1\p_2Y^3-\p_2Y^1(1+\p_3Y^3),\\
&b_{13}=\p_2Y^1\p_3Y^2-\p_3Y^1(1+\p_2Y^2),\qquad\quad\ \
b_{21}=\p_3Y^2\p_1Y^3-\p_1Y^2(1+\p_3Y^3),\\
&b_{22}=(1+\p_1Y^1)(1+\p_3Y^3)-\p_3Y^1\p_1Y^3,\quad
b_{23}=\p_3Y^1\p_1Y^2-(1+\p_1Y^1)\p_3Y^2,\\
&b_{31}=\p_1Y^2\p_2Y^3-(1+\p_2Y^2)\p_1Y^3,\qquad\quad\ \
b_{32}=\p_2Y^1\p_1Y^3-(1+\p_1Y^1)\p_2Y^3,\\
&b_{33}=(1+\p_1Y^1)(1+\p_2Y^2)-\p_2Y^1\p_1Y^2.
\end{aligned}
\eeq It is easy to observe that $\sum_{i=1}^3\f{\p b_{ij}}{\p
y_i}=0$ (see also  Lemma 2.1 of \cite{XZZ}). Moreover, as
$\mathrm{det}\, U=1,$ $\cA_Y=(I+\na_y Y)^{-1}$.  Then  it follows
from \eqref{B7} that \beq\label{B9} \vv b\circ
X(t,y)=(\p_{y_3}Y^1,\p_{y_3}Y^2,1+\p_{y_3}Y^3)^T\quad\text{and}\quad\cA_Y(\vv
b\circ X)=(0,0,1)^T, \eeq from which, we infer \beq\label{B10}
\begin{split}
(\vv b\cdot\na_x\vv b)\circ X(t,y)&=[\dv_x(\vv b\otimes\vv b)]\circ X(t,y)\\
&=\na_y\cdot[\cA_Y(\vv b\circ X)\otimes(\vv b\circ X)]=\p_{y_3}(\vv
b\circ X)=\p_{y_3}^2Y(t,y).
\end{split} \eeq
Thanks to \eqref{B6} and \eqref{B10}, we can equivalently
reformulate \eqref{B4} as \beq\label{B11}\left\{\begin{aligned}
&Y_{tt}-\na_Y\cdot\na_Y Y_t-\p_{y_3}^2Y+\na_Yq=\vv 0,\\
&\na_Y\cdot Y_t=0,\\
&Y|_{t=0}=Y_0,\quad Y_t|_{t=0}=\vv u_0\circ X_0(y)\eqdefa Y_1,
\end{aligned}\right.\eeq
where $q(t,y)\eqdefa (p+\f12|\vv b|^2)\circ X(t,y)$ and
$\na_Y\eqdefa \mathcal{A}_Y^T\na_y$  with
$\cA_Y=(b_{ij})_{i,j=1,2,3}$ being determined by \eqref{B8}.
 Here and in
what follows, we always assume that $\|\na_y
Y\|_{L^\infty}\leq\f{1}{2}$. Under this assumption, we  rewrite
\eqref{B11} as \beq\label{B12}\left\{\begin{aligned}
&Y_{tt}-\Delta_y Y_t-\p_{y_3}^2 Y=\vv f(Y,q),\\
&\na_y\cdot Y=\r(Y),\\
&Y|_{t=0}=Y_0,\quad Y_t|_{t=0}=Y_1.
\end{aligned}\right.\eeq
where \beq\label{B13}\begin{aligned}
\vv f(Y,q)=&(\na_Y\cdot\na_Y-\Delta_y)Y_t-\na_Yq,\\
\r(Y)=&\na_y\cdot Y_0-\int_0^t(\na_Y-\na_y)\cdot Y_s
ds\\
=&-\sum_{i<j}\p_iY^i\p_jY^j+\sum_{i<j}\p_iY^j\p_jY^i-\p_1Y^1\p_2Y^2\p_3Y^3-\p_3Y^1\p_1Y^2\p_2Y^3\\
&-\p_2Y^1\p_3Y^2\p_1Y^3+\p_1Y^1\p_3Y^2\p_2Y^3+\p_3Y^1\p_2Y^2\p_1Y^3+\p_2Y^1\p_1Y^2\p_3Y^3.
\end{aligned}\eeq
Here we used \eqref{B8} and $\det\,(I+\na Y_0)=1$ to derive the
second equality of \eqref{B13}. Indeed thanks to \eqref{B8}, one has
\beq\label{B13ag}
\begin{split}
(\na_Y-\na_y)\cdot
Y_t=&\f{d}{dt}\Bigl(\sum_{i<j}\p_iY^i\p_jY^j-\sum_{i<j}\p_iY^j\p_jY^i+\p_1Y^1\p_2Y^2\p_3Y^3\\
&+\p_3Y^1\p_1Y^2\p_2Y^3+\p_2Y^1\p_3Y^2\p_1Y^3-\p_1Y^1\p_3Y^2\p_2Y^3\\
&-\p_3Y^1\p_2Y^2\p_1Y^3-\p_2Y^1\p_1Y^2\p_3Y^3\Bigr)\\
=&\f{d}{dt}\bigl(\mathrm{det}\,\bigl(I+\na_y Y\bigr)-1-\na_y\cdot
Y\bigr),
\end{split} \eeq which together with $\det\,(I+\na_y Y_0)=1$ ensures the second equality of \eqref{B13}. Moreover, the
equation $\na_y\cdot Y=\r(Y)$ implies that $\mathrm{det}\,(I+\na_y
Y)=1$ and $\na_Y\cdot Y_t=0$.

 For notational
convenience, we shall neglect the subscripts $x$ or $y$ in $\p, \na$
and $\D$ in the sequel. We make the convention that whenever $\na$
acts on $(U,\vv u, p),$  we understand $(\na U, \na\vv u, \na p)$ as
$(\na_x U, \na_x\vv u, \na_x p).$ While $\na$ acts on $(Y,q),$ we
understand $(\na Y, \na q)$ as $(\na_y Y, \na_y q).$ Similar
conventions for $\p$ and $\D.$

For  \eqref{B12}-\eqref{B13}, we have  the following global wellposedness
result:

\begin{thm}\label{T} {\sl Let $s_1>\f{5}{4}$, $s_2\in(-\f{1}{2},-\f{1}{4})$. Let $(Y_0, Y_1)$ satisfy $(\p_3Y_0,\,\D Y_0)
\in\dot{H}^{s_1}\cap\dot{H}^{s_2}\cap\cB^{\f{1}{2},0}\cap\cB^{s_1,0}(\R^3)$,
$Y_1\in\dot{H}^{s_1+1}\cap\dot{H}^{s_2}\cap\cB^{\f{1}{2},0}\cap\cB^{s_1,0}(\R^3)$
and \beq\label{B14}
 \mathrm{det}\,(I+\na Y_0)=1,\quad\na_{Y_0}\cdot Y_1=0, \quad\text{and}
\eeq
\beq\label{B15}\begin{aligned}
&\|
Y_0\|_{\dot{H}^{s_1+2}\cap\dot{H}^{s_2+2}}+\|\p_3Y_0\|_{\dot{H}^{s_2}}+\|Y_1\|_{\dot{H}^{s_1+1}\cap\dot{H}^{s_2}}\\
&\quad+\|
Y_0\|_{\cB^{\f{5}{2},0}\cap\cB^{s_1+2,0}}+\|\p_3Y_0\|_{\cB^{\f{1}{2},0}\cap\cB^{s_1,0}}+\|Y_1\|_{\cB^{\f{1}{2},0}\cap\cB^{s_1,0}}\leq
\e_0
\end{aligned}\eeq for some $\e_0$ sufficiently small.
 Then  \eqref{B12}-\eqref{B13} has a unique global solution $(Y,q)$ (up to a constant for $q$)
 so that
 \beq\label{B16}\begin{split}
 & Y\in C([0,\infty);\dot{H}^{s_1+2}\cap\dot{H}^{s_2+2}\cap\cB^{\f{5}{2},0}\cap\cB^{s_1+2,0}(\R^3))\quad\mbox{and}\\
 &\p_3Y\in C([0,\infty);\dot{H}^{s_2}\cap\cB^{\f{1}{2},0}\cap\cB^{s_1,0}(\R^3))\cap L^2(\R^+; \dot{H}^{s_1+1}
 \cap \dot{H}^{s_2+1}(\R^3)),\\
&Y_t\in C([0,\infty);\dot{H}^{s_1+1}\cap\dot{H}^{s_2}\cap\cB^{\f{1}{2},0}\cap\cB^{s_1,0}(\R^3))\cap
L^2(\R^+;\dot{H}^{s_1+2}\cap\dot{H}^{s_2+1}(\R^3))\\
&\qquad\quad\cap L^1(\R^+;
\cB^{\f{5}{2},0}\cap\cB^{s_1+2,0}(\R^3)),\\
& \na q\in L^2(\R^+;\dot{H}^{s_1}\cap\dot{H}^{s_2}(\R^3))\cap
L^1(\R^+;\dot{H}^{s_1}\cap\dot{H}^{s_2}(\R^3)). \end{split} \eeq
Moreover, there hold $\mathrm{det}\,(I+\na Y)=1,\, \na_Y\cdot
Y_t=0$, and \beq\label{B17}\begin{aligned}
&\|Y\|_{L^\infty_T(\dot{H}^{s_1+2}\cap\dot{H}^{s_2+2})}^2
+\|\p_3Y\|_{L^\infty_T(\dot{H}^{s_2})}^2+\|Y_t\|_{L^\infty_T(\dot{H}^{s_1+1}\cap\dot{H}^{s_2})}^2
+\|\p_3Y\|_{L^2_T(\dot{H}^{s_1+1}\cap\dot{H}^{s_2+1})}^2\\
 &\quad+\|Y_t\|_{L^2_T(\dot{H}^{s_1+2}\cap\dot{H}^{s_2+1})}^2
 +\|\na q\|_{L^2_T(\dot{H}^{s_1}\cap\dot{H}^{s_2})}^2+\|\na q\|_{L^1_T(\dot{H}^{s_1}\cap\dot{H}^{s_2})}^2\\
&\quad
+\|Y\|_{L^\infty_T(\cB^{\f{5}{2},0}\cap\cB^{s_1+2,0})}^2+\|\p_3Y\|_{L^\infty_T(\cB^{\f{1}{2},0}\cap\cB^{s_1,0})}^2
+\|Y_t\|_{L^1_T(\cB^{\f{5}{2},0}\cap\cB^{s_1+2,0})}^2\\
&\leq C\bigl(\|\p_3Y_0\|_{\dot{H}^{s_2}}^2+\|
Y_0\|_{\dot{H}^{s_1+2}\cap\dot{H}^{s_2+2}}^2+\|Y_1\|_{\dot{H}^{s_1+1}\cap\dot{H}^{s_2}}^2\\
&\quad+\|\p_3Y_0\|_{\cB^{\f{1}{2},0}\cap\cB^{s_1,0}}^2+\|
Y_0\|_{\cB^{\f{5}{2},0}\cap\cB^{s_1+2,0}}^2+\|Y_1\|_{\cB^{\f{1}{2},0}\cap\cB^{s_1,0}}^2\bigr).
\end{aligned}\eeq}
\end{thm}

\begin{rmk} The norm of $\|\cdot\|_{\cB^{s,0}}$ is  given by
Definition \ref{def2}.
 We should mention once again that the equation $\na\cdot Y=\r(Y)$
in \eqref{B12}  plays a key role in the proof of Theorem \ref{T}. In
particular, we need to use this equation to derive the globally
$L^1$ in time  estimates of $\na q$ and $\na Y_t$, which will be
crucial for us
 to close the energy estimates for \eqref{B12}-\eqref{B13}.
\end{rmk}

\no{\bf Scheme of the proof and organization of the paper.}\\

To avoid the difficulty caused by propagating  anisotropic
regularity for the  transport equation in \eqref{B4}, we shall first
prove the global wellposedness of the Lagrangian formulation
\eqref{B12}-\eqref{B13} with small initial data.

Let $(Y,q)$ be a  smooth enough solution of \eqref{B12}, applying
standard energy estimate to \eqref{B12} leads to \beq \label{B18}
\begin{split}
&\f{d}{dt}\Bigl\{\f{1}{2}\bigl(\|Y_t\|_{\dot{H}^s}^2+\|Y_t\|_{\dot{H}^{s+1}}^2+\|\p_3Y\|_{\dot{H}^s}^2
+\|\p_3Y\|_{\dot{H}^{s+1}}^2+\f{1}{4}\|Y\|_{\dot{H}^{s+2}}^2\bigr)\\
&\quad
-\f{1}{4}(Y_t\ |\ \Delta Y)_{\dot{H}^s}\Bigr\}+\f{3}{4}\|Y_t\|_{\dot{H}^{s+1}}^2+\|Y_t\|_{\dot{H}^{s+2}}^2+\f{1}{4}\|\p_3Y\|_{\dot{H}^{s+1}}^2\\
&=\bigl(\vv f\ |\ Y_t-\f{1}{4}\Delta Y-\D Y_t\bigr)_{\dot{H}^s}.
\end{split}
\eeq where $(a\ |\ b)_{\dot{H}^s}$ denotes the standard $\dot{H}^s$
inner product of $a$ and $b.$ \eqref{B18} shows that  $\p_3 Y$
belongs to  $L^2(\R^+;\dot{H}^{s+1}(\R^3)),$ however, there is no
time dissipative estimate of $\D Y.$ Therefore, in order to close
the energy estimate in \eqref{B18}, we would require the source term
$\vv f$ in \eqref{B12} belonging to $L^1(\R^+;\dH^s(\R^3)).$ To
achieve this, we  need also the
$L^1(\R^+;\cB^{\f{5}{2},0}\cap\cB^{s_1+2,0}(\R^3))$ estimate of
$Y_t.$  Toward this, we shall use the dissipative estimates for
$\p_3 Y$ as well as  the fact that $\na\cdot Y=\r(Y)$ in a rather
crucial way.

In the first part of Section 3, we shall present a heuristic
analysis to the linearized system of \eqref{B12}-\eqref{B13}, which
motivates us to use anisotropic Littlewood-Paley theory below,  then
we shall collect some basic facts on functional framework and
Littlewood-Paley analysis in Subsection 3.2. \par

In Section 4, we apply anisotropic Littlewood-Paley theory to
explore the dissipative mechanism for  a linearized model of
\eqref{B12}-\eqref{B13}.\par

 In Section 5, we present the proof of
Theorem \ref{T},  and we present the proof of  Theorems \ref{th1}
and \ref{th2} in Section 6.\par

 Finally, we present the proofs of
some technical lemmas in the Appendices.

 \setcounter{equation}{0}
\section{Preliminary}\label{sect3}
\subsection{Spectral analysis to the linearized system of \eqref{B12}-\eqref{B13}} \label{sect3.1}
 We first investigate heuristically  the spectrum properties to
the following linearized system of \eqref{B12}-\eqref{B13}:
 \beq\label{B19}\left\{\begin{aligned}
&Y_{tt}-\Delta Y_t-\p_3^2 Y=\vv f,\\
&Y|_{t=0}=Y_0,\quad Y_t|_{t=0}=Y_1.
\end{aligned}\right.\eeq
Note that the symbolic equation corresponds to \eqref{B19} reads
\beno \la^2+|\xi|^2\la+\xi_3^2=0\quad\mbox{for}\quad
\xi=(\xi_h,\xi_3)\quad\mbox{and}\quad \xi_h=(\xi_1,\xi_2). \eeno It
is easy to calculate that this equation has two different
eigenvalues \beq\label{C1} \la_\pm =-\f{|\xi|^2\pm
\sqrt{|\xi|^4-4\xi_3^2}}{2}. \eeq The Fourier modes correspond to
$\la_+$ decays like $e^{-t|\xi|^2}$.  Whereas the decay property of
the Fourier modes corresponding to $\la_-$  varies with directions
of $\xi$ as \beq\label{C2}
\la_-(\xi)=-\f{2\xi_3^2}{|\xi|^2\bigl(1+\sqrt{1-\f{4\xi_3^2}{|\xi|^4}}\bigr)}
\to -1\quad \mbox{as}\quad |\xi|\to \infty \eeq only in the $\xi_3$
direction. This shows that smooth solution of \eqref{B19} decays in
a very subtle way. In order to capture this delicate decay property
for the linear equation \eqref{B19}, we shall  decompose our
frequency space into two parts: $\bigl\{ \xi=(\xi_h,\xi_3):\
|\xi|^2\leq 2|\xi_3|\ \bigr\}$ and $\bigl\{ \xi=(\xi_h,\xi_3):\
|\xi|^2> 2|\xi_3|\ \bigr\}$.

 This
heuristic  analysis shows that the dissipative properties of the
solutions to \eqref{B19} may be more complicated than that for the
linearized system of isentropic compressible Navier-Stokes system in
\cite{Da}, and this brief analysis also suggests us to employ the
tool of anisotropic Littlewood-Paley theory as in \cite{XLZMHD1} for
2-D incompressible MHD system and \cite{LZ} for a modified 3-D MHD
system, which has also been used in the study of the global
wellposedness to 3-D anisotropic incompressible Navier-Stokes
equations \cite{CDGG,CPZ1,CZ,GZ2, DI, Pa02,PZ1,Zhangt2}. One may
check Section \ref{sect4} below for the detailed rigorous analysis
corresponding to this scenario.

\subsection{Littlewood-Paley theory} \label{subsect3.2} The
proof  of Theorem \ref{T} requires a dyadic decomposition of the
Fourier variables, or the Littlewood-Paley decomposition. Let us
briefly explain how it may be built in the case $x\in\R^3$ (see e.g.
\cite{bcd}).  Let $\varphi$ and $\chi$ be smooth functions supported
in $\mathcal{C}\eqdefa \{ \tau\in\R^+,\
\frac{3}{4}\leq\tau\leq\frac{8}{3}\}$ and $\cB\eqdefa \{
\tau\in\R^+,\ \tau\leq\frac{4}{3}\}$ such that
\begin{equation*}
 \sum_{j\in\Z}\varphi(2^{-j}\tau)=1 \quad\hbox{for}\quad \tau>0\quad\mbox{and}\quad  \chi(\tau)+ \sum_{j\geq
0}\varphi(2^{-j}\tau)=1\quad\hbox{for}\quad \tau\geq 0.
\end{equation*}
For $a\in{\mathcal S}'(\R^2),$ we set \beq
\begin{split}
&\Delta_k^ha\eqdefa\cF^{-1}(\varphi(2^{-k}|\xi_h|)\widehat{a}),\qquad
S^h_ka\eqdefa\cF^{-1}(\chi(2^{-k}|\xi_h|)\widehat{a}),
\\
& \Delta_\ell^va
\eqdefa\cF^{-1}(\varphi(2^{-\ell}|\xi_3|)\widehat{a}),\qquad \
S^v_\ell a \eqdefa \cF^{-1}(\chi(2^{-\ell}|\xi_3|)\widehat{a}),
 \quad\mbox{and}\\
&\Delta_ja\eqdefa\cF^{-1}(\varphi(2^{-j}|\xi|)\widehat{a}),
 \qquad\ \ \
S_ja\eqdefa \cF^{-1}(\chi(2^{-j}|\xi|)\widehat{a}), \end{split}
\label{1.0}\eeq where $\cF a$ and $\widehat{a}$ denote the Fourier
transform of the distribution  $a.$  The dyadic operators satisfy
the property of almost orthogonality:
\begin{equation}\label{C4}
\Delta_k\Delta_j a\equiv 0 \quad\mbox{if}\quad| k-j|\geq 2
\quad\mbox{and}\quad \Delta_k( S_{j-1}a \Delta_j b) \equiv
0\quad\mbox{if}\quad| k-j|\geq 5.
\end{equation}
Similar properties hold for $\D_k^h$ and $\D_\ell^v.$

\begin{defi}\label{def1}[Definition 2.15 of \cite{bcd}]
{\sl   Let $(p,r)\in[1,+\infty]^2,$ $s\in\R$ and $u\in{\mathcal
S}_h'(\R^3),$  (see Definition 1.26 of \cite{bcd}), which means
$u\in\cS'(\R^d)$ and
$\lim_{j\to-\infty}\|\chi(2^{-j}D)u\|_{L^\infty}$ $=0,$  we set
$$
\|u\|_{\dB^s_{p,r}}\eqdefa\Big(2^{js}\|\Delta_j
u\|_{L^{p}}\Big)_{\ell ^{r}}.
$$
\begin{itemize}

\item
For $s<\frac{3}{p}$ (or $s=\frac{3}{p}$ if $r=1$), we define $
\dB^s_{p,r}(\R^3)\eqdefa \big\{u\in{\mathcal S}_h'(\R^3)\;\big|\; \|
u\|_{\dB^s_{p,r}}<\infty\big\}.$

\item
If $k\in\N$ and $\frac{3}{p}+k-1\leq s<\frac{3}{p}+k$ (or
$s=\frac{3}{p}+k$ if $r=1$), then $ \dB^s_{p,r}(\R^3)$ is defined as
the subset of distributions $u\in{\mathcal S}_h'(\R^3)$ such that
$\partial^\beta u\in \dB^{s-k}_{p,r}(\R^3)$ whenever $|\beta|=k.$
\end{itemize}

Inhomogenous Besov spaces $B^s_{p,r}(\R^3)$ can be defined similarly
(see Definition 2.68 of \cite{bcd}). For simplicity, we shall
abbreviate $\dB^s_{2,1}(\R^3)$ (resp. $B^s_{2,1}(\R^3)$) as
$\dB^s(\R^3)$ (resp. $B^s(\R^3)$) in all that follows.}
\end{defi}

\begin{rmk}\label{rmk1}
(1) It is easy to observe that
$\dot{B}_{2,2}^s(\R^3)=\dot{H}^s(\R^3)$.

(2) Let  $(p,r)\in[1,+\infty]^2$, $s\in\R$ and $u\in{\mathcal
S}'(\R^3)$. Then $u\in\dot{B}_{p,r}^s(\R^3)$  if and only if there
exists $\{c_{j,r}\}_{j\in\Z}$ such that $\|c_{j,r}\|_{\ell^r}=1$ and
\beno \|\D_ju\|_{L^p}\leq Cc_{j,r}2^{-js}\|u\|_{\dot{B}^s_{p,r}}
\quad\text{for all}\quad j\in\Z. \eeno

(3) Let $s, s_1, s_2\in\R$ with $s_1<s<s_2$ and
$u\in\dot{H}^{s_1}\cap\dot{H}^{s_2}(\R^3)$. Then $u\in
\dot{B}^s(\R^3)$, and  there holds \beq\label{C5}
\|u\|_{\dot{B}^s}\lesssim
\|u\|_{\dot{H}^{s_1}}^{\f{s_2-s}{s_2-s_1}}\|u\|_{\dot{H}^{s_2}}^{\f{s-s_1}{s_2-s_1}}
\lesssim\|u\|_{\dot{H}^{s_1}}+\|u\|_{\dot{H}^{s_2}}. \eeq
\end{rmk}

 For the convenience of the readers, we  recall the
following Bernstein type lemma from \cite{bcd, CZ, Pa02}:

\begin{lem}\label{L2} {\sl Let $\cB_{h}$ (resp.~$\cB_{v}$) be a ball
of~$\R^2$ (resp.~$\R$), and~$\cC_{h}$ (resp.~$\cC_{v}$) a ring
of~$\R^2$ (resp.~$\R$); let~$1\leq p_2\leq p_1\leq \infty$ and
~$1\leq q_2\leq q_1\leq \infty.$ Then there holds:
\smallbreak\noindent If the support of~$\wh a$ is included
in~$2^k\cB_{h}$, then
\[
\|\partial_h^\alpha a\|_{L^{p_1}_h(L^{q_1}_v)} \lesssim
2^{k\left(|\al|+2\left(\frac1{p_2}-\frac1{p_1}\right)\right)}
\|a\|_{L^{p_2}_h(L^{q_1}_v)},\quad\text{for}\quad \p_h=(\p_1,\p_2).
\]
If the support of~$\wh a$ is included in~$2^\ell\cB_{v}$, then
\[
\|\partial_3^\beta a\|_{L^{p_1}_h(L^{q_1}_v)} \lesssim
2^{\ell\left(\beta+\left(\frac1{q_2}-\frac1{q_1}\right)\right)} \|
a\|_{L^{p_1}_h(L^{q_2}_v)}.
\]
If the support of~$\wh a$ is included in~$2^k\cC_{h}$, then
\[
\|a\|_{L^{p_1}_h(L^{q_1}_v)} \lesssim 2^{-kN} \|\partial_h^N
a\|_{L^{p_1}_h(L^{q_1}_v)}.
\]
If the support of~$\wh a$ is included in~$2^\ell\cC_{v}$, then
\[
\|a\|_{L^{p_1}_h(L^{q_1}_v)} \lesssim 2^{-\ell N} \|\partial_3^N
a\|_{L^{p_1}_h(L^{q_1}_v)}.
\]}
\end{lem}

In order to obtain the $L^1(\R^+; \mbox{Lip}(\R^3))$ estimate of
$Y_t$ for the linearized equation \eqref{B19}, we  recall the
following anisotropic Besov type space from \cite{LZ, XLZMHD1}:

\begin{defi}\label{def2}
{\sl  Let  $s_1,s_2\in\R$ and $u\in{\mathcal S}_h'(\R^3),$ we define
the norm
$$
\|u\|_{\cB^{s_1,s_2}}\eqdefa\sum_{j,k\in\Z^2}2^{js_1}2^{ks_2}\|\Delta_j\D_k^v
u\|_{L^{2}}.
$$
}
\end{defi}

Then we have the following three dimensional version of Lemma 3.2 in
\cite{XLZMHD1}:

\begin{lem}\label{L1}
Let $s_1,s_2,\tau_1,\tau_2\in\R,$  which satisfy
$s_1<\tau_1+\tau_2<s_2$ and $\tau_2>0.$ Let
$a\in\dot{H}^{s_1}\cap\dot{H}^{s_2}(\R^3).$ Then
$a\in\cB^{\tau_1,\tau_2}(\R^3),$ and there holds \beno
\|a\|_{\cB^{\tau_1,\tau_2}}\lesssim\|a\|_{\dot{B}^{\tau_1+\tau_2}}
\lesssim\|a\|_{\dot{H}^{s_1}}+\|a\|_{\dot{H}^{s_2}}. \eeno
\end{lem}

\begin{proof}
By virtue of Definition \ref{def2} and the fact: $j\geq k-N_0$ for
some fixed positive integer $N_0$ in dyadic operator $\D_j\D^v_k$,
we infer
\begin{multline*}
\|a\|_{\cB^{\tau_1,\tau_2}}=\sum_{\substack{j,k\in\Z^2\\k\leq
j+N_0}}2^{j\tau_1}2^{k\tau_2}\|\D_j\D_k^va\|_{L^2}\lesssim
\sum_{j\in\Z}2^{j\tau_1}\|\D_ja\|_{L^2}\sum_{k\leq
j+N_0}2^{k\tau_2}\\
\lesssim\sum_{j\in\Z}2^{j(\tau_1+\tau_2)}\|\D_ja\|_{L^2}\lesssim
\|a\|_{\dot{B}^{\tau_1+\tau_2}},
\end{multline*}
 which together with \eqref{C5} completes the proof of the lemma.
\end{proof}

In  order to obtain a better description of the regularizing effect
for the transport-diffusion equation, we will use Chemin-Lerner type
spaces $\widetilde{L}^{q}_T(B^s_{p,r}(\R^3))$ (see \cite{bcd} for
instance).

\begin{defi}\label{def3}
Let  $(r,q,p)\in[1,+\infty]^3$ and $T\in(0,+\infty]$.
 We define the $\wt{L}^q_T(\dot{B}^s_{p,r}(\R^3))$  and  $\wt{L}^q_T(\cB^{s_1,s_2}(\R^3))$ by
\beno
&&\|u\|_{\wt{L}^q_T(\dot{B}^s_{p,r})}\eqdefa\Bigl(\sum_{j\in\Z}2^{jrs}
\|\D_ju\|_{L^q_T(L^p)}^r\Bigr)^{\f{1}{r}},\quad
\|u\|_{\wt{L}^q_T(\cB^{s_1,s_2})}\eqdefa\sum_{j,k\in\Z^2}2^{js_1}2^{ks_2}
\|\D_j\D_k^vu\|_{L^q_T(L^2)}, \eeno with the usual change if
$r=\infty$.
\end{defi}

\begin{rmk}\label{rmk2}
The proof of Lemma \ref{L1} ensures that \beq\label{C6}
\|u\|_{\wt{L}^2_T(\cB^{\tau_1,\tau_2})}\lesssim\|u\|_{\wt{L}^2_T(\dot{B}^{\tau_1+\tau_2})}
\lesssim\|u\|_{L^2_T(\dot{H}^{s_1})}+\|u\|_{L^2(\dot{H}^{s_2})},
\eeq for $\tau_1,\tau_2$ and $s_1,s_2$  given by Lemma \ref{L1}.
\end{rmk}

We also recall the isotropic para-differential decomposition of Bony
from \cite{Bo}: let $a, b\in \cS'(\R^3),$  \beq
\label{C7}\begin{split} &ab=T(a,b)+\cR(a,b), \quad\mbox{or}\quad
ab=T(a,b)+\bar{T}(a,b)+ R(a,b), \quad\hbox{where}\\
& T(a,b)\eqdefa\sum_{j\in\Z}S_{j-1}a\Delta_jb, \quad
\bar{T}(a,b)\eqdefa T(b,a),\quad
 \cR(a,b)\eqdefa\sum_{j\in\Z}\Delta_jaS_{j+2}b, \andf\\
&R(a,b)\eqdefa\sum_{j\in\Z}\Delta_ja\tilde{\Delta}_{j}b,\quad\hbox{with}\quad
\tilde{\Delta}_{j}b\eqdefa\sum_{\ell=j-1}^{j+1}\D_\ell b.
\end{split} \eeq
 Considering the special
structure of the functions in $\cB^{s_1,s_2}(\R^3),$ we sometime use
both isentropic Bony's decomposition\eqref{C7}  and  \eqref{C7} for
the vertical variable $x_3$ simultaneously.

As an application of the above basic facts on Littlewood-Paley
theory, we present the following product laws in space
$\cB^{s_1,s_2}(\R^3)$.

\begin{lem}\label{L3}
{\sl Let $s_1,s_2,\tau_1,\tau_2\in\R,$ which satisfy  $s_1,
s_2\leq1$, $\tau_1,\tau_2\leq\f{1}{2}$ and $s_1+s_2>0$,
$\tau_1+\tau_2>0$. Then for $a\in\cB^{s_1,\tau_1}(\R^3)$ and
$b\in\cB^{s_2,\tau_2}(\R^3)$,
$ab\in\cB^{s_1+s_2-1,\tau_1+\tau_2-\f{1}{2}}(\R^3)$ and there
holds \beq\label{C9}
\|ab\|_{\cB^{s_1+s_2-1,\tau_1+\tau_2-\f{1}{2}}}\lesssim\|a\|_{\cB^{s_1,\tau_1}}\|b\|_{\cB^{s_2,\tau_2}}.
\eeq}
\end{lem}
\begin{proof} The proof of this lemma is identical to that of Lemma
3.3 in \cite{XLZMHD1}, we omit the details here.
\end{proof}

\begin{lem}\label{L3ad}
{\sl Let $\d\in [0,\f12),$ $s_1\leq\f32-\d,$ $s_2\leq 1+\d$ and
$s_1+s_2>\f12.$ Then one has \beno
\|ab\|_{\cB^{s_1+s_2-\f32,0}}\lesssim
\|a\|_{\dB^{s_1}}\|b\|_{\cB^{s_2-\d,\d}}.\eeno}
\end{lem}

\begin{proof} By virtue of Lemma \ref{L1} and Lemma \ref{L3}, we
have \beno \|ab\|_{\cB^{s_1+s_2-\f32,0}}\lesssim
\|a\|_{\cB^{s_1-\f12+\d,\f12-\d}}\|b\|_{\cB^{s_2-\d,\d}} \lesssim
\|a\|_{\dB^{s_1}}\|b\|_{\cB^{s_2-\d,\d}}. \eeno This completes the
proof of the lemma.
\end{proof}

\begin{rmk} It follows from Lemma \ref{L3} and Lemma \ref{L3ad} that
\beq\label{C9ag}
\begin{split}
&
\|ab\|_{\cB^{s,0}}\lesssim\|a\|_{\dot{B}^{\f{3}{2}}}\|b\|_{\cB^{s,0}},\quad\mbox{and}\quad
\|ab\|_{\cB^{s,0}}\lesssim\|a\|_{\dot{B}^{\f{5}{4}}}\|b\|_{\dot{B}^{s+\f{1}{4}}}\quad\mbox{for}\
\  -1<s\leq\,1,\\ &
\|ab\|_{\cB^{s,0}}\lesssim\|a\|_{\dot{B}^{\f{1}{2}}}\|b\|_{\dot{B}^{s+1}}\quad\mbox{for}\
\ -1<s\leq \,1. \end{split} \eeq
\end{rmk}

\begin{lem}\label{L7} {\sl  For any $s>-1$, there holds
\beq\label{C9jh} \begin{split} & \|ab\|_{\cB^{s,0}}\lesssim
\|a\|_{\dB^{\f32}}\|b\|_{\cB^{s,0}}+\|b\|_{\dB^{\f32}}\|a\|_{\cB^{s,0}},\\
& \|ab\|_{\cB^{s,0}}\lesssim
\|a\|_{\dB^{\f32}}\|b\|_{\cB^{s,0}}+\|b\|_{\cB^{1,0}}\|a\|_{\cB^{s,\f12}},
\end{split}
\eeq and \beq \label{C9bn} \|ab\|_{\cB^{s,0}}\lesssim
\|a\|_{\dB^{\f32-\d_1}}\|b\|_{\dB^{s+\d_1}}+\|b\|_{\dB^{\f32-\d_2}}\|a\|_{\dB^{s+\d_2}},
\eeq for $\d_1,\d_2\in (0,\f12).$}
\end{lem}

\begin{proof}

We first get, by using Bony's decomposition \eqref{C7} and
\eqref{C7} for the vertical variable, that
\beq\label{C10}\begin{aligned}
ab=&\bigl(TT^v+T\bar{T}^v+TR^v+\bar{T}T^v+\bar{T}\bar{T}^v+\bar{T}R^v
+RT^v+R\bar{T}^v+RR^v\bigr)(a,b).
\end{aligned}\eeq
We shall present the detailed estimates to typical terms above.
Indeed applying Lemma \ref{L2} gives \beno
\begin{split}
\|\D_j\D_k^v(TR^v(a,b))\|_{L^2}&\lesssim
2^{\f{k}2}\sum_{{|j'-j|\leq4}\atop{k'\geq
k-N_0}}\|S_{j'-1}\D_{k'}^va\|_{L^\infty_h(L^2_v)}
\|\D_{j'}\wt{\D}_{k'}^vb\|_{L^2}\\
&\lesssim2^{\f{k}2}\sum_{{|j'-j|\leq4}\atop{k'\geq k-N_0}}d_{j',k'}
2^{-j's}2^{-\f{k'}2}\|a\|_{\cB^{1,\f12}} \|b\|_{\cB^{s,0}}\lesssim
d_{j,k}2^{-js} \|a\|_{\dB^{\f32}}\|b\|_{\cB^{s,0}},
\end{split}\eeno as
$\|S_{j'-1}\D_{k'}^va\|_{L^\infty_h(L^2_v)}\lesssim
2^{-\f{k'}2}\|a\|_{\cB^{1,\f12}}.$ Similar estimate holds for
$\D_j\D_k^v(\bar{T}R^v(a,b)).$

Along the same line, we have \beno
\begin{split}
\|\D_j\D_k^v(RR^v(a,b))\|_{L^2}&\lesssim
2^{j}2^{\f{k}{2}}\sum_{{j'\geq j-N_0}\atop{k'\geq k-N_0}}
\|\D_{j'}\D_{k'}^va\|_{L^2}\|\wt{\D}_{j'}\wt{\D}_{k'}^vb\|_{L^2}\\
&\lesssim 2^{j}2^{\f{k}{2}}\sum_{{j'\geq j-N_0}\atop{k'\geq k-N_0}}d_{j',k'}2^{-j'(s+1)}2^{-\f{k'}2}\|a\|_{\cB^{1,\f12}}\|b\|_{\cB^{s,0}}\\
&\lesssim d_{j,k}2^{-js} \|a\|_{\dB^{\f32}}\|b\|_{\cB^{s,0}}
\end{split} \eeno due to the fact:  $s+1>0.$  The estimate to the remaining terms in
\eqref{C10} is identical, and we omit the details here.

Whence thanks to \eqref{C10}, we arrive at \beno
\|\D_j\D_k^v(ab)\|_{L^2}\lesssim d_{j,k}2^{-js}\bigl(
\|a\|_{\dB^{\f32}}\|b\|_{\cB^{s,0}}+\|b\|_{\dB^{\f32}}\|a\|_{\cB^{s,0}}\bigr),
\eeno which implies the first inequality of \eqref{C9jh}. Exactly
along the same line, we can prove the second inequality of
\eqref{C9jh}. Finally  notice from Lemma \ref{L1} that
$\dB^{\f32-\d}(\R^3) \hookrightarrow \cB^{1,\f12-\d}(\R^3)$ and
$\dB^{s+\d}(\R^3)\hookrightarrow \cB^{s,\d}(\R^3)$ for $\d\in
(0,\f12),$ the proof of \eqref{C9bn} is identical to that of
\eqref{C9jh}, we omit the details here. This concludes the proof of
Lemma \ref{L7}.
\end{proof}

\setcounter{equation}{0}
\section{$L^1_T(\cB^{s+2,0})$ estimate of $Y_t$ for $s=\f12$ and $s>1$}\label{sect4}

\subsection{The estimate of $\|Y_t\|_{L^1_T(\cB^{s+2,0})}$  for the linearized system \eqref{B19}}

\begin{prop}\label{p1}{\sl  Let $Y$ be a smooth enough solution of \eqref{B19} on $[0,T]$. Then for
any $s\in\R,$ there holds \beq\label{d1}\begin{aligned}
\|Y_t\|_{\wt{L}^\infty_T(\cB^{s,0})}&+\|\p_3Y\|_{\wt{L}^\infty_T(\cB^{s,0})}+\|Y\|_{\wt{L}^\infty_T(\cB^{s+2,0})}
+\|Y_t\|_{L^1_T(\cB^{s+2,0})}\\
&+\|\p_3Y\|_{\wt{L}^2_T(\cB^{s+1,0})}\lesssim\|Y_1\|_{\cB^{s,0}}+\|\p_3
Y_0\|_{\cB^{s,0}}+\| Y_0\|_{\cB^{s+2,0}}+\|\vv
f\|_{L^1_T(\cB^{s,0})}.\end{aligned}\eeq}
\end{prop}
\begin{proof}
We first get, by applying $\D_j\D_k^v$ to \eqref{B19}, that
\beq\label{d2}
\D_j\D_k^vY_{tt}-\D\D_j\D_k^vY_t-\p_3^2\D_j\D_k^vY=\D_j\D_k^v\vv f.
\eeq Taking the $L^2$ inner product of \eqref{d2} with
$\D_j\D_k^vY_t$ gives \beq\label{d3}
\f12\f{d}{dt}\Bigl(\|\D_j\D_k^vY_t\|_{L^2}^2+\|\p_3\D_j\D_k^vY\|_{L^2}^2\Bigr)+\|\na\D_j\D_k^vY_t\|_{L^2}^2
=(\D_j\D_k^v \vv f\ |\ \D_j\D_k^vY_t).\eeq While taking the $L^2$
inner product of \eqref{d2} with $\D\D_j\D_k^vY$ leads to \beno
(\D_j\D_k^vY_{tt}\ |\ \D\D_j\D_k^vY)-\f12\f{d}{dt}\|\D\D_j\D_k^v
Y\|_{L^2}^2-\|\p_3\na\D_j\D_k^vY\|_{L^2}^2=(\D_j\D_k^v \vv f\ |\
\D\D_j\D_k^vY). \eeno

Notice that \beno (\D_j\D_k^vY_{tt}\ |\
\D\D_j\D_k^vY)=\f{d}{dt}(\D_j\D_k^vY_{t}\ |\
\D\D_j\D_k^vY)-(\D_j\D_k^vY_{t}\ |\ \D\D_j\D_k^vY_t), \eeno so that
there holds \beq\label{d4}
\begin{split}
\f{d}{dt}\Bigl(\f12\|\D\D_j\D_k^v Y\|_{L^2}^2-&(\D_j\D_k^vY_{t}\ |\
\D\D_j\D_k^vY)\Bigr)\\
&-\|\na\D_j\D_k^vY_t\|_{L^2}^2+\|\p_3\na\D_j\D_k^vY\|_{L^2}^2=-(\D_j\D_k^v\vv
f\ |\ \D\D_j\D_k^vY).
\end{split}
\eeq
\eqref{d3}+$\f14$\eqref{d4} gives rise to
\beq\label{d5}
\begin{split}
\f{d}{dt}g_{j,k}^2(t)+\f34\|\na\D_j\D_k^vY_t\|_{L^2}^2&+\f14\|\p_3\na\D_j\D_k^vY\|_{L^2}^2\\
&=\bigl(\D_j\D_k^v\vv f\ |\ \D_j\D_k^vY_t-\f14\D\D_j\D_k^vY\bigr),
\end{split}
\eeq
where
 \beno\begin{split}
g_{j,k}^2(t)\eqdefa\f12\bigl(\|\D_j\D_k^vY_t(t)\|_{L^2}^2&+\|\p_3\D_j\D_k^vY(t)\|_{L^2}^2
\\
&+\f14\|\D\D_j\D_k^vY(t)\|_{L^2}^2\bigr)-\f14\bigl(\D_j\D_k^vY_t(t)\
|\ \D\D_j\D_k^vY(t)\bigr).
\end{split}
\eeno

It is easy to observe that
 \beq \label{d6}
g_{j,k}^2(t)\sim\|\D_j\D_k^vY_t(t)\|_{L^2}^2+\|\p_3\D_j\D_k^vY(t)\|_{L^2}^2+\|\D\D_j\D_k^vY(t)\|_{L^2}^2.
\eeq With \eqref{d5}, \eqref{d6}, according to the heuristic
discussions in Subsection \ref{sect3.1} and similar to that in
\cite{LZ,XLZMHD1}, we shall separate the analysis  of \eqref{d5}
into two cases: one is when $j\leq \f{k+1}2,$ and the other one is
when $j> \f{k+1}2.$

\no {\bf Case (1)}: $j\leq \f{k+1}2.$ In this case, we infer from
Lemma \ref{L2} and \eqref{d6} that \beno g_{j,k}^2(t)\sim
\|\D_j\D_k^vY_t(t)\|_{L^2}^2+\|\p_3\D_j\D_k^vY(t)\|_{L^2}^2, \eeno
and \beno
\begin{split}
\|\na\D_j\D_k^vY_t(t)&\|_{L^2}^2+\|\p_3\na\D_j\D_k^vY(t)\|_{L^2}^2\\
&\geq
c2^{2j}\bigl(\|\D_j\D_k^vY_t(t)\|_{L^2}^2+\|\p_3\D_j\D_k^vY(t)\|_{L^2}^2\bigr)
\geq c2^{2j}g_{j,k}^2(t), \end{split}\eeno from which, for any
$\e>0,$ dividing \eqref{d5} by $g_{j,k}(t)+\e,$ then taking $\e\to
0$ and integrating the resulting equation over $[0,T],$ we obtain
\beq \label{d7}
\begin{split}
\|&\D_j\D_k^vY_t\|_{L^\infty_T(L^2)}+\|\p_3\D_j\D_k^vY\|_{L^\infty_T(L^2)}+\|\D\D_j\D_k^vY\|_{L^\infty_T(L^2)}\\
&\quad+
c2^{2j}\bigl(\|\D_j\D_k^vY_t\|_{L^1_T(L^2)}+\|\p_3\D_j\D_k^vY\|_{L^1_T(L^2)}\bigr)\\
&\lesssim
\|\D_j\D_k^vY_1\|_{L^2}+\|\p_3\D_j\D_k^vY_0\|_{L^2}+\|\D_j\D_k^v\vv
f\|_{L^1_T(L^2)}.
\end{split}
\eeq

\no {\bf Case (2)}: $j> \f{k+1}2.$  Notice from Lemma \ref{L2} that
in this case, one has
 \beno g_{j,k}^2(t) \sim
\|\D_j\D_k^vY_t(t)\|_{L^2}^2+\|\D\D_j\D_k^vY(t)\|_{L^2}^2, \eeno
and
 \beno
\begin{split}
&\|\na\D_j\D_k^vY_t(t)\|_{L^2}^2+\|\p_3\na\D_j\D_k^vY(t)\|_{L^2}^2\\
&\qquad\geq
c\frac{2^{2k}}{2^{2j}}\bigl(\|\D_j\D_k^vY_t(t)\|_{L^2}^2+\|\D\D_j\D_k^vY(t)\|_{L^2}^2\bigr)
\geq c\frac{2^{2k}}{2^{2j}} g_{j,k}^2(t),
\end{split}
\eeno from which and \eqref{d5},  we deduce by a similar derivation
of \eqref{d7}
 that \beq \label{d8}
\begin{split}
\|&\D_j\D_k^vY_t\|_{L^\infty_T(L^2)}+\|\p_3\D_j\D_k^vY\|_{L^\infty_T(L^2)}+\|\D\D_j\D_k^vY\|_{L^\infty_T(L^2)}\\
&\qquad+
c\f{2^{2k}}{2^{2j}}\bigl(\|\D_j\D_k^vY_t\|_{L^1_T(L^2)}+\|\D\D_j\D_k^vY\|_{L^1_T(L^2)}\bigr)\\
&\lesssim \|\D_j\D_k^vY_1\|_{L^2}+\|\D\D_j\D_k^vY_0\|_{L^2}+
\|\D_j\D_k^v\vv f\|_{L^1_T(L^2)}.
\end{split}
\eeq

On the other hand, standard energy estimate applied to  \eqref{d2}
yields that \beno
\f12\f{d}{dt}\|\D_j\D_k^vY_t(t)\|_{L^2}^2+\|\na\D_j\D_k^vY_t(t)\|_{L^2}^2=\bigl(\p_3^2\D_j\D_k^vY+\D_j\D_k^v\vv
f\ |\ \D_j\D_k^vY_t\bigr), \eeno from which, Lemma \ref{L2} and
\eqref{d8}, we infer \beq\label{d9}
\begin{split}
\|\D_j\D_k^v&Y_t\|_{L^\infty_T(L^2)}+c2^{2j}\|\D_j\D_k^vY_t\|_{L^1_T(L^2)}\\
&\leq \|\D_j\D_k^vY_1\|_{L^2}+C\bigl(2^{2k}\|\D_j\D_k^vY\|_{L^1_T(L^2)}+\|\D_j\D_k^v\vv f\|_{L^1_T(L^2)}\bigr)\\
&\lesssim
\|\D_j\D_k^vY_1\|_{L^2}+\|\D\D_j\D_k^vY_0\|_{L^2}+\|\D_j\D_k^v\vv
f\|_{L^1_T(L^2)}\quad\mbox{for}\quad j>\f{k+1}2.
\end{split}
\eeq

Therefore according to Definitions \ref{def2} and \ref{def3}, we get,
by summing up \eqref{d7}, \eqref{d8} and \eqref{d9}, that
\beq\label{d9op}\begin{aligned}
\|Y_t\|_{\wt{L}^\infty_T(\cB^{s,0})}+\|\p_3Y\|_{\wt{L}^\infty_T(\cB^{s,0})}&+\|Y\|_{\wt{L}^\infty_T(\cB^{s+2,0})}
+\|Y_t\|_{L^1_T(\cB^{s+2,0})}\\
&\lesssim\|Y_1\|_{\cB^{s,0}}+\|\p_3 Y_0\|_{\cB^{s,0}}+\|
Y_0\|_{\cB^{s+2,0}}+\|\vv f\|_{L^1_T(\cB^{s,0})}.\end{aligned}\eeq
On the other hand, it follows from \eqref{d5} that \beno
\|\p_3\na\D_j\D_k^vY\|_{L^2_T(L^2)}\lesssim \|\D_j\D_k^v\vv
f\|_{L^1_T(L^2)}+
\|\D_j\D_k^vY_t\|_{L^\infty_T(L^2)}+\|\D\D_j\D_k^vY\|_{L^\infty_T(L^2)},
\eeno so that \beno \|\p_3Y\|_{\wt{L}^2_T(\cB^{s+1,0})}\lesssim \|\vv
f\|_{L^1_T(\cB^{s,0})}+\|Y\|_{\wt{L}^\infty_T(\cB^{s,0})}+\|Y\|_{\wt{L}^\infty_T(\cB^{s+2,0})},
\eeno which together with \eqref{d9op} concludes the proof of
\eqref{d1}.
\end{proof}

\subsection{$L^1_T(\cB^{s,0})$ estimate of $\vv f(Y,q)$ given by
\eqref{B13}}

\begin{prop}\label{p2} {\sl Let $(Y, q)$ be a smooth enough solution of \eqref{B11} (or equivalently of
 \eqref{B12}-\eqref{B13}) on $[0,T]$ with the initial data $(Y_0,Y_1)$ satisfying $\mathrm{det}\, (I+\na
Y_0)=1$ and $\na_{Y_0}\cdot Y_1=0$. If we assume moreover that
\beq\label{d10} \|\na Y\|_{L^\infty_T(\dot{B}^{\f{3}{2}})}\leq
c_0\quad\mbox{and}\quad \bigl\langle\|
Y\|_{L^\infty_T(\dot{B}^{s+\f{5}{4}})}+\|
Y\|_{L^\infty_T(\dot{B}^{s+\f{3}{2}})}\bigr\rangle_{s>1}\leq 1 \eeq
for some $c_0$ sufficiently small. Then there holds \beq
\label{d10ap} \|\na q\|_{L^1_T(\cB^{s,0})}\lesssim\|
Y_t\|_{L^2_T(\dot{B}^{\f{9}{4}})}
\|Y_t\|_{L^2_T(\dot{B}^{s+\f{1}{4}})}+\|\p_3
Y\|_{L_T^2(\dot{B}^{\f{9}{4}})}\|\p_3
Y\|_{L_T^2(\dot{B}^{s+\f{1}{4}})} \eeq if $0<s\leq 1,$ and
\beq\label{d10ah}
\begin{split} \|\na
q\|_{L^1_T(\cB^{s,0})}\lesssim&\|Y_t\|_{L^2_T(\dot{B}^{\f{5}{4}})}^2+\|
Y_t\|_{L^2_T(\dot{B}^{\f{5}{2}})}^2+\|Y_t\|_{L^2_T(\dot{B}^{s+\f{5}{4}})}^2\\
&+ \|\p_3 Y\|_{L_T^2(\dot{B}^{\f{3}{2}})}^2+ \|\p_3
Y\|_{L_T^2(\dot{B}^{\f{9}{4}})}^2+\|\p_3
Y\|_{L_T^2(\dot{B}^{s+\f{1}{4}})}^2\quad\mbox{if}\ \  s>1.
\end{split}
\eeq Here and in all that follows, $A_s=B_s+\w{C_s}_{s>s_0}$ means
$A_s=B_s$ if $s\leq s_0$ and $A_s=B_s+C_s$ if $s>s_0.$}
\end{prop}
\begin{proof}
By virtue of \eqref{B11}, we get, by taking $\p_t$ to $\na_Y\cdot
Y_t=0,$ that \beq\label{d11} \na_Y\cdot Y_{tt}=-\p_t\cA_Y^T\na\cdot
Y_t. \eeq Note that for $c_0$ in \eqref{d10} being so small that
\beno \|\na Y\|_{L^\infty_T(L^\infty)}\leq C\|\na
Y\|_{L^\infty_T(\dot{B}^{\f32})}\leq Cc_0\leq\f{1}{2}, \eeno
$X(t,y)$ determined by \eqref{B6} has a smooth inverse map
$X^{-1}(t,x)$ with $X(t,X^{-1}(t,x))=x$ and $ X^{-1}(t,X(t,y))=y.$
Moreover, as $\mathrm{det}\, (I+\na Y_0)=1,$ we deduce from
\eqref{B13ag} that $\mathrm{det}\, (I+\na Y)=1,$ which together with
$\na_Y\cdot Y_t=0$ ensures that \beno \na_Y\cdot(\na_Y\cdot\na_Y
Y_t)=[\na_x\cdot\D_x(Y_t\circ X^{-1}(t,x))]\circ
X(t,y)=\na_Y\cdot\na_Y(\na_Y\cdot Y_t)=0,\eeno from which and
\eqref{d11}, we get, by  taking $\na_Y\cdot$ to the first equation
of \eqref{B11}, that \beno \na_Y\cdot\na_Y q=\p_t\cA_Y^T\na\cdot
Y_t+\na_Y\cdot\p_3^2Y, \eeno  or equivalently \beq\label{d12}
\begin{split}
\D q=&-(\na_Y-\na)\cdot\na_Y q-\na\cdot(\na_Y-\na)q+\p_t\cA_Y^T\na\cdot Y_t+\na_Y\cdot\p_3^2Y\\
=&-\na\cdot\bigl((\cA_Y-I)\cA_Y^T\na
q\bigr)-\na\cdot\bigl((\cA_Y^T-I)\na q\bigr)+\na\cdot(\p_t\cA_Y
Y_t)+\na_Y\cdot\p_3^2Y, \end{split} \eeq  from which, Lemma
\ref{L2} and Definition \ref{def2}, we infer
\beq\label{d13}\begin{aligned}
\|\na q\|_{L^1_T(\cB^{s,0})}\leq&\|(\cA_Y-I)\cA_Y^T\na q\|_{L^1_T(\cB^{s,0})}+\|(\cA_Y^T-I)\na q\|_{L^1_T(\cB^{s,0})}\\
&+\|\p_t\cA_Y
Y_t\|_{L^1_T(\cB^{s,0})}+\|\na_Y\cdot\p_3^2Y\|_{L^1_T(\cB^{s-1,0})}.
\end{aligned}\eeq
Applying \eqref{C9ag} and \eqref{B8} gives for $0<s\leq1$ \beno
\begin{split}
\|(\cA_Y^T-I)\na q\|_{L^1_T(\cB^{s,0})}&+\|(\cA_Y-I)\cA_Y^T\na q\|_{L^1_T(\cB^{s,0})}\\
&\lesssim\bigl(\|\cA_Y^T-I\|_{L^\infty_T(\dot{B}^{\f{3}{2}})}
+\|(\cA_Y-I)\cA_Y^T\|_{L^\infty_T(\dot{B}^{\f{3}{2}})}\bigr)
\|\na q\|_{L^1_T(\cB^{s,0})}\\
&\lesssim\bigl(1+\|\na
Y\|_{L^\infty_T(\dot{B}^{\f{3}{2}})}\bigr)^3\|\na
Y\|_{L^\infty_T(\dot{B}^{\f{3}{2}})}\|\na q\|_{L^1_T(\cB^{s,0})},
\end{split}
\eeno and
 \beno 
\begin{split}
\|\p_t\cA_Y Y_t\|_{L^1_T(\cB^{s,0})}\lesssim
&\|\p_t\cA_Y\|_{L^2_T(\dot{B}^{\f{5}{4}})}
\|Y_t\|_{L^2_T(\dot{B}^{s+\f{1}{4}})}\\
\lesssim & \bigl(1+\|\na
Y\|_{L^\infty_T(\dot{B}^{\f{3}{2}})}\bigr)\|\na
Y_t\|_{L^2_T(\dot{B}^{\f{5}{4}})}
\|Y_t\|_{L^2_T(\dot{B}^{s+\f{1}{4}})}.
\end{split}
\eeno While thanks to \eqref{B8}, \eqref{B12} and \eqref{B13}, a
tedious yet interesting calculation shows that  \beq\label{d18}
\na_Y\cdot\p_3^2Y=\na\cdot\bigl((\cA_Y-I)\p_3^2Y\bigr)+\p_3^2\r(Y)
=Q(\na\p_3Y,\na\p_3Y,\na Y),\eeq where $Q(\na\p_3Y,\na\p_3Y,\na Y)$
is a linear combination of quadratic terms like
$\p_3\p_iY^j\p_3\p_kY^l$ and cubic terms like
$\p_pY^q\p_3\p_iY^j\p_3\p_kY^l$. Then applying \eqref{C9ag} to
\eqref{d18} ensures that for $0<s\leq2$ \beq
\label{d19}\begin{aligned}
\|\na_Y\cdot\p_3^2Y\|_{L^1_T(\cB^{s-1,0})}\lesssim &\|(I+\na
Y)\p_3\na Y\|_{L_T^2(\dot{B}^{\f{5}{4}})}\|\p_3\na
Y\|_{L_T^2(\dot{B}^{s-\f{3}{4}})}\\
\lesssim &\bigl(1+\|\na
Y\|_{L^\infty_T(\dot{B}^{\f{3}{2}})}\bigr)\|\p_3\na
Y\|_{L_T^2(\dot{B}^{\f{5}{4}})}\|\p_3
Y\|_{L_T^2(\dot{B}^{s+\f{1}{4}})}.
\end{aligned}\eeq
Thus for $0<s\leq 1,$ resuming the above estimates into \eqref{d13}
gives rise to \beno
\begin{split}
\|\na q\|_{L^1_T(\cB^{s,0})}\lesssim & \bigl(1+\|\na
Y\|_{L^\infty_T(\dot{B}^{\f{3}{2}})}\bigr)^3\Bigl(\|\na
Y\|_{L^\infty_T(\dot{B}^{\f{3}{2}})}\|\na
q\|_{L^1_T(\cB^{s,0})}\\
&+\|Y_t\|_{L^2_T(\dot{B}^{\f{9}{4}})}
\|Y_t\|_{L^2_T(\dot{B}^{s+\f{1}{4}})}+\|\p_3
Y\|_{L_T^2(\dot{B}^{\f{9}{4}})}\|\p_3
Y\|_{L_T^2(\dot{B}^{s+\f{1}{4}})}\Bigr), \end{split} \eeno  which
along with  \eqref{d10} implies (\ref{d10ap}).

On the other hand, we get, by applying \eqref{C9jh} and
\eqref{C9bn}, that
 for $s>1$ \beq\label{d16}
\begin{split}
\|(&\cA_Y^T-I)\na q\|_{L^1_T(\cB^{s,0})}+\|(\cA_Y-I)\cA_Y^T\na q\|_{L^1_T(\cB^{s,0})}\\
\lesssim&\Bigl(\|\cA_Y^T-I\|_{L^\infty_T(\dot{B}^{\f{3}{2}})}+\|(\cA_Y-I)\cA_Y^T\|_{L^\infty_T(\dot{B}^{\f{3}{2}})}\Bigr)\|\na
q\|_{L^1_T(\cB^{s,0})}\\
&+\Bigl(\|\cA_Y^T-I\|_{L^\infty_T(\dot{B}^{s+\f{1}{2}})}
+\|(\cA_Y-I)\cA_Y^T\|_{L^\infty_T(\dot{B}^{s+\f{1}{2}})}\Bigr)
\|\na q\|_{L^1_T(\cB^{1,0})}\\
\lesssim &\bigl(1+\|\na
Y\|_{L^\infty_T(\dot{B}^{\f{3}{2}})}\bigr)^3\Bigl(\|\na
Y\|_{L^\infty_T(\dot{B}^{\f{3}{2}})}\|\na q\|_{L^1_T(\cB^{s,0})}
+\|\na Y\|_{L^\infty_T(\dot{B}^{s+\f{1}{2}})}\|\na
q\|_{L^1_T(\cB^{1,0})}\Bigr),
\end{split}
\eeq and \beq\label{d17}
\begin{split}
\|\p_t\cA_Y Y_t\|_{L^1_T(\cB^{s,0})}\lesssim&
\|\p_t\cA_Y\|_{L^2_T(\dot{B}^{\f{5}{4}})}\|Y_t\|_{L^2_T(\dot{B}^{s+\f{1}{4}})}
+\|\p_t\cA_Y\|_{L^2_T(\dot{B}^{s+\f{1}{4}})}\|Y_t\|_{L^2_T(\dot{B}^{\f{5}{4}})}\\
\lesssim&\bigl(1+\|\na Y\|_{L^\infty_T(\dot{B}^{\f{3}{2}})}+\|\na
Y\|_{L^\infty_T(\dot{B}^{s+\f14})}\bigr) \Bigl(\|\na
Y_t\|_{L^2_T(\dot{B}^{\f{5}{4}})}\|Y_t\|_{L^2_T(\dot{B}^{s+\f{1}{4}})}\\
&+\bigl(\|\na Y_t\|_{L^2_T(\dot{B}^{\f{3}{2}})}+ \|\na
Y_t\|_{L^2_T(\dot{B}^{s+\f{1}{4}})}\bigr)\|Y_t\|_{L^2_T(\dB^{\f54})}\Bigr).
\end{split}
\eeq
 Moreover, applying \eqref{C9bn} to \eqref{d18} leads to
 \begin{multline*}
\|\na_Y\cdot\p_3^2Y\|_{L^1_T(\cB^{s-1,0})}\lesssim \|(I+\na
Y)\p_3\na Y\|_{L_T^2(\dot{B}^{\f{5}{4}})}\|\p_3\na
Y\|_{L_T^2(\dot{B}^{s-\f{3}{4}})}\\
 +\|(I+\na Y)\p_3\na Y\|_{L_T^2(\dot{B}^{s-\f{3}{4}})}\|\p_3\na
Y\|_{L_T^2(\dot{B}^{\f{5}{4}})}, \end{multline*} however, by
applying Bony's decomposition \eqref{C7}, one has \begin{multline*}
\|(I+\na Y)\p_3\na Y\|_{L_T^2(\dot{B}^{s-\f{3}{4}})}\lesssim
\bigl(1+\|\na Y\|_{L^\infty_T(\dB^{\f32})}\bigr)\|\p_3\na
Y\|_{L_T^2(\dot{B}^{s-\f{3}{4}})}\\
+\|\na Y\|_{L^\infty_T(\dB^{s+\f14})}\|\p_3\na
Y\|_{L_T^2(\dot{B}^{\f{1}{2}})},\end{multline*} so that for $s>2,$
we achieve \beq\label{d20} \begin{aligned}
\|\na_Y\cdot\p_3^2Y\|_{L^1_T(\cB^{s-1,0})}\lesssim &\Bigl(1+\|\na
Y\|_{L^\infty_T(\dot{B}^{\f{3}{2}})}+\|
Y\|_{L^\infty_T(\dot{B}^{s+\f{5}{4}})}\Bigr)\\
&\quad\times \Bigl(\|\p_3 Y\|_{L_T^2(\dot{B}^{\f{9}{4}})}^2+ \|\p_3
Y\|_{L_T^2(\dot{B}^{\f{3}{2}})}^2+\|\p_3
Y\|_{L_T^2(\dot{B}^{s+\f{1}{4}})}^2\Bigr).
\end{aligned}\eeq

Whence plugging \eqref{d16}, \eqref{d17}, \eqref{d19} and
\eqref{d20} into \eqref{d13}, we obtain  \beno
\begin{split}
\|\na q\|_{L^1_T(\cB^{s,0})}\lesssim &\bigl(1+\|\na
Y\|_{L^\infty_T(\dot{B}^{\f{3}{2}})}\bigr)^3\Bigl(\|\na
Y\|_{L^\infty_T(\dot{B}^{\f{3}{2}})}\|\na q\|_{L^1_T(\cB^{s,0})}\\
&+\|Y\|_{L^\infty_T(\dot{B}^{s+\f{3}{2}})}\|\na
q\|_{L^1_T(\cB^{1,0})}\Bigr) +\bigl(1+\|\na
Y\|_{L^\infty_T(\dot{B}^{\f{3}{2}})}+\|
Y\|_{L^\infty_T(\dot{B}^{s+\f{5}{4}})}\bigr)\\
& \times\Bigl(\|Y_t\|_{L^2_T(\dot{B}^{\f{5}{4}})}^2+\|
Y_t\|_{L^2_T(\dot{B}^{\f{5}{2}})}^2+\|Y_t\|_{L^2_T(\dot{B}^{s+\f{5}{4}})}^2\\
&+  \|\p_3
Y\|_{L_T^2(\dot{B}^{\f{3}{2}})}^2+\|\p_3Y\|_{L_T^2(\dot{B}^{\f{9}{4}})}^2+\|\p_3
Y\|_{L_T^2(\dot{B}^{s+\f{1}{4}})}^2\Bigr)\quad\mbox{for}\ \ s>1,
\end{split} \eeno which together with  \eqref{d10} and \eqref{d10ap} implies \eqref{d10ah}. This completes the
proof of Proposition \ref{p2}.
\end{proof}

\begin{col}\label{col2}
{\sl Under the assumption of Proposition \ref{p2}, one has
\beq\label{cc3} \|\na
q\|_{L^1_T(\dB^{\f12})}\lesssim\|Y_t\|_{L^2_T(\dB^{\f32})}^2+\|\p_3Y\|_{L^2_T(\dB^{\f32})}^2.
\eeq}
\end{col}
\begin{proof}
We first deduce from \eqref{d12} that \beq\label{cc4}\begin{aligned}
\|\na q\|_{L^1_T(\dB^{\f12})}\leq&\|(\cA_Y-I)\cA_Y^T\na q\|_{L^1_T(\dB^{\f12})}+\|(\cA_Y^T-I)\na q\|_{L^1_T(\dB^{\f12})}\\
&+\|\p_t\cA_Y
Y_t\|_{L^1_T(\dB^{\f12})}+\|\na_Y\cdot\p_3^2Y\|_{L^1_T(\dB^{-\f12})}.
\end{aligned}\eeq
It follows from product laws in Besov space (see \cite{bcd} for
instance) that
\begin{multline*}
\|(\cA_Y-I)\cA_Y^T\na q\|_{L^1_T(\dB^{\f12})}+\|(\cA_Y^T-I)\na q\|_{L^1_T(\dB^{\f12})}\\
\lesssim\bigl(1+\|\na Y\|_{L^\infty_T(\dB^{\f{3}{2}})}\bigr)^3\|\na
Y\|_{L^\infty_T(\dB^{\f{3}{2}})}\|\na q\|_{L^1_T(\dB^{\f12})},
\end{multline*}
and \beno \|\p_t\cA_Y
Y_t\|_{L^1_T(\dB^{\f12})}\lesssim\|\p_t\cA_Y\|_{L^2_T(\dB^{\f12})}\|Y_t\|_{L^2_T(\dB^{\f32})}
\lesssim\bigl(1+\|\na
Y\|_{L^\infty_T(\dB^{\f{3}{2}})}\bigr)\|Y_t\|_{L^2_T(\dB^{\f32})}^2.
\eeno Along the same line, we deduce from \eqref{d18} that
\begin{multline*}
\|\na_Y\cdot\p_3^2Y\|_{L^1_T(\dB^{-\f12})}\lesssim\|Q(\na\p_3Y,\na\p_3Y,\na Y)\|_{L^1_T(\dB^{-\f12})}\\
\lesssim\|(I+\na
Y)\na\p_3Y\|_{L^2_T(\dB^{\f12})}\|\na\p_3Y\|_{L^2_T(\dB^{\f12})}
\lesssim\bigl(1+\|\na
Y\|_{L^\infty_T(\dB^{\f{3}{2}})}\bigr)\|\p_3Y\|_{L^2_T(\dB^{\f32})}^2.
\end{multline*}
Resuming the above estimates into \eqref{cc4} leads to \eqref{cc3}.
\end{proof}

\begin{prop}\label{p3} {\sl Let $s>1$, and $\vv f(Y,q)$ be given by \eqref{B13}.
Then under the assumptions of Proposition \ref{p2} and
\beq\label{A2}
\|Y\|_{L^\infty_T(\cB^{s+2,0})}\leq1,
\eeq
 one has
\beq\label{d27}\begin{aligned} &\|\vv
f(Y,q)\|_{L^1_T(\cB^{\f{1}{2},0})}+\|\vv
f(Y,q)\|_{L^1_T(\cB^{s,0})}\lesssim\|\na q\|_{L^1_T(\cB^{\f12,0})}
+\|\na Y\|_{L^\infty_T(\dB^{\f{3}{2}})}\|Y_t\|_{L^1_T(\cB^{s+2,0})}\\
&\ +\|\na q\|_{L^1_T(\cB^{s,0})}
+\|Y_t\|_{L^1_T(\dB^{\f{5}{2}})}\Bigl(\|Y\|_{L^\infty_T(\cB^{s+2,0})}+\|\na
Y\|_{L^\infty_T(\dB^{\f{3}{2}})}\Bigr).
\end{aligned}\eeq}
\end{prop}
\begin{proof} Thanks to \eqref{B13},
 we split $\vv f(Y,q)$ as follows:
\beq\label{d22}\begin{aligned}
&\vv f(Y,q)=\bar{\vv f}(Y)+\tilde{\vv f}(Y,q),\quad\text{with}\\
&\bar{\vv
f}(Y)\eqdefa(\na_Y\cdot\na_Y-\Delta)Y_t\quad\mbox{and}\quad\tilde{\vv
f}(Y,q)\eqdefa-\na_Yq.
\end{aligned}\eeq
As $\tilde{\vv f}(Y,q)=-(\cA^T-I)\na q-\na q,$ by virtue of
\eqref{B8}, we deduce from \eqref{C9ag} and  \eqref{d10} that
\beq\label{d23}\begin{aligned}
\|\tilde{\vv f}(Y,q)\|_{L^1_T(\cB^{\f{1}{2},0})}&\lesssim\bigl(1+\|\cA^T-I\|_{L^\infty_T(\dot{B}^{\f{3}{2}})}\bigr)\|\na q\|_{L^1_T(\cB^{\f12,0})}\\
&\lesssim\bigl(1+\|\na
Y\|_{L^\infty_T(\dot{B}^{\f{3}{2}})}\bigr)^2\|\na
q\|_{L^1_T(\cB^{\f12,0})}\lesssim\|\na q\|_{L^1_T(\cB^{\f12,0})},
\end{aligned}\eeq
and for $s>1$, we infer from \eqref{C9jh} and \eqref{d10} that
\beq\label{d24}\begin{aligned} \|\tilde{\vv
f}(Y,q)\|_{L^1_T(\cB^{s,0})}\lesssim
&\|\cA^T-I\|_{L^\infty_T(\dot{B}^{\f{3}{2}})}\|\na
q\|_{L^1_T(\cB^{s,0})}\\
&+\|\cA^T-I\|_{L^\infty_T(\dot{B}^{s+\f{1}{2}})}\|\na q\|_{L^1_T(\cB^{1,0})}+\|\na q\|_{L^1_T(\cB^{s,0})}\\
\lesssim &\bigl(1+\|\na
Y\|_{L^\infty_T(\dot{B}^{\f{3}{2}})}\bigr)\Bigl(\|\na
Y\|_{L^\infty_T(\dot{B}^{\f{3}{2}})}\|\na q\|_{L^1_T(\cB^{s,0})}\\
&+\|\na Y\|_{L^\infty_T(\dot{B}^{s+\f{1}{2}})}\|\na q\|_{L^1_T(\cB^{1,0})}\Bigr)+\|\na q\|_{L^1_T(\cB^{s,0})}\\
\lesssim&\|\na q\|_{L^1_T(\cB^{\f12,0})}+\|\na
q\|_{L^1_T(\cB^{s,0})},
\end{aligned}\eeq
where in the last step,   we used  the trivial fact that
\beq\label{C14} \|a\|_{\cB^{s,\d}}\lesssim
\|a\|_{\cB^{s_1,\d}}+\|a\|_{\cB^{s_2,\d}}\quad\mbox{for any}\quad
s\in [s_1,s_2]. \eeq

On the other hand, notice that \beq\label{d31}\begin{aligned}
\bar{\vv f}^i(Y)&=(\na_Y-\na)\cdot\na_YY_t^i+\na\cdot(\na_Y-\na)Y_t^i\\
&=\na\cdot\bigl[(\cA_Y-I)\cA_Y^T\na Y_t^i+(\cA_Y^T-I)\na
Y_t^i\bigr],
\end{aligned}\eeq
which leads to \beq\label{d25} \|\bar{\vv f}(Y)\|_{L^1_T(\cB^{s,0})}
\lesssim\|(\cA_Y-I)\cA_Y^T\na
Y_t\|_{L^1_T(\cB^{s+1,0})}+\|(\cA_Y^T-I)\na
Y_t\|_{L^1_T(\cB^{s+1,0})}. \eeq Whereas according to  \eqref{B8},
we get, by applying \eqref{C9jh}, that \beno\begin{aligned}
\|(\cA_Y^T&-I)\na Y_t\|_{L^1_T(\cB^{s+1,0})}\\
\lesssim
&\|\cA_Y^T-I\|_{L^\infty_T(\dB^{\f{3}{2}})}\|\na
Y_t\|_{L^1_T(\cB^{s+1,0})}
+\|\cA_Y^T-I\|_{L^\infty_T(\cB^{s+1,0})}\|\na
Y_t\|_{L^1_T(\dB^{\f{3}{2}})} \\
\lesssim& \bigl(1+\|\na
Y\|_{L^\infty_T(\dB^{\f{3}{2}})}\bigr)\Bigl(\|\na
Y\|_{L^\infty_T(\dB^{\f{3}{2}})}\|Y_t\|_{L^1_T(\cB^{s+2,0})}+\|Y_t\|_{L^1_T(\dB^{\f{5}{2}})}\|\na
Y\|_{L^\infty_T(\cB^{s+1,0})} \Bigr).
\end{aligned}\eeno
Along the same line, and thanks to \eqref{d10} and \eqref{A2}, we
have \beno\begin{aligned} \|(\cA_Y-I)&\cA_Y^T\na
Y_t\|_{L^1_T(\cB^{s+1,0})}\lesssim \|\na
Y\|_{L^\infty_T(\dB^{\f{3}{2}})}\|Y_t\|_{L^1_T(\cB^{s+2,0})}\\
&\qquad+\|\na Y_t\|_{L^1_T(\dB^{\f32})}\Bigl(\|\na
Y\|_{L^\infty_T(\dB^{\f{3}{2}})}+\|\na
Y\|_{L^\infty_T(\cB^{s+1,0})}\Bigr).
\end{aligned}\eeno
Resuming the above two estimates into \eqref{d25} gives rise to
\beno
\begin{aligned} \|\bar{\vv
f}(Y)\|_{L^1_T(\cB^{s,0})}\lesssim & \|\na
Y\|_{L^\infty_T(\dB^{\f{3}{2}})}\|Y_t\|_{L^1_T(\cB^{s+2,0})}\\
&+\| Y_t\|_{L^1_T(\dB^{\f52})}\Bigl(\|\na
Y\|_{L^\infty_T(\dB^{\f{3}{2}})}+\|
Y\|_{L^\infty_T(\cB^{s+2,0})}\Bigr),
\end{aligned}\eeno
which together with \eqref{d23} and  \eqref{d24} ensures
\eqref{d27}. This concludes the proof of Proposition \ref{p3}.
\end{proof}

\subsection{$L^1_T(\cB^{s+2,0}(\R^3))$ estimate of $Y_t$ for $s=\f12$ and $s>1$}
\begin{prop}\label{p4}{\it
Let $s>1$, then under assumptions of Proposition \ref{p2} and
\beq\label{A3} \|Y\|_{L^\infty_T(\cB^{s+2,0})}\leq c_0, \eeq for
some $c_0$ sufficiently small, we have
\beq\label{d29}\begin{aligned}
\|Y_t&\|_{\wt{L}^\infty_T(\cB^{\f12,0}\cap\cB^{s,0})}
+\|\p_3Y\|_{\wt{L}^\infty_T(\cB^{\f12,0}\cap\cB^{s,0})}+\|Y\|_{\wt{L}^\infty_T(\cB^{\f52,0}\cap\cB^{s+2,0})}\\
&\quad+\|Y_t\|_{L^1_T(\cB^{\f52,0}\cap\cB^{s+2,0})}+\|\p_3Y\|_{\wt{L}^2_T(\cB^{\f32,0}\cap\cB^{s+1,0})}\\
&\lesssim\|Y_1\|_{\cB^{\f12,0}\cap\cB^{s,0}}+\|\p_3
Y_0\|_{\cB^{\f12,0}\cap\cB^{s,0}}+\|
Y_0\|_{\cB^{\f52,0}\cap\cB^{s+2,0}}+\|Y_t\|_{L^2_T(\dot{B}^{\f{3}{4}})}^2\\
&\quad+\|
Y_t\|_{L^2_T(\dot{B}^{\f{5}{2}})}^2+\|Y_t\|_{L^2_T(\dot{B}^{s+\f{5}{4}})}^2+
\|\p_3Y\|_{L_T^2(\dot{B}^{\f{3}{4}})}^2+
\|\p_3Y\|_{L_T^2(\dot{B}^{\f{9}{4}})}^2\\
&\quad+\|\p_3
Y\|_{L_T^2(\dot{B}^{s+\f{1}{4}})}^2+\|Y\|_{L^\infty_T(\cB^{s+2,0})}^2.\end{aligned}\eeq
}
\end{prop}
\begin{proof}
Thanks to Propositions \ref{p1} and \ref{p3}, we conclude that for
$s>1$, \beno\begin{aligned}
\|Y_t&\|_{\wt{L}^\infty_T(\cB^{\f12,0}\cap\cB^{s,0})}
+\|\p_3Y\|_{\wt{L}^\infty_T(\cB^{\f12,0}\cap\cB^{s,0})}+\|Y\|_{\wt{L}^\infty_T(\cB^{\f52,0}\cap\cB^{s+2,0})}\\
&\quad+\|Y_t\|_{L^1_T(\cB^{\f52,0}\cap\cB^{s+2,0})}+\|\p_3Y\|_{\wt{L}^2_T(\cB^{\f32,0}\cap\cB^{s+1,0})}\\
&\lesssim\|Y_1\|_{\cB^{\f12,0}\cap\cB^{s,0}}+\|\p_3
Y_0\|_{\cB^{\f12,0}\cap\cB^{s,0}}+\|
Y_0\|_{\cB^{\f52,0}\cap\cB^{s+2,0}}+\|\na q\|_{L^1_T(\cB^{\f12,0})}\\
&\quad+\|\na q\|_{L^1_T(\cB^{s,0})}+\|\na Y\|_{L^\infty_T(\dB^{\f{3}{2}})}\|Y_t\|_{L^1_T(\cB^{s+2,0})}\\
&\quad+\|Y_t\|_{L^1_T(\dB^{\f{5}{2}})}\Bigl(\|Y\|_{L^\infty_T(\cB^{s+2,0})}+\|\na
Y\|_{L^\infty_T(\dB^{\f{3}{2}})} \Bigr),\end{aligned}\eeno which
together  with \eqref{d10}, \eqref{A3} and the fact that
$\|Y_t\|_{L^1_T(\dB^{\f52})}\lesssim\|Y_t\|_{L^1_T(\cB^{\f52,0})}$
ensures that \beno
\begin{aligned}
\|Y_t&\|_{\wt{L}^\infty_T(\cB^{\f12,0}\cap\cB^{s,0})}
+\|\p_3Y\|_{\wt{L}^\infty_T(\cB^{\f12,0}\cap\cB^{s,0})}+\|Y\|_{\wt{L}^\infty_T(\cB^{\f52,0}\cap\cB^{s+2,0})}\\
&\quad+\|Y_t\|_{L^1_T(\cB^{\f52,0}\cap\cB^{s+2,0})}+\|\p_3Y\|_{\wt{L}^2_T(\cB^{\f32,0}\cap\cB^{s+1,0})}\\
&\lesssim\|Y_1\|_{\cB^{\f12,0}\cap\cB^{s,0}}+\|\p_3
Y_0\|_{\cB^{\f12,0}\cap\cB^{s,0}}+\|
Y_0\|_{\cB^{\f52,0}\cap\cB^{s+2,0}}\\
&\quad+\|\na q\|_{L^1_T(\cB^{\f12,0})}+\|\na
q\|_{L^1_T(\cB^{s,0})}.\end{aligned}\eeno
Hence applying Proposition \ref{p2} gives rise to \eqref{d29}. This
complete the proof of Proposition \ref{p4}.
\end{proof}

\setcounter{equation}{0}
\section{The proof of Theorem \ref{T}}\label{sect5}

\subsection{{\it A priori} estimate of \eqref{B11}} The goal of this subsection is to present the {\it a priori} energy estimate to smooth
enough solutions of \eqref{B11}.

\begin{lem}\label{p5}
{\sl Let  $Y$ be a smooth enough solution of \eqref{B12}-\eqref{B13}
(or equivalently \eqref{B11}) on $[0,T].$ Then there holds
 \beq \label{e1}
\begin{split}
&\|\Delta_jY_t\|_{L^\infty_T(L^2)}^2+\|\na\Delta_jY_t\|_{L^\infty_T(L^2)}^2+\|\p_3\Delta_jY\|_{L^\infty_T(L^2)}^2
+\|\Delta\Delta_jY\|_{L^\infty_T(L^2)}^2\\
&\quad
+\|\na\Delta_jY_t\|_{L^2_T(L^2)}^2+\|\D\Delta_jY_t\|_{L^2_T(L^2)}^2+\|\p_3\na\Delta_jY\|_{L^2_T(L^2)}^2
\\
&\lesssim\|\Delta_jY_1\|_{L^2}^2+\|\na\Delta_jY_1\|_{L^2}^2+\|\p_3\Delta_jY_0\|_{L^2}^2
+\|\Delta\Delta_jY_0\|_{L^2}^2\\
&\quad
+\Bigl|\int_0^T\bigl(\Delta_j\vv f\
|\ \Delta_jY_t-\f{1}{4}\Delta\Delta_jY-\D\D_jY_t\bigr)\,dt\Bigr|.
\end{split}\eeq
}\end{lem}
\begin{proof}
Applying $\Delta_j$ to \eqref{B12} gives
\beq\label{e2} \Delta_j
Y_{tt}-\Delta \Delta_j Y_t-\p_3^2 \Delta_j Y=\Delta_j\vv f. \eeq
 Taking the $L^2$ inner product of \eqref{e2} with
$\Delta_j Y_t-\f{1}{4}\Delta\Delta_jY-\Delta\Delta_jY_t$, we get, by
a similar derivation of \eqref{d5}, that
\beq\label{e3}\begin{aligned}
&\f{d}{dt}\Bigl\{\f{1}{2}\bigl(\|\Delta_jY_t\|_{L^2}^2+\|\na\Delta_jY_t\|_{L^2}^2+\|\p_3\Delta_jY\|_{L^2}^2
+\|\p_3\na\Delta_jY\|_{L^2}^2
+\f{1}{4}\|\Delta\Delta_jY\|_{L^2}^2\bigr)
\\
&\quad-\f{1}{4}(\Delta_jY_t\ |\ \Delta\Delta_jY)\Bigr\}+\f{3}{4}\|\na\Delta_jY_t\|_{L^2}^2+\|\D\Delta_jY_t\|_{L^2}^2
+\f{1}{4}\|\p_3\na\Delta_jY\|_{L^2}^2\\
&=\bigl(\Delta_j\vv
f\ |\ \Delta_jY_t-\f{1}{4}\Delta\Delta_jY-\Delta\Delta_jY_t\bigr).
\end{aligned}\eeq
However, as
 \beno
 \begin{split}
 &\f{1}{2}\bigl(\|\Delta_jY_t\|_{L^2}^2+\|\na\Delta_jY_t\|_{L^2}^2+\|\p_3\Delta_jY\|_{L^2}^2
 +\|\p_3\na\Delta_jY\|_{L^2}^2
+\f{1}{4}\|\Delta\Delta_jY\|_{L^2}^2\bigr)
\\
&\quad-\f{1}{4}(\Delta_jY_t\ |\ \Delta\Delta_jY)
\sim\|\D_jY_t(t)\|_{L^2}^2+\|\na\Delta_jY_t\|_{L^2}^2+\|\p_3\D_jY(t)\|_{L^2}^2+\|\D\D_jY(t)\|_{L^2}^2,
\end{split}
\eeno by integrating \eqref{e3} over $[0,T]$, we obtain \eqref{e1}.
\end{proof}

To deal with the last line of \eqref{e1}, we need to estimate the
$\vv f(Y,q)$ given by \eqref{B13}. Toward this, we first deal with
the the pressure term in \eqref{B12}.

\begin{lem}\label{p6}
Under the assumptions of Proposition \ref{p2},
 for any $s>-\f12$, one
has \beq\label{e4}\begin{aligned} \|\na q(t)\|_{\dot{H}^s}\lesssim
&\|Y(t)\|_{\dot{H}^{s+2}}\|\na
q(t)\|_{\dot{B}^{\f12}}+\|Y_t(t)\|_{\dot{B}^{\f32}}\|Y_t(t)\|_{\dot{H}^{s+1}}\\
&
+\|Y(t)\|_{\dH^{s+2}}\bigl(\|Y_t(t)\|_{\dot{B}^{\f32}}^2+\|\p_3Y(t)\|_{\dot{B}^{\f32}}^2\bigr)
+\|\p_3 Y(t)\|_{\dot{B}^{\f{3}{2}}}\|\p_3 Y(t)\|_{\dH^{s+1}},
\end{aligned}\eeq for all $t\in [0,T].$
\end{lem}
\begin{proof} For any $t\in [0,T],$ we deduce from
\eqref{d12} that \beq\label{e6}\begin{aligned}
\|\na q(t)\|_{\dot{H}^s}\leq&\|((\cA_Y-I)\cA_Y^T\na q)(t)\|_{\dot{H}^s}+\|((\cA_Y^T-I)\na q)(t)\|_{\dot{H}^s}\\
&+\|(\p_t\cA_Y
Y_t)(t)\|_{\dot{H}^s}+\|(\na_Y\cdot\p_3^2Y)(t)\|_{\dot{H}^{s-1}}.
\end{aligned}\eeq
By virtue of \eqref{B8}, we get, by using Bony's decomposition
\eqref{C7},  that for any $s>-\f12$, \beno
\begin{split}
&\|((\cA_Y-I)\cA_Y^T\na q)(t)\|_{\dot{H}^s}+\|((\cA_Y^T-I)\na q)(t)\|_{\dot{H}^s}\\
&\qquad\quad\lesssim\bigl(1+\|\na
Y(t)\|_{\dot{B}^{\f32}}\bigr)^3\Bigl(\|\na
Y(t)\|_{\dot{B}^{\f32}}\|\na q(t)\|_{\dot{H}^s}+\|\na
Y(t)\|_{\dot{H}^{s+1}}\|\na q(t)\|_{\dot{B}^{\f12}}\Bigr),
\end{split}
\eeno and \beno
\begin{split} \|(\p_t\cA_Y
Y_t)(t)\|_{\dot{H}^s}\lesssim&\bigl(1+\|\na
Y(t)\|_{\dot{B}^{\f32}}\bigr)
\|Y_t(t)\|_{\dot{B}^{\f32}}\|Y_t(t)\|_{\dot{H}^{s+1}}+\|\na
Y(t)\|_{\dot{H}^{s+1}}\|Y_t(t)\|_{\dot{B}^{\f32}}^2.
\end{split}
\eeno Along the same line, due to \eqref{d18}, we obtain for any
$s>-\f12$, \beno\begin{aligned}
&\|(\na_Y\cdot\p_3^2Y)(t)\|_{\dot{H}^{s-1}}\lesssim
\|Q(\na\p_3Y,\na\p_3Y,\na Y)(t)\|_{\dot{H}^{s-1}}\\
&\quad\lesssim\|((I+\na Y)\p_3\na Y)(t)\|_{\dot{B}^{\f{1}{2}}}\|\p_3
Y(t)\|_{\dH^{s+1}}+\|((I+\na Y)\p_3\na Y)(t)\|_{\dH^s}\|\p_3
Y(t)\|_{\dot{B}^{\f{3}{2}}}\\
&\quad\lesssim\bigl(1+\|\na Y(t)\|_{\dot{B}^{\f32}}\bigr)\|\p_3
Y(t)\|_{\dot{B}^{\f{3}{2}}}\|\p_3Y(t)\|_{\dH^{s+1}}+\|
Y(t)\|_{\dH^{s+2}}\|\p_3Y(t)\|_{\dot{B}^{\f{3}{2}}}^2.
\end{aligned}\eeno Resuming the above estimates into \eqref{e6} and
using  \eqref{d10} ensures that for any $s>-\f12$,
\beno\begin{aligned}  \|\na q(t)\|_{\dot{H}^s}\lesssim&\|\na
Y(t)\|_{\dB^{\f32}}\|\na q(t)\|_{\dot{H}^s}+ \|\na
Y(t)\|_{\dot{H}^{s+1}}\|\na
q(t)\|_{\dot{B}^{\f12}}+\|Y_t(t)\|_{\dot{B}^{\f32}}\|Y_t(t)\|_{\dot{H}^{s+1}}\\
&
+\|Y(t)\|_{\dH^{s+2}}\bigl(\|Y_t(t)\|_{\dot{B}^{\f32}}^2+\|\p_3Y(t)\|_{\dot{B}^{\f32}}^2\bigr)
+\|\p_3 Y(t)\|_{\dot{B}^{\f{3}{2}}}\|\p_3 Y(t)\|_{\dH^{s+1}},
\end{aligned}\eeno
for any $t\in [0,T],$ which together \eqref{d10}  leads to
\eqref{e4}.
\end{proof}

\begin{lem}\label{L8}
Under the assumptions of Proposition \ref{p2},  for  any $s>-\f12$,
we have \beq\label{e18}\begin{aligned} \|\vv f(Y,q)\|_{L^2_T(\dH^s)}
\lesssim&\|Y\|_{L^\infty_T(\dot{H}^{s+2})}\Bigl(\|\na
q\|_{L^2_T(\dot{B}^{\f12})}+\|
Y_t\|_{L^2_T(\dot{B}^{\f52})}+\|Y_t\|_{L^\infty_T(\dB^{\f32})}\|Y_t\|_{L^2_T(\dB^{\f32})}\\
&+\|\p_3Y\|_{L^\infty_T(\dB^{\f32})}\|\p_3Y\|_{L^2_T(\dB^{\f32})}\Bigr)
+\|Y_t\|_{L^\infty_T(\dot{B}^{\f32})}\|Y_t\|_{L^2_T(\dot{H}^{s+1})}\\
&+\|\p_3Y\|_{L^\infty_T(\dot{B}^{\f{3}{2}})}\|\p_3
Y\|_{L^2_T(\dH^{s+1})} +\|\na Y\|_{L^\infty_T(\dot{B}^{\f32})}\|
Y_t\|_{L^2_T(\dH^{s+2})}\\
&,
\end{aligned}\eeq
and for $-\f12<s\leq \f12$, \beq\label{e20}\begin{aligned} \|\vv
f(Y,&q)\|_{L^1_T(\dH^s)}\lesssim
\|Y\|_{L^\infty_T(\dot{H}^{s+2})}\Bigl(\|\na
q\|_{L^1_T(\dot{B}^{\f12})}+\|
Y_t\|_{L^1_T(\dB^{\f52})}+\|Y_t\|_{L^2_T(\dot{B}^{\f32})}^2\\
&+\|\p_3
Y\|_{L^2_T(\dot{B}^{\f{3}{2}})}^2\Bigr)+\|Y_t\|_{L^2_T(\dot{B}^{\f32})}\|Y_t\|_{L^2_T(\dot{H}^{s+1})}
 +\|\p_3Y\|_{L^2_T(\dot{B}^{\f{3}{2}}_{2,1})}\|\p_3
Y\|_{L^2_T(\dH^{s+1})}.
\end{aligned}\eeq
\end{lem}
\begin{proof} According to \eqref{d22}, we split the estimate of
$\vv f(Y,q)$ into that of $\tilde{\vv f}(Y,q)$ and $\bar{\vv f}(Y).$
\begin{itemize}
\item {\bf Estimates on $\tilde{\vv f}(Y,q)=-\na_Yq$.}\end{itemize}

Thanks to \eqref{B8}, we get, by using product laws in Besov spaces
(\cite{bcd}), that $s>-\f12$, \begin{multline*} \|\tilde{\vv
f}(Y,q)(t)\|_{\dH^s}\lesssim \bigl(1+\|\na
Y(t)\|_{\dot{B}^{\f32}}\bigr)
\bigl(\|\na Y(t)\|_{\dot{B}^{\f32}}\|\na q(t)\|_{\dH^s}\\
+\|\na Y(t)\|_{\dH^{s+1}}\|\na q(t)\|_{\dot{B}^{\f12}}\bigr)+\|\na
q(t)\|_{\dH^s}
\end{multline*}
which along with \eqref{d10} implies that for any $s>-\f12$,
\beq\label{e10}  \|\tilde{\vv f}(Y,q)(t)\|_{\dH^s}\lesssim \|\na
q(t)\|_{\dH^s} +\|Y(t)\|_{\dH^{s+2}}\|\na
q(t)\|_{\dot{B}^{\f12}}. \eeq
\begin{itemize}
\item {\bf Estimates on $\bar{\vv f}(Y)=(\na_Y\cdot\na_Y-\D)Y_t$.}
\end{itemize}

It follows from \eqref{d31} that  \beno
\|\bar{\vv f}(Y)(t)\|_{\dH^s}\lesssim\|[(\cA_Y-I)\cA_Y^T\na
Y_t](t)\|_{\dH^{s+1}}+\|[(\cA_Y^T-I)\na Y_t](t)\|_{\dH^{s+1}}, \eeno
so that by virtue of  \eqref{B8}, we get, by applying product laws
in Besov spaces, that for any $s>-\f{5}{2}$, \beno\begin{aligned}
\|\bar{\vv f}(Y)(t)\|_{\dH^s}\lesssim&\bigl(1+\|\na
Y(t)\|_{\dot{B}^{\f32}}\bigr)^3 \Bigl(\|\na
Y(t)\|_{\dot{B}^{\f32}}\|\na Y_t(t)\|_{\dH^{s+1}}+\|\na
Y(t)\|_{\dH^{s+1}}\|\na Y_t(t)\|_{\dot{B}^{\f32}}\Bigr),
\end{aligned}\eeno
which along with \eqref{d10} implies that for any $s>-\f52$,
\beq\label{e12} \|\bar{\vv f}(Y)\|_{L^2_T(\dH^s)}\lesssim\|\na
Y\|_{L^\infty_T(\dot{B}^{\f32})}\|\na Y_t\|_{L^2_T(\dH^{s+1})}
+\|\na Y\|_{L^\infty_T(\dH^{s+1})}\|\na
Y_t\|_{L^2_T(\dot{B}^{\f32})}. \eeq

On the other hand,  we get by using Bony's decomposition \eqref{C7}
that for $-\f52<s\leq \f12$, \beno\begin{aligned}
\|(\cA_Y^T-I)\na Y_t\|_{L^1_T(\dH^{s+1})}&\lesssim\|\cA_Y^T-I\|_{L^\infty_T(\dH^{s+1})}\|\na Y_t\|_{L^1_T(\dB^{\f32})}\\
&\lesssim\bigl(1+\|\na Y\|_{L^\infty_T(\dot{B}^{\f32})}\bigr)\|\na
Y\|_{L^\infty_T(\dH^{s+1})}\| Y_t\|_{L^1_T(\dB^{\f52})}.
\end{aligned}\eeno
Together with \eqref{d10}, this gives \beno\begin{aligned}
\|(\cA_Y^T-I)\na Y_t\|_{L^1_T(\dH^{s+1})}\lesssim\|\na
Y\|_{L^\infty_T(\dH^{s+1})}\| Y_t\|_{L^1_T(\dB^{\f52})}.
\end{aligned}\eeno
The same estimate holds for term $\|[(\cA_Y-I)\cA_Y^T\na
Y_t](t)\|_{\dH^{s+1}}$. We thus obtain for $-\f52<s\leq\f12$,
\beq\label{e21} \|\bar{\vv f}(Y)\|_{L^1_T(\dH^{s})}\lesssim\|\na
Y\|_{L^\infty_T(\dH^{s+1})}\| Y_t\|_{L^1_T(\dB^{\f52})}. \eeq


By summing up \eqref{e10} and \eqref{e12}, we obtain for
$s>-\f12,$
\begin{multline*}
 \|\vv f(Y,q)\|_{L^2_T(\dH^s)}\lesssim\|\na q
\|_{L^2_T(\dH^s)} +\|Y\|_{L^\infty_T(\dH^{s+2})}\bigl(\|\na q
\|_{L^2_T(\dB^{\f12})}\\+\|Y_t\|_{L^2_T(\dot{B}^{\f52})}\bigr)+\|\na
Y\|_{L^\infty_T(\dot{B}^{\f32})}\|Y_t\|_{L^2_T(\dH^{s+2})},
\end{multline*} which together with \eqref{e4} yields
\eqref{e18}.

 On the other hand,  by summing up  \eqref{e10} and
\eqref{e21},  we get  for $-\f12<s\leq\f12$ \beno \|\vv
f(Y,q)\|_{L^1_T(\dH^s)}\lesssim\|\na q \|_{L^1_T(\dH^s)}
+\|Y\|_{L^\infty_T(\dH^{s+2})}\bigl(\|\na q
\|_{L^1_T(\dB^{\f12})}+\|Y_t\|_{L^1_T(\dot{B}^{\f52})}\bigr), \eeno
from which and \eqref{e4}, we achieve \eqref{e20}. This concludes
the proof of Lemma \ref{L8}.
\end{proof}


\subsection{The proof of Theorem \ref{T}}
The proof of Theorem \ref{T} is  based on the following proposition:
\begin{prop}\label{p8}
{\sl Let  $s_1>\f{5}{4}$ and $s_2\in(-\f{1}{2},-\f{1}{4})$. Let
$(Y,q)$ be a smooth enough solution of \eqref{B12}-\eqref{B13} on
$[0,T].$ We denote \beq\label{e15}\begin{aligned}
&\mathcal{E}_T^{s_1,s_2}(Y,q)\eqdefa
E_T^{s_1}(Y,q)+E_T^{s_2}(Y,q)+\|Y_t\|_{\wt{L}^\infty_T(\cB^{\f12,0}\cap\cB^{s_1,0})}^2
+\|\p_3Y\|_{\wt{L}^\infty_T(\cB^{\f12,0}\cap\cB^{s_1,0})}^2\\
&\qquad\qquad\qquad
+\|Y\|_{\wt{L}^\infty_T(\cB^{\f52,0}\cap\cB^{s_1+2,0})}^2+\|Y_t\|_{L^1_T(\cB^{\f52,0}\cap\cB^{s_1+2,0})}^2
+\|\p_3Y\|_{\wt{L}^2_T(\cB^{\f32,0}\cap\cB^{s_1+1,0})}^2,\\
&\mathcal{E}_0^{s_1,s_2}\eqdefa E^{s_1}_0+E^{s_2}_0+\|Y_1\|_{\cB^{\f12,0}\cap\cB^{s_1,0}}^2
+\|\p_3Y_0\|_{\cB^{\f12,0}\cap\cB^{s_1,0}}^2
+\|Y_0\|_{\cB^{\f52,0}\cap\cB^{s_1+2,0}}^2,
\end{aligned}\eeq where  \beno\begin{aligned}
E_T^s(Y,q)\eqdefa&\|Y_t\|_{\wt{L}^\infty_T(\dot{H}^s\cap\dot{H}^{s+1})}^2
+\|\p_3Y\|_{\wt{L}^\infty_T(\dot{H}^s)}^2+\|Y\|_{\wt{L}^\infty_T(\dot{H}^{s+2})}^2
\\
&+\|Y_t\|_{L^2_T(\dot{H}^{s+1}\cap\dot{H}^{s+2})}^2+\|\p_3Y\|_{L^2_T(\dot{H}^{s+1})}^2
+\|\na q\|_{L^2_T(\dot{H}^s)}^2+\|\na q\|_{L^1_T(\dot{H}^s)}^2,
\end{aligned}\eeno
and  \beno
E_0^s\eqdefa\|Y_1\|_{\dot{H}^s\cap\dot{H}^{s+1}}^2+\|\p_3Y_0\|_{\dot{H}^s}^2+\|Y_0\|_{\dot{H}^{s+2}}^2.
\eeno We assume that \beq\label{e16} \|\na
Y\|_{L^\infty_T(\dot{B}^{\f32})}+\|
Y\|_{L^\infty_T(\cB^{s_1+2,0})}\leq c_0\quad\mbox{and}\quad
\|Y\|_{L^\infty_T(\dB^{s_1+\f32})}\leq 1, \eeq for some $c_0$
sufficiently small,
then there holds \beq\label{e23} \mathcal{E}_T^{s_1,s_2}(Y,q)\leq
C_1\mathcal{E}_0^{s_1,s_2}+C_1\bigl(\mathcal{E}_T^{s_1,s_2}(Y,q)^{1/2}+\mathcal{E}_T^{s_1,s_2}(Y,q)+\mathcal{E}_T^{s_1,s_2}(Y,q)^2\bigr)\mathcal{E}_T^{s_1,s_2}(Y,q),
\eeq for some uniform positive constant $C_1.$}
\end{prop}
\begin{proof}
Under the assumptions \eqref{e16},  for $s=s_1$ and $s=s_2$, we
deduce from Lemmas \ref{p5} and \ref{p6} that \beq\label{e22}
\begin{split}
E_T^s(Y,q)\lesssim & E_0^s+\|Y\|_{L^\infty_T(\dH^{s+2})}^2(\|\na
q\|_{L^1_T(\dot{B}^{\f12})}^2+\|\na q\|_{L^2_T(\dot{B}^{\f12})}^2)
\\
&+\Bigl(\|Y_t\|_{L^\infty_T(\dot B^{\f32})}^2+\|Y_t\|_{L^2_T(\dot
B^{\f32})}^2+\|\p_3 Y\|_{L^\infty_T(\dot B^{\f32})}^2+\|\p_3
Y\|_{L^2_T(\dot B^{\f32})}^2\Bigr)\\
&\times \Bigl(1+\|Y_t\|_{L^2_T(\dot B^{\f32})}^2
+\|\p_3 Y\|_{L^2_T(\dot B^{\f32})}^2\Bigr)E_T^s(Y,q)\\
&+\|\vv f\|_{L_T^2(\dH^s)}\|\D Y_t\|_{L_T^2(\dH^s)} +\|\vv
f\|_{\wt{L}_T^1(\dH^s)}\bigl(\|Y_t\|_{\wt{L}_T^\infty(\dH^s)}+\|\D
Y\|_{\wt{L}_T^\infty(\dH^s)}\bigr).
\end{split}
\eeq However, taking $s=s_1$ in \eqref{e18}  gives rise to \beno
\begin{split}
\|\vv f\|_{L_T^2(\dH^{s_1})}\lesssim &\Bigl(\|\na q\|_{L^2_T(\dot
B^{\f12})}+\|Y_t\|_{L^\infty_T(\dot B^{\f32})}
\bigl(1+\|Y_t\|_{L^2_T(\dot B^{\f32})}\bigr)+\|\na Y\|_{L^\infty_T(\dot B^{\f32})}\\
&+\|\p_3Y\|_{L^\infty_T(\dot B^{\f32})}
\bigl(1+\|\p_3Y\|_{L^2_T(\dot B^{\f32})}\bigr)+\| Y_t\|_{L^2_T(\dot
B^{\f52})}\Bigr)\bigl(E_T^{s_1}(Y,q)\bigr)^{\f12},
\end{split}
\eeno and applying Proposition \ref{p2} and Proposition \ref{p3}
leads to \beno\begin{aligned} \|\vv
f\|_{\wt{L}_T^1(\dH^{s_1})}\lesssim& \|\vv
f\|_{L_T^1(\dH^{s_1})}\lesssim\|\vv f\|_{L_T^1(\cB^{s_1,0})}\\
\lesssim & \|Y_t\|_{L^2_T(\dot{B}^{\f{3}{4}})}^2 +\|\p_3
Y\|_{L_T^2(\dot{B}^{\f{3}{4}})}^2+\|
Y_t\|_{L^2_T(\dot{B}^{\f{5}{2}})}^2+\|Y_t\|_{L^2_T(\dot{B}^{s_1+\f{5}{4}})}^2+
\|\p_3Y\|_{L_T^2(\dot{B}^{\f{3}{2}})}^2\\
&+\|\p_3 Y\|_{L_T^2(\dot{B}^{s_1+\f{1}{4}})}^2
+\|Y\|_{L^\infty_T(\dB^{\f{5}{2}}\cap\cB^{s_1+2,0})}^2+
\|Y_t\|_{L^1_T(\dB^{\f{5}{2}}\cap\cB^{s_1+2,0})}^2,
\end{aligned}\eeno
where we used the fact that $\|Y_t\|_{L^1_T(\dB^{\f52})}\lesssim\|Y_t\|_{L^1_T(\cB^{\f52,0})}$.
While as  $s_1>\f{5}{4}$ and $s_2\in(-\f12,-\f14),$ one has
\begin{multline*}\|Y_t\|_{L^2_T(\dot{B}^{\f{3}{4}})}+\|\p_3
Y\|_{L_T^2(\dot{B}^{\f{3}{4}})}+\|
Y_t\|_{L^2_T(\dot{B}^{\f{5}{2}})}+\|Y_t\|_{L^2_T(\dot{B}^{s_1+\f{5}{4}})}
+ \|\p_3Y\|_{L_T^2(\dot{B}^{\f{3}{2}})}\\ +\|\p_3
Y\|_{L_T^2(\dot{B}^{s_1+\f{1}{4}})}\lesssim \|\p_3
Y\|_{L_T^2(\dH^{s_2+1}\cap\dH^{s_1+1})}+\|Y_t\|_{L_T^2(\dH^{s_2+1}\cap\dH^{s_1+2})}.
\end{multline*}
 Along the same line, we deduce from  Corollary \ref{col2} and its
proof that \beno \begin{split} \|\na
q\|_{L^1_T(\dot{B}^{\f12})}+\|\na q\|_{L^2_T(\dot{B}^{\f12})}
\lesssim&
\|Y_t\|_{L^2_T(\dB^{\f32})}\bigl(\|Y_t\|_{L^\infty_T(\dB^{\f32})}+\|Y_t\|_{L^2_T(\dB^{\f32})}\bigr)\\
&+
\|\p_3Y\|_{L^2_T(\dB^{\f32})}\bigl(\|\p_3Y\|_{L^\infty_T(\dB^{\f32})}+\|\p_3Y\|_{L^2_T(\dB^{\f32})}\bigr)
\lesssim
\cE^{s_1,s_2}_T(Y,q).
\end{split}
\eeno As a consequence, we obtain \beno \|\vv
f\|_{L^2_T(\dH^{s_1})}+\|\vv f\|_{\wt{L}^1_T(\dH^{s_1})}\lesssim
\Bigl(\cE^{s_1,s_2}_T(Y,q)+\cE^{s_1,s_2}_T(Y,q)^{\f12}\Bigr)E^{s_1}_T(Y,q)^{\f12}+\cE^{s_1,s_2}_T(Y,q).
\eeno Resuming the above estimates into \eqref{e22} yields
\beq\label{e22ad}\begin{aligned} E^{s_1}_T(Y,q)\lesssim&
E_0^{s_1}+\bigl(\cE^{s_1,s_2}_T(Y,q)^{\f12}+\cE^{s_1,s_2}_T(Y,q)+\cE^{s_1,s_2}_T(Y,q)^2\bigr)E^{s_1}_T(Y,q)\\
&\quad+\cE^{s_1,s_2}_T(Y,q)E^{s_1}_T(Y,q)^{\f12}.
\end{aligned}\eeq

On the other hand, it follows from Lemma \ref{L8} and the fact that $\|Y_t\|_{L^1_T(\dB^{\f52})}\lesssim\|Y_t\|_{L^1_T(\cB^{\f52,0})}$,
that
\beno
\|\vv f\|_{L_T^2(\dH^{s_2})}+\|\vv
f\|_{\wt{L}_T^1(\dH^{s_2})}\lesssim\bigl(\cE^{s_1,s_2}_T(Y,q)^{\f12}+\cE^{s_1,s_2}_T(Y,q)\bigr)E^{s_1}_T(Y,q)^{\f12}
\eeno from which, we infer, by a similar derivation of
\eqref{e22ad}, that \beq\label{e22af} E^{s_2}_T(Y,q)\lesssim
E_0^{s_2}+\bigl(\cE^{s_1,s_2}_T(Y,q)^{\f12}+\cE^{s_1,s_2}_T(Y,q)+\cE^{s_1,s_2}_T(Y,q)^2\bigr)E^{s_2}_T(Y,q).
\eeq

Therefore, by summing up \eqref{e22ad}, \eqref{e22af}  and
\eqref{d29} with $s=s_1$, we achieve \eqref{e23}. This completes the
proof of Proposition \ref{p8}.
\end{proof}

Now we are in a position to complete the proof of Theorem \ref{T}.

\begin{proof}[Proof of Theorem \ref{T}] Given initial data  $(Y_0, Y_1)$ satisfying the assumptions
listed in Theorem \ref{T}, we deduce by Proposition \ref{p8} and  a
standard argument that \eqref{B12}-\eqref{B13} has a unique solution
$(Y,q)$ on $[0,T],$ which satisfies \eqref{B16} on $[0,T]$. Let
$T^*$ be the largest possible time so that \eqref{B16} holds. Then
to complete the proof of Theorem \ref{T}, we only need to show that
$T^\ast=\infty$ and there holds \eqref{B17} under the assumptions of
\eqref{B14} and \eqref{B15}. Otherwise, if $T^\ast<\infty,$  we
denote
 \beq\label{e24} \begin{split}
&\bar{T}\eqdefa \max\bigl\{\ T<T^\ast:\ \
\mathcal{E}_T^{s_1,s_2}(Y,q)\leq \eta_0^2\ \ \ \bigr\},
\end{split}
\eeq for $\eta_0$ so small that $C_1(\eta_0+\eta_0^2+\eta_0^4)\leq\f12,$ and
\beq\label{e25}
\begin{split}
&\|\na Y\|_{L^\infty_T(\dot{B}^{\f32})}+\|
Y\|_{L^\infty_T(\cB^{s_1+2,0})}+\|
Y\|_{L^\infty_T(\dot{B}^{s_1+\f32})}
\leq C_2\mathcal{E}_T^{s_1,s_2}(Y,q)^{\f12}\leq C_2\eta_0\leq 1,
\end{split}
\eeq for the same $C_1$  as that in \eqref{e23}.

 We shall prove that $\bar{T}=\infty$ provided that
$\e_0$ is sufficiently small in \eqref{B15}.  In fact, thanks to
\eqref{e25}, we get by applying Proposition \ref{p8} that
 \beq\label{e26}
\mathcal{E}_{\bar T}^{s_1,s_2}(Y,q)\leq 2C_1\mathcal{E}_0^{s_1,s_2}.
\eeq In particular, if we take $\e_0$ so small that
$2C_1\e_0^2\leq\f12\eta_0^2,$ \eqref{e26} contradicts with
\eqref{e24} if $\bar{T}<\infty.$ This in turn shows that
$\bar{T}=T^\ast=\infty,$ and there holds \eqref{B17}. This completes
the proof of Theorem \ref{T}.
\end{proof}

\section{The proof of Theorem \ref{th1} and Theorem \ref{th2}}\label{sect6}
\setcounter{equation}{0}

With Theorem \ref{T} and  Lemma \ref{funct} in the Appendix
\ref{appendc} in hand, we can now present the proof of Theorem
\ref{th1}.

\begin{proof}[Proof of Theorem \ref{th1}] Under  the assumptions of Theorem \ref{th1}, we
get, by applying Lemma \ref{LL2}, that there exists a vector-valued
function $Y_0(y)=(Y_0^1(y),Y_0^2(y),Y_0^3(y))^T$ so that \beno
X_0(y)= I+Y_0(y)\quad\mbox{and}\quad
 U_0\circ
X_0(y)=\na_yX_0(y)=I+\na_yY_0(y), \eeno which in particular implies
\beq\label{f0} \f{\p X_0^{-1}(x)}{\p x}=(I+\na_y Y_0)^{-1}\circ
X_0^{-1}(x)=U_0^{-1}(x)=I-\na_x\vv\Psi. \eeq  Let $Y_1(y)\eqdefa\vv
u_0(X_0(y)).$ Applying Lemma \ref{funct} with $\Phi(x)=X^{-1}_0(x)$
gives
 \beno
\begin{split}
\|Y_1\|_{\dot{H}^{s_1+1}}+\|Y_1\|_{\dot{H}^{s_2}}\leq &
C\bigl(\|\na\vv\Psi\|_{\dot{B}^{\f32}}\bigr)\Bigl(\|\vv
u_0\|_{\dH^{s_2}}+\|\vv
u_0\|_{\dH^{s_1+1}}+\|\na\vv\Psi\|_{\dH^{s_1+\f12}}\|\vv
u_0\|_{\dH^{2}}\\
&+(1+\|\D\vv\Psi\|_{H^{s_1-1}})\|\na\vv
u_0\|_{H^{s_1}}\Bigr)\\
\leq &C\bigl(\|\na\vv\Psi\|_{\dot{B}^{\f3p}_{p,1}}\bigr)\Bigl(\|\vv
u_0\|_{\dH^{s_2}}+\bigl(1+\|\na\vv\Psi\|_{B^{s_1+\f3p-\f32}_{p,2}}\bigr)\|\na\vv
u_0\|_{B^{s_1+\f3p-\f32}_{p,2}}\Bigr),
\end{split}
 \eeno where in the last step, we used Lemma \ref{L2} so that\beno
 \|\na\vv u_0\|_{H^{s_1}}\lesssim \|\na\vv
 u_0\|_{B^{s_1+\f3p-\f32}_{p,1}}\quad\mbox{for}\quad p\in (1,2).
 \eeno
Similarly as $\|\p_3Y_0\|_{\cB^{\f12,0}}\lesssim
\|Y_0\|_{\cB^{\f12,1}}\lesssim \|Y_0\|_{\dB^{\f32}},$ and by virtue
of \eqref{app3ad},  $Y_0=\vv \Psi(X_0(y)),$ one has
\begin{multline*}
\|Y_0\|_{\dH^{s_2+2}\cap\dH^{s_1+2}}+\|\p_3Y_0\|_{\dH^{s_2}}+\|\p_3Y_0\|_{\cB^{\f12,0}}\\
\leq C\bigl(\|\na\vv\Psi\|_{\dot{B}^{\f3p}_{p,1}}\bigr)\Bigl(\|\vv
\Psi\|_{\dB^{s_2+\f3p-\f12}_{p,2}}+\bigl(1+\|\D\vv\Psi\|_{B^{s_1+\f3p-\f32}_{p,2}}\bigr)\|\na\vv
\Psi\|_{B^{s_1+\f3p-\f12}_{p,2}}\Bigr),
\end{multline*}
Whereas as $p<2,$ it follows from Definition \ref{def2} and Lemma
\ref{L2} that \beq\label{f0pq}
\begin{split}
\|a\|_{\cB^{s,0}}=\sum_{j,k\in\Z^2}2^{js}\|\D_j\D_k^v
a\|_{L^2}\lesssim&
\sum_{j\in\Z}2^{j(s+\f2p-1)}\|\D_ja\|_{L^p}\Bigl(\sum_{k\leq j+N_0}2^{k(\f1p-\f12)}\Bigr)\\
\lesssim & \sum_{j\in\Z}2^{j(s+\f3p-\f32)}\|\D_ja\|_{L^p}\lesssim
\|a\|_{\dB^{s+\f3p-\f32}_{p,1}}.
\end{split}
\eeq Thanks to \eqref{f0pq}, we get, by applying  Lemma \ref{funct},
that \beno
\begin{split}
\|Y_0\|_{\cB^{\f52,0}\cap\cB^{s_1+2,0}}\lesssim&
\|Y_0\|_{\dB^{1+\f3p}_{p,1}\cap\dB^{s_1+\f3p+\f12}_{p,1}}\\
\leq & C\bigl(\|\na\vv\Psi\|_{\dot{B}^{\f3p}_{p,1}}\bigr)\Bigl(\|\vv
\Psi\|_{\dB^{1+\f3p}_{p,1}}+\bigl(1+\|\D\vv\Psi\|_{B^{s_1+\f3p-\f32}_{p,1}}\bigr)\|\na\vv
\Psi\|_{B^{s_1+\f3p-\f12}_{p,1}}\Bigr),
\end{split} \eeno and
\beno
 \begin{split}
 \|Y_1\|_{\cB^{\f12,0}\cap\cB^{s_1,0}}\lesssim& \|Y_1\|_{\dB^{\f3p-1}_{p,1}\cap\dB^{s_1+\f3p-\f32}_{p,1}}\\
 \leq &C\bigl(\|\na\vv\Psi\|_{\dot{B}^{\f3p}_{p,1}}\bigr)\Bigl(\|\vv
u_0\|_{\dB^{\f3p-1}_{p,1}}+\bigl(1+\|\na\vv\Psi\|_{B^{s_1+\f3p-\f52}_{p,1}}\bigr)\|\na\vv
u_0\|_{B^{s_1+\f3p-\f52}_{p,1}}\Bigr).
\end{split} \eeno Therefore,
thanks to \eqref{B2}, we conclude that \beq\label{f1}\begin{aligned}
&\|Y_0\|_{\dot{H}^{s_1+2}\cap\dot{H}^{s_2+2}}+\|Y_0\|_{\cB^{s_1+2,0}\cap\cB^{\f52,0}}
+\|\p_3Y_0\|_{\dot{H}^{s_2}}+\|\p_3Y_0\|_{\cB^{s_1,0}\cap\cB^{\f12,0}}\\
&\qquad+\|Y_1\|_{\dot{H}^{s_1+1}}+\|Y_1\|_{\dot{H}^{s_2}}+ \|Y_1\|_{\cB^{s_1,0}\cap\cB^{\f12,0}}\\
&\leq C\bigl(\|\na\vv\Psi\|_{\dB^{s_2+\f3p-\f32}_{p,2}\cap
B^{s_1+\f3p-\f12}_{p,1}}+\|\vv
u_0\|_{\dH^{s_2}\cap\dB^{\f3p-1}_{p,1}}+\|\na\vv
u_0\|_{B^{s_1+\f3p-\f32}_{p,1}}\bigr)\leq C\e_0,
\end{aligned}\eeq
from which,  and Theorem \ref{T}, we deduce that the system
\eqref{B11} (equivalently \eqref{B12}-\eqref{B13}) has a unique
global solution $(Y,q)$ which satisfies \eqref{B16} and \eqref{B17}
provided that $\e_0$ in \eqref{B2} is sufficiently small.

We denote $X(t,y)\eqdefa y+Y(t,y).$  Then it follows from
\eqref{B17} that $X(t,y)$ is invertible with respect to $y$
variables and  we denote its inverse mapping by $X^{-1}(t,x).$ Since
$\det(1+\na Y)=1$, the adjoint matrix $\cA_Y$ of $I+\na Y$ satisfies
\beno \na\cdot\cA_Y=\vv 0\quad\text{and}\quad \cA_Y=(I+\na Y)^{-1}
\eeno which implies \beq\label{f2} \na_x\cdot[(I+\na Y)\circ
X^{-1}]=\bigl(\na\cdot[\cA_Y(I+\na Y)]\bigr)\circ X^{-1}=\vv 0. \eeq
Then we  define $U(t,x)=\bigl(\bar{\vv b}(t,x), \tilde{\vv b}(t,x),
\vv b(t,x)\bigr)$ and $(\vv u(t,x), p(t,x))$  through \beq\label{f3}
\begin{split}
&U(t,x)=(\bar{\vv b}, \tilde{\vv b}, \vv b)(t,x)\eqdefa (I+\na Y)(t,
X^{-1}(t,x))\quad \mbox{and}\\
&\vv u(t,x)\eqdefa Y_t(t, X^{-1}(t,x)), \quad p(t,x)\eqdefa
q(t,X^{-1}(t,x))-\f12|\vv b(t,x)|^2,  \end{split}\eeq from which and
\eqref{f2}, we infer that \beno \dv\bar{\vv b}=\dv\tilde{\vv
b}=\dv\vv b=0. \eeno Hence according to Section \ref{sect2}, $(U,
\vv u, p)$ thus defined globally solves \eqref{B4}. Then to complete
the proof of Theorem \ref{th1}, it amounts to prove \eqref{th1wr}.
For this, we first notice from \eqref{f3} that \beno (\na\vv u)\circ
X(t,y)=\na_yY_t(t,y) \bigl(I+\na_yY(t,y)\bigr)^{-1}, \eeno which
along with the proof of \eqref{poit} in the Appendix \ref{appendc}
implies \beq\label{f5}
\begin{aligned} \|\vv u\|_{L^1(\R^+;\dB^{\f52})}
&\lesssim(1+\|\na_yY\|_{L^\infty(\R^+;\dB^{\f32})})^2\|\na_y
Y_t\|_{L^1(\R^+;\dB^{\f32})}\\
&\leq C(\|\na Y\|_{L^\infty(\R^+;\dot{B}^{\f32})})\|
Y_t\|_{L^1(\R^+;\cB^{\f52,0})}.
\end{aligned}\eeq
Again thanks to \eqref{f3}, we get, by applying Lemma \ref{funct}
with $\Phi=X(t,y),$ that \beq\label{f6}\begin{aligned}
&\|(\bar{\vv b}-\vv e_1, \tilde{\vv b}-\vv e_2)\|_{L^\infty(\R^+;\dH^{s_2+1})}+\|\vv b-\vv e_3\|_{L^\infty(\R^+;\dH^{s_2})}\\
&\quad+\|\vv u\|_{L^\infty(\R^+; \dH^{s_2})}+\|\vv b-\vv e_3\|_{L^2(\R^+;\dH^{s_2+1})}+\|\vv u\|_{L^2(\R^+; \dH^{s_2+1})}\\
&\leq C(\|\na Y\|_{L^\infty(\R^+;\dot{B}^{\f32})})\Bigl(\|(\p_1 Y,\p_2 Y)\|_{L^\infty(\R^+;\dH^{s_2+1})}+\|\p_3 Y\|_{L^\infty(\R^+;\dH^{s_2})}\\
&\quad+\|Y_t\|_{L^\infty(\R^+;\dH^{s_2})}+\|\p_3
Y\|_{L^2(\R^+;\dot{H}^{s_2+1})}+\|Y_t\|_{L^2(\R^+; \dH^{s_2+1})}
\Bigr).
\end{aligned}\eeq
Along the same line, one has \beq\label{f7}\begin{aligned} \|\vv
u\|_{L^1(\R^+;\dH^{s_1+2})} \leq C(\|\na
Y\|_{L^\infty(\R^+;\dot{B}^{\f32})})(1+\|\D
Y\|_{L^\infty(\R^+;H^{s_1})})\|\na Y_t\|_{L^1(\R^+;H^{s_1+1})},
\end{aligned}\eeq
and
\beq\label{f8}\begin{aligned}
&\|(\bar{\vv b}-\vv e_1, \tilde{\vv b}-\vv e_2)\|_{L^\infty(\R^+;\dH^{s_1+1})}+\|\vv b-\vv e_3\|_{L^\infty(\R^+;\dH^{s_1+1})}\\
&\quad+\|\vv u\|_{L^\infty(\R^+; \dH^{s_1+1})}+\|\vv b-\vv e_3\|_{L^2(\R^+;\dH^{s_1+1})}+\|\vv u\|_{L^2(\R^+; \dH^{s_1+2})}\\
&\leq C(\|\na Y\|_{L^\infty(\R^+;\dot{B}^{\f32})})\Bigl(\|\na
Y\|_{L^\infty(\R^+;\dH^{s_1+1})}+\|Y_t\|_{L^\infty(\R^+;\dH^{s_1+1})}
\\
&\quad+\|\p_3Y\|_{L^2(\R^+;\dH^{s_1+1})}+\|Y\|_{L^\infty(\R^+;\dH^{s_1+\f32})}\bigl(\|\na
Y\|_{L^\infty(\R^+; \dH^{2})}+\|Y_t\|_{L^\infty(\R^+;
\dH^{2})}\\
&\quad+\|\p_3 Y\|_{L^2(\R^+;
\dH^{2})}\bigr)+(1+\|\D Y\|_{L^\infty(\R^+;H^{s_1})})\bigl(\|\D Y\|_{L^\infty(\R^+; H^{s_1})}\\
&\quad+\|\na Y_t\|_{L^\infty(\R^+;H^{s_1})}+\|\p_3\na
Y\|_{L^2(\R^+; H^{s_1})}+\|\na Y_t\|_{L^2(\R^+;H^{s_1+1})}\bigr)
\Bigr).
\end{aligned}\eeq

Consequently, we deduce from \eqref{B17},  \eqref{C5}, \eqref{f1}, \eqref{f5} to \eqref{f8}
and the fact that $\|u\|_{\dH^s}\lesssim\|u\|_{\cB^{s,0}}$ that{\bf
\beq\label{f9}\begin{aligned}
&\|(\bar{\vv b}-\vv e_1, \tilde{\vv b}-\vv e_2)\|_{L^\infty(\R^+;\dH^{s_1+1}\cap\dH^{s_2+1})}+\|\vv b-\vv e_3\|_{L^\infty(\R^+;\dH^{s_1+1}\cap\dH^{s_2})}\\
&\quad+\|\vv u\|_{L^\infty(\R^+;\dH^{s_1+1}\cap\dH^{s_2})}+\|\vv b-\vv e_3\|_{L^2(\R^+;\dH^{s_1+1}\cap\dH^{s_2+1})}
\\
&\quad+\|\vv u\|_{L^2(\R^+; \dH^{s_1+2}\cap\dH^{s_2+1})}+\|\vv u\|_{L^1(\R^+; \dH^{s_1+2}\cap\dB^{\f52})}\\
&\leq C\bigl(\|\na\vv\Psi\|_{\dB^{s_2+\f3p-\f32}_{p,2}\cap
B^{s_1+\f3p-\f12}_{p,1}}+\|\vv
u_0\|_{\dH^{s_2}\cap\dB^{\f3p-1}_{p,1}}+\|\na\vv
u_0\|_{B^{s_1+\f3p-\f32}_{p,1}}\bigr),
\end{aligned}\eeq}
provided that $\e_0$ is sufficiently small in \eqref{B2}.

On the other hand, taking space divergence to the momentum equations
of \eqref{B1} gives rise to \beq\label{lpqw} \na p=-\f12\na(|\vv
b|^2)+\na(-\D)^{-1}\dv(\vv u\cdot\na\vv u-\vv b\cdot\na\vv b), \eeq
 then applying product laws in Sobolev spaces
gives rise to \beno
\begin{split}
\|&\na p\|_{L^2(\R^+;\dot{H}^{s_1}\cap\dot{H}^{s_2})}\\
&\leq C \bigl(\|\vv u\|_{L^\infty(\R^+;\dB^{\f32})}\|\na\vv
u\|_{L^2(\R^+;\dot{H}^{s_1}\cap\dot{H}^{s_2})}+\|\vv
u\|_{L^\infty(\R^+;\dot{H}^{s_1}\cap\dot{H}^{s_2})}\|\na \vv
u\|_{L^2(\R^+;\dB^{\f32})}\\
&\quad+(1+\|\vv b-\vv e_3\|_{L^\infty(\R^+;\dot{B}^{\f32})})
\|\na\vv b\|_{L^2(\R^+;\dot{H}^{s_1}\cap\dot{H}^{s_2})}+\|\vv b-\vv
e_3\|_{L^\infty(\R^+;\dot{H}^{s_1+1})}\|\na\vv
b\|_{L^2(\R^+;\dot{B}^{\f12})}\bigr),
\end{split}
\eeno which together with \eqref{f9} ensures that
 \beno
\|\na p\|_{L^2(\R^+;\dH^{s_1}\cap\dot{H}^{s_2})}\leq
C\bigl(\|\na\vv\Psi\|_{\dB^{s_2+\f3p-\f32}_{p,2}\cap
B^{s_1+\f3p-\f12}_{p,1}}+\|\vv
u_0\|_{\dH^{s_2}\cap\dB^{\f3p-1}_{p,1}}+\|\na\vv
u_0\|_{B^{s_1+\f3p-\f32}_{p,1}}\bigr), \eeno provided that $\e_0$ is
sufficiently small in \eqref{B2}. This completes the proof of
\eqref{th1wr} and thus Theorem \ref{th1}.
\end{proof}

Before we present the proof of Theorem \ref{th2}, we shall first
prove the following blow-up criterion for smooth enough solutions of
\eqref{B1}.

\begin{prop}\label{p9}
{\sl Let  $\vv b_0-\vv e_3\in H^{s}(\R^3)$ and $\vv u_0\in
H^s(\R^3)$ for $s>\f32,$ \eqref{B1} has a unique solution $(\vv b,
\vv u)$ on $[0,T]$ for  some $T>0$ so that $\vv b-\vv e_3\in
C([0,T]; H^{s}(\R^3)),$ $ \vv u\in  C([0,T]; H^{s}(\R^3))$ {\bf with
$\na\vv u\in L^2((0,T);\dot{H}^{s+1}(\R^3))$ }and $ \na p\in C([0,T];
H^{s-1}(\R^3)).$ Moreover, if $T^\ast$ is the life span to this
solution, and $T^\ast<\infty,$ one has  \beq
\int_0^{T^\ast}\bigl(\|\na\vv u(t)\|_{L^\infty}+\|\vv
b(t)\|_{L^\infty}^2\bigr)\,dt=\infty. \label{qs2}\eeq}
\end{prop}
\begin{proof} It is well-known that the existence of solution to a
nonlinear PDE basically follows from the uniform estimates to some
smooth enough approximate solutions. For simplicity, we may only
present {\it a priori} estimates to smooth enough solutions of
\eqref{B1} (one may check \cite{Majda} for the detailed proof to the
related system). As a matter of fact, let $\vec{b}\eqdefa\vv b-\vv
e_3,$ we first get, by using a standard energy estimate for
\eqref{B1}, that \beq\label{qs3} \f12\f{d}{dt}\bigl(\|\vec{
b}\|_{L^2}^2+\|\vv u\|_{L^2}^2\bigr)+\|\na\vv u\|_{L^2}^2=0.  \eeq
Along the same line, applying $\D_j$ to the system \eqref{B1} and
then taking $L^2$ inner product of the resulting equations with
$(\D_j\vec{b},\D_j\vv u),$ we obtain \beq\label{qs4}\begin{split}
\f12\f{d}{dt}&(\|\D_j\vec b\|_{L^2}^2+\|\D_j\vv u\|_{L^2}^2)+\|\na\D_j\vv u\|_{L^2}^2\\
=&-\bigl(\D_j(\vv u\cdot\na\vec b)\ |\ \D_j\vec b\bigr)+\bigl(\D_j(\vec b\cdot\na\vv u)\ |\ \D_j\vec b\bigr)\\
& -\bigl(\D_j(\vv u\cdot\na\vv u)\ |\ \D_j\vv
u\bigr)+\bigl(\D_j(\vec b\cdot\na\vec b)\ |\ \D_j\vv u\bigr).
\end{split}
\eeq By virtue of the commutator estimates (see Section 2.10 of
\cite{bcd}), we write \beno
\begin{split}
&\bigl|\bigl(\D_j(\vv u\cdot\na \vec b)\ |\ \D_j\vec
b\bigr)\bigr|\lesssim c_j(t)^22^{-2js}\bigl(\|\na\vv
u(t)\|_{L^\infty}\|\vec b(t)\|_{\dH^s}^2+\|\vec
b(t)\|_{L^\infty}\|\na\vv u(t)\|_{\dH^s}\|\vec b(t)\|_{\dH^s}\bigr),\\
&\bigl|\bigl(\D_j(\vv u\cdot\na \vv u)\ |\ \D_j\vv
u\bigr)\bigr|\lesssim c_j(t)^22^{-2js}\|\na\vv
u(t)\|_{L^\infty}\|\vv u(t)\|_{\dH^s}^2\quad\mbox{for any}\quad s>0.
\end{split}
\eeno Whereas it follows from product laws in Sobolev spaces that
\beno \bigl|\bigl(\D_j(\vec b\cdot\na \vv u)\ |\ \D_j\vec
b\bigr)\bigr|\lesssim c_j(t)^22^{-2js}\bigl(\|\vec
b(t)\|_{L^\infty}\|\na\vv u(t)\|_{\dH^s}+\|\na \vv
u(t)\|_{L^\infty}\|\vec b(t)\|_{\dH^s}\bigr)\|\vec
b(t)\|_{\dH^s},\eeno and \beno \bigl|\bigl(\D_j(\vec b\cdot\na \vec
b)\ |\ \D_j\vv u\bigr)\bigr|=\bigl|\bigl(\D_j(\vec b\otimes \vec b)\
|\ \D_j\na\vv u\bigr)\bigr|\lesssim c_j(t)^22^{-2js}\|\vec
b(t)\|_{L^\infty}\|\vec b(t)\|_{\dH^s}\|\na \vv u(t)\|_{\dH^s}.\eeno

Resuming the above estimates into \eqref{qs4} and using \eqref{qs3},
we conclude that for any $s>0$, \beq\label{qs6} \begin{split} \|\vv
u(t)\|_{H^s}^2+\|\vec b(t)&\|_{H^s}^2+\|\na\vv
u\|_{L^2_t(H^s)}^2\leq\|\vv u_0\|_{\dH^s}^2+\|\vec
b_0\|_{\dH^s}^2\\
&+C\int_0^t\bigl(\|\na\vv u(t')\|_{L^\infty}+\|\vec
b(t')\|_{L^\infty}^2\bigr)\bigl(\|\vv u(t')\|_{\dH^s}^2+\|\vec
b(t')\|_{\dH^s}^2\bigr)\,dt'.
\end{split}\eeq
Notice that $s>\frac32,$ one has, $\|\na\vv
u(t)\|_{L^\infty}\lesssim \|\na\vv u(t)\|_{H^s},$ we thus achieve
\begin{multline*}
\|\vv u(t)\|_{H^s}^2+\|\vec b(t)\|_{H^s}^2+\|\na\vv u\|_{L^2_t(H^s)}^2
\leq\|\vv u_0\|_{\dH^s}^2+\|\vec
b_0\|_{\dH^s}^2\\+C\int_0^t\bigl(\|\vv
u(t')\|_{H^s}^2+\|\vec b(t')\|_{H^s}^2\bigr)\bigl(\|\vv
u(t')\|_{\dH^s}^2+\|\vec b(t')\|_{\dH^s}^2\bigr)d{t'},
\end{multline*}
from which, we infer that there exists a positive time $T^\ast,$ so
that there holds \beq\label{qs7} \|\vv
u\|_{L^\infty_T(H^s)}^2+\|\vec b(t)\|_{L^\infty_T(H^s)}^2+\|\na\vv
u\|_{L^2_T(H^s)}^2\leq C_T(\|\vv u_0\|_{\dH^s}^2+\|\vec
b_0\|_{\dH^s}^2)\quad\mbox{for any}\quad T<T^\ast.\eeq  which along
with  \eqref{lpqw} ensures that $\na p\in C([0,T]; H^{s-1}(\R^3))$
for any $T<T^\ast.$ This concludes the existence part of Proposition
\ref{p9}.

Finally, applying Gronwall's inequality to \eqref{qs6} yields
\begin{multline*}
\|\vv u(t)\|_{H^s}^2+\|\vec b(t)\|_{H^s}^2+\|\na\vv u\|_{L^2_t(H^s)}^2\\
\leq(\|\vv u_0\|_{H^s}^2+\|\vec
b_0\|_{H^s}^2)\exp\Bigl(C\int_0^t\bigl(\|\na\vv
u(t')\|_{L^\infty}+\|\vec b(t')\|_{L^\infty}^2\bigr)\,d{t'}\Bigr)
\quad\text{for}\quad t<T^*,
\end{multline*}\
which together with a classical continuous argument implies
\eqref{qs2}. This completes the proof of Proposition \ref{p9}.
\end{proof}

Now we are in a position to complete the proof of Theorem \ref{th2}.

\begin{proof}[Proof of Theorem \ref{th2}] Under the assumption of
Theorem \ref{th2}, we deduce from Proposition \ref{LL1} that there
exists a $\vv \Psi=(\psi_1,\psi_2,\psi_3)^T$ so that there holds
\eqref{LL1pq}-\eqref{B3p3}. Notice that for $s_2\in (-\f12,-\f14)$
and $p\in (1,2),$ $s_2+\f3p-\f12>0,$ then it is easy to observe that
\beno \|\na\vv\Psi\|_{\dB^{s_2+\f3p-\f32}_{p,2}\cap
B^{s_1+\f3p-\f12}_{p,1}}\lesssim \|\vv\Psi\|_{
B^{s_1+\f3p+\f12}_{p,1}}\lesssim\|\vv b_0-\vv e_3\|_{
B^{s_1+\f3p+\f12}_{p,1}}. \eeno  Therefore, under the assumption of
\eqref{A1}, for $U_0=(I-\na\vv\Psi)^{-1},$ we infer from Theorem
\ref{th1} that  \eqref{B4} has a unique global solution $(U,\vv
u,p)$ so that there holds \eqref{th1wq} and \eqref{th1wr}. Let
$U=(\bar{\vv b}, \tilde{\vv b},\vv b),$ then according to the
discussions at the beginning of Section \ref{sect2}, $(\vv b, \vv u,
p)$ thus obtained solves \eqref{B1}, which is in fact the unique
solution of \eqref{B1} with initial data $(\vv b_0, \vv u_0),$ and
there holds \eqref{th2wr}.

On the other hand, by virtue of Proposition \ref{p9}, given initial
data $(\vv b_0, \vv u_0)$ with $\vv b_0-\vv e_3\in H^{s}(\R^3)$ and
$\vv u_0\in H^{s}(\R^3)$ (since $\vv u_0\in\dH^{s_2}(\R^3)$ and
$\na\vv u_0\in  H^{s-1}(\R^3)$) for $s\geq s_1+2,$ \eqref{B1} has a
unique solution $(\vv b, \vv u, p)$ with $\vv b-\vv e_3\in
C([0,T];H^{s}(\R^3)),$ $\na p\in C([0,T]; H^{s-1}(\R^3)),$ and $\vv
u\in C([0,T]; H^{s}(\R^3))$ {\bf with $\na\vv u\in L^2((0,T); \dH^{s+1}(\R^3))$} for any
fixed $T<T^\ast.$  Furthermore, if $T^\ast<\infty,$ there holds
\eqref{qs2}.  Due to the uniqueness, this solution must coincide
with the one obtained in the last paragraph. Then thanks to
\eqref{th2wr}, \eqref{qs2} can not be true for any finite $T^\ast,$
and therefore $T^\ast=\infty$ and there holds \eqref{th2wq}. This
completes the proof of Theorem \ref{th2}.

\end{proof}

\appendix

\setcounter{equation}{0}
\section{The Besov estimates to  functions composed with a measure preserving diffeomorphism }\label{appendc}

\begin{lem}\label{funct}
{\sl Let $\Phi(y)=y+\Psi(y)$ be a smooth volume preserving
diffeomorphism on $\R^3.$
 Then  for $u, v\in\cS(\R^3),$ there hold
\beq\label{apb1}\begin{split}
 &\|u\circ\Phi\|_{\dot{B}^s_{p,r}}\leq
C\bigl(\|\na\Psi\|_{\dot{B}^{\f{3}{p}}_{p,1}}\bigr)\|u\|_{\dot{B}^s_{p,r}}\quad\mbox{and} \\
&\|v\circ\Phi^{-1}\|_{\dot{B}^s_{p,r}}\leq
C\bigl(\|\na\Psi\|_{\dot{B}^{\f{3}{p}}_{p,1}}\bigr)\|v\|_{\dot{B}^s_{p,r}}\quad
\mbox{
for}\ \ s\in (-1,2],\\
&\|u\circ\Phi\|_{\dot{B}^s_{p,r}}\leq
C\bigl(\|\na\Psi\|_{\dot{B}^{\f{3}{p}}_{p,1}}\bigr)\bigl(\|u\|_{\dot{B}^s_{p,r}}+\|\Psi\|_{\dot{B}^{s+\f12}_{p,r}}\|u\|_{\dot{B}^2_{p,r}}\bigr)\quad\text{and}
\\
&\|v\circ\Phi^{-1}\|_{\dot{B}^s_{p,r}}\leq
C\bigl(\|\na\Psi\|_{\dot{B}^{\f{3}{p}}_{p,1}}\bigr)\bigl(\|v\|_{\dot{B}^s_{p,r}}+\|\Psi\|_{\dot{B}^{s+\f12}_{p,r}}\|v\|_{\dot{B}^2_{p,r}}\bigr)\quad
\mbox{ for}\ \ s\in (2,3],\\
&\|u\circ\Phi\|_{\dot{B}^s_{p,r}}\leq
C\bigl(\|\na\Psi\|_{\dot{B}^{\f{3}{p}}_{p,1}}\bigr)\bigl(1+\|\D\Psi\|_{B^{s-2}_{p,r}}\bigr)\|\na
u\|_{B^{s-1}_{p,r}}\quad\text{and}
\\
&\|v\circ\Phi^{-1}\|_{\dot{B}^s_{p,r}}\leq
C\bigl(\|\na\Psi\|_{\dot{B}^{\f{3}{p}}_{p,1}}\bigr)\bigl(1+\|\D\Psi\|_{B^{s-2}_{p,r}}\bigr)\|\na
v\|_{B^{s-1}_{p,r}}\quad \mbox{ for}\ \ s>3,
\end{split} \eeq  where $C(\lam)$ denotes a positive constant non-decreasingly depending on $\lam$. }
\end{lem}

\begin{proof} Let
\beno \cA=(a_{ij})_{i,j=1,2,3}\eqdefa I+\na_y\Psi,\quad
\cB=(b_{ij})_{i,j=1,2,3}\eqdefa (I+\na_y\Psi)^{-1}. \eeno Then due
to $\det\cA=1,$ the matrix $\cB$ equals the adjoint matrix of
$\cA$. This leads to \beq\label{app5}
(\p_{x_i}u)\circ\Phi=\sum_{j=1}^3b_{ji}\p_{y_j}(u\circ\Phi)\quad\text{and}
\quad(\p_{y_i}v)\circ\Phi^{-1}=\sum_{j=1}^3a_{ji}\circ\Phi^{-1}\p_{x_j}(v\circ\Phi^{-1}).
\eeq

In what follows, we shall only present  the  proof of the related
estimates involving $u\circ\Phi$, and the ones involving
$v\circ\Phi^{-1}$ are identical. We first deduce from Lemma 2.7 of
\cite{bcd} that \beno \|\D_j\bigl((\D_ku)\circ\Phi\bigr)\|_{L^p}\leq
C\min\bigl(2^{j-k},2^{k-j}\bigr)\|\na\Psi\|_{L^\infty}\|\D_ku\|_{L^p}\quad
\mbox{for all}\quad j,k\in \Z, \eeno so that for $s\in (-1,1)$ and
$u\in\dB^s_{p,r}(\R^3),$ one has \beno
\begin{split}
\|\D_j(u\circ\Phi)\|_{L^p}\leq
&\sum_{k\in\Z}\|\D_j\bigl((\D_ku)\circ\Phi\bigr)\|_{L^p}\\
\leq& C \Bigl(\sum_{k\leq
j}2^{k-j}+\sum_{k>j}2^{j-k}\Bigr)\|\na\Psi\|_{L^\infty}\|\D_k
u\|_{L^p}\\
\leq &C\|\na\Psi\|_{L^\infty}2^{-js} \Bigl(\sum_{k\leq
j}c_{k,r}2^{(k-j)(1-s)}+\sum_{k>j}c_{k,r}2^{(j-k)(1+s)}\Bigr)\|u\|_{\dB^s_{p,r}}\\
\leq
&Cc_{j,r}2^{-js}\|\na\Psi\|_{L^\infty}\|u\|_{\dB^s_{p,r}}\quad{for}\
\ (c_{j,r})_{j\in\Z}\in \ell^r(\Z).\end{split} \eeno This gives
\beq\label{qqqp} \|u\circ\Phi\|_{\dot{B}^s_{p,r}}\leq
C(\|\na\Psi\|_{L^\infty})\|u\|_{\dot{B}^s_{p,r}}\quad{for}\ \ s\in
(-1,1). \eeq Whereas we deduce from \eqref{app5} and \eqref{qqqp}
that \beno
\|u\circ\Phi\|_{\dot{B}^1_{p,r}}\leq
\|(\na_xu)\circ\Phi\cB\|_{\dB^0_{p,r}} \leq
C(\|\na\Psi\|_{L^\infty})\bigl(1+\|\na\Psi\|_{\dB^{\f3p}_{p,1}}\bigr)^2\|u\|_{\dot{B}^1_{p,r}}.
\eeno For $1<s\leq2$, we get, by using \eqref{app5} and product laws
in Besov spaces, that \beno
\begin{aligned}
\|u\circ\Phi\|_{\dot{B}^s_{p,r}}\leq &
\|(\na_xu)\circ\Phi\cB\|_{\dot{B}^{s-1}_{p,r}} \\
\leq & C
\bigl(1+\|\cB-I\|_{\dB^{\f3p}_{p,1}}\bigr)\|(\na_xu)\circ\Phi\|_{\dB^{s-1}_{p,r}}\leq
C(\|\na\Psi\|_{\dot{B}^{\f{3}{p}}_{p,1}})\|u\|_{\dot{B}^s_{p,r}}.
\end{aligned}\eeno
This proves the first  line of \eqref{apb1}.

On the other hand, it is easy to observe from Bony's decomposition
\eqref{C7} that \beno \|ab\|_{\dB^\tau_{p,r}}\lesssim
\|a\|_{L^\infty}\|b\|_{\dB^\tau_{p,r}}+\|a\|_{\dB^{\tau+\f12}_{p,r}}\|b\|_{\dB^1_{p,\infty}}\quad\mbox{for}\quad
\tau>0, \eeno from which, \eqref{app5}, and  the first line of
\eqref{apb1}, we infer that for $s\in (2,3]$ \beq\label{poit}
\begin{aligned}
&\|u\circ\Phi\|_{\dot{B}^s_{p,r}}\lesssim\|(\na_xu)\circ\Phi\cB\|_{\dot{B}^{s-1}_{p,r}}\\
&\lesssim(1+\|\na\Psi\|_{L^\infty})^2\|(\na
u)\circ\Phi\|_{\dot{B}^{s-1}_{p,r}}
+(1+\|\na\Psi\|_{L^\infty})\|\na\Psi\|_{\dot{B}^{s-\f12}_{p,r}}\|(\na u)\circ\Phi\|_{\dB^1_{p,r}}\\
&\leq
C\bigl(\|\na\Psi\|_{\dot{B}^{\f{3}{p}}_{p,1}}\bigr)\bigl(\|u\|_{\dot{B}^s_{p,r}}+\|\Psi\|_{\dot{B}^{s+\f12}_{p,r}}\|u\|_{\dot{B}^2_{p,r}}\bigr).
\end{aligned}\eeq
This proves the third line of \eqref{apb1}.

Inductively we assume that for $k\in\N$ and  $k+1<s-1\leq k+2$ ,
\beq  \label{app9a}
\begin{aligned} \|u\circ\Phi\|_{\dot{B}^{s-1}_{p,r}} \leq
C\bigl(\|\na\Psi\|_{\dot{B}^{\f{3}{p}}_{p,1}}\bigr)\bigl(\|u\|_{\dot{B}^{s-1}_{p,r}}
+\sum_{j=0}^{k-1}\|\Psi\|_{\dot{B}^{s-\f12-j}_{p,r}}\|u\|_{\dot{B}^{j+2}_{p,r}}\bigr).
\end{aligned}\eeq
Then by virtue of \eqref{app5} and \eqref{app9a}, we deduce that
\beno\begin{aligned}
&\|u\circ\Phi\|_{\dot{B}^s_{p,r}}\lesssim\|(\na_xu)\circ\Phi\cB\|_{\dot{B}^{s-1}_{p,r}}\\
&\leq C(\|\na\Psi\|_{\dot{B}^{\f{3}{p}}_{p,1}})\bigl(\|(\na
u)\circ\Phi\|_{\dot{B}^{s-1}_{p,r}}
+\|\na\Psi\|_{\dot{B}^{s-1}_{p,r}}\|(\na u)\circ\Phi\|_{L^\infty}\bigr)\\
&\leq
C(\|\na\Psi\|_{\dot{B}^{\f{3}{p}}_{p,1}})\Bigl(\|u\|_{\dot{B}^{s}_{p,r}}
+\sum_{j=0}^{k-1}\|\Psi\|_{\dot{B}^{s-\f12-j}_{p,r}}\|\na u\|_{\dot{B}^{j+2}_{p,r}}+\|\Psi\|_{\dot{B}^{s}_{p,r}}\|\na u\|_{L^\infty}\Bigr)\\
&\leq
C\bigl(\|\na\Psi\|_{\dot{B}^{\f{3}{p}}_{p,1}}\bigr)\bigl(1+\|\D\Psi\|_{B^{s-2}_{p,r}}\bigr)\|\na
u\|_{B^{s-1}_{p,r}},
\end{aligned}\eeno
which leads to the fifth inequality of \eqref{apb1}.
 This concludes the proof of Lemma \ref{funct}.
\end{proof}

\setcounter{equation}{0}
\section{The proof of Proposition \ref{LL1} }\label{appenda}

The proof of  Proposition \ref{LL1} will be based on the following
lemma:

\begin{lem}\label{lemappa}
{\sl Let $s>2+\f3p,$ $p\in (\f32,2),$  and $f\in B^s_{p,1}(\R^3)$
with $\Supp f(x_1,x_2,\cdot)\subset [-K,K]$ for some positive
constant $K.$ We assume moreover that $f$ and $\vv b_0$ are
admissible on $\R^2\times\{0\}$ in the sense of Definition
\ref{def1.1ad} and  \eqref{LL1pq} holds. Then \eqref{apa4}
 has a
solution $\psi\in B^s_{p,1}(\R^3)$ so that \beq \label{apa5}
\|\psi\|_{B^s_{p,1}}\leq C\bigl(K,\|\na \vv
b_0\|_{B^{s-1}_{p,1}}\bigr)\|f\|_{B^s_{p,1}}. \eeq}
\end{lem}

\begin{proof}  Due to \eqref{LL1pq}, \eqref{apa6} has a unique global solution on $\R$
so that for all $t\in\R,$ \beq \label{apa7} \|\na
X(t,\cdot)\|_{L^\infty}\leq \exp\Bigl(\|\na\vv
b_0\|_{L^\infty}|t|\Bigr)\ \ \ \and\ \ \ \det\Bigl(\frac{\p
X(t,x)}{\p x}\Bigr)=1. \eeq While it follows from \eqref{apa4} and
\eqref{apa6} that \beno \frac{d}{dt}\psi(X(t,x))=f(X(t,x)),\eeno
from which, we define \beno \psi(x)= \left\{\begin{array}{l}
\displaystyle -\int_0^\infty f(X(t,x))\,dt\quad\mbox{if}\quad x_3\geq 0, \\
\displaystyle \ \ \int_{-\infty}^0 f(X(t,x))\,dt\quad\mbox{if}\quad
x_3\leq 0.
\end{array}\right. \eeno
Thanks to the assumption that $f$ and $\vv b_0$ are admissible on
$\R^2\times\{0\}$ in the sense of Definition \ref{def1.1ad}, the
values of $\psi(x)$ at $(x_1,x_2,0)$ are compatible. We remark that
 $ b_0^3\p_3\psi=-b_0^1\p_1\psi-b_0^2\p_2\psi+f$ and
$b_0^3\geq\f12$ implies that the derivatives of $\psi$ in the
$x_1,x_2$ variables yields  the derivatives of $\psi$ with respect
to $x_3$ variable. Therefore, we do not require any admissible
condition for the derivatives of $f$ and $\vv b_0.$

On the other hand, it follows from \eqref{LL1pq} that $b^3_0\geq
\f12$ as long as $\e_0$ is small enough. So that we deduce from
\eqref{apa6} that \beno && X^3(t,x)\geq
x_3+\f{t}2\geq K \quad\ \ \mbox{if}\quad t\geq 2K,\,\ \  x_3\geq 0,\quad \andf\\
 && X^3(t,x)\leq
x_3+\f{t}2\leq -K \quad \mbox{if}\quad t\leq -2K,\, x_3\leq 0,\eeno
which together with the assumption: $\Supp f(x_1,x_2,\cdot)\subset
[-K,K]$ for some positive constant $K,$ implies that \beq
\label{apa8} \psi(x)= \left\{\begin{array}{l}
\displaystyle -\int_0^{2K} f(X(t,x))\,dt\quad\mbox{if}\quad x_3\geq 0, \\
\displaystyle \ \ \int_{-2K}^0 f(X(t,x))\,dt\quad\mbox{if}\quad
x_3\leq 0.
\end{array}\right. \eeq

With this solution formula for \eqref{apa4}, it amounts to prove
\eqref{apa5} in order to complete the proof of Lemma \ref{lemappa}.
Indeed for any $s>0,$ we deduce from \eqref{apa6} and product laws
in Besov spaces that for any $t\in [-2K,2K]$ \beq \label{apa9}
\begin{split} \|\na_xX(t,\cdot)-I\|_{\dB^s_{p,1}}\lesssim &
\int^{|t|}_0\Bigl(\|\na\vv b_0\|_{L^\infty}\|\na_xX(t',\cdot)-I\|_{\dB^s_{p,1}}\\
&+\|(\na\vv
b_0)(X(t',\cdot))\|_{\dB^s_{p,1}}\bigl(1+\|\na_xX(t',\cdot)-I\|_{L^\infty}\bigr)\Bigr)\,dt',
\end{split}
\eeq from which, \eqref{qqqp} and \eqref{apa7}, we get, by using
Gronwall's inequality, that \beno \max_{t\in
[-2K,2K]}\|\na_xX(t,\cdot)-I\|_{\dB^s_{p,1}}\leq C\bigl(K,\|\na \vv
b_0\|_{L^\infty}\bigr)\|\na\vv
b_0\|_{\dB^s_{p,1}}\quad\mbox{for}\quad s\in (0,1). \eeno Then for
$s\in (1,2)$ and $t\in [-2K,2K],$ we infer \beno
\begin{split}
\|f(X(t,\cdot))\|_{\dB^s_{p,1}}=&\|\na
f(X(t,\cdot))\na_xX(t,\cdot)\|_{\dB^{s-1}_{p,1}}\\
\lesssim &\|\na
f\|_{L^\infty}\|\na_xX(t,\cdot)-I\|_{\dB^{s-1}_{p,1}}\\
&+\|\na
f(X(t,\cdot))\|_{\dB^{s-1}_{p,1}}\bigl(1+\|\na_xX(t,\cdot)-I\|_{L^\infty}\bigr)\\
\leq& C\bigl(K,\|\na \vv b_0\|_{L^\infty}\bigr)\bigl(\|\na
f\|_{L^\infty}\|\vv b_0\|_{\dB^s_{p,1}}+\|f\|_{\dB^s_{p,1}}\bigr).
\end{split}
\eeno Notice that for $p\in (\f32,2),$ $\f3p\in (\f32,2),$ we thus
deduce from \eqref{apa9} that \beno
\begin{split}
\|\na_x X(t,\cdot)-I\|_{\dB^{\f3p}_{p,1}}\lesssim &
\int_0^{|t|}\|\na\vv b_0\|_{L^\infty}\|\na_xX(t',\cdot)-I\|_{\dB^{\f3p}_{p,1}}\,dt'\\
&+C\bigl(K,\|\na\vv b_0\|_{L^\infty}\bigr)\bigl(\|\na^2\vv
b_0\|_{L^\infty}\|\vv b_0\|_{\dB^{\f{3}{p}}_{p,1}}+\|\na\vv
b_0\|_{\dB^{\f3p}_{p,1}}\bigr),
\end{split}
\eeno for $t\in [-2K,2K].$ Applying Gronwall's inequality gives rise
to \beq\label{apa10} \max_{t\in [-2K,2K]}\|\na
X(t,\cdot)-I\|_{\dB^{\f3p}_{p,1}}\leq C\bigl(K,\|\na \vv
b_0\|_{B^{1+\f3p}_{p,1}}\bigr)\|\na\vv b_0\|_{B^{\f3p}_{p,1}}. \eeq
While it is easy to observe from \eqref{poit} that \beno
\|u\circ\Phi\|_{\dB^s_{p,1}}\leq
C\bigl(\|\na\Psi\|_{\dB^{\f3p}_{p,1}}\bigr)\bigl(\|u\|_{\dB^s_{p,1}}+\|\na\Psi\|_{\dB^s_{p,1}}\|
u\|_{\dB^{\f3p}_{p,1}}\bigr) \eeno for $s\in (2,3],$  so that for
$2<s-1\leq 3,$ \eqref{apa9} implies \beno
\begin{split}
\|\na X(t,\cdot)&-I\|_{\dB^{s-1}_{p,1}}\lesssim
\int_0^{|t|}\Bigl(\|\na\vv b_0\|_{L^\infty}\|\na
X(t',\cdot)-I\|_{\dB^{s-1}_{p,1}}\\
&+C\bigl(K,\|\na\vv b_0\|_{B^{1+\f3p}_{p,1}}\bigr)\bigl(\|\na\vv
b_0\|_{\dB^{s-1}_{p,1}}+\|\na
X(t',\cdot)-I\|_{\dB^{s-1}_{p,1}}\|\na\vv
b_0\|_{\dB^{\f3p}_{p,1}}\Bigr)\,dt',
\end{split}
\eeno for $t\in [-2K,2K],$ applying Gronwall's inequality gives
\beq\label{apa11} \max_{t\in [-2K,2K]}\|\na
X(t,\cdot)-I\|_{\dB^{s-1}_{p,1}}\leq C\bigl(K,\|\na \vv
b_0\|_{B^{1+\f3p}_{p,1}}\bigr)\|\na\vv
b_0\|_{\dB^{s-1}_{p,1}}\quad\mbox{for}\quad s\in (3,4]. \eeq For
$s>4,$ we deduce from \eqref{apb1} and \eqref{apa9} that \beno
\begin{split}
\|\na X(t,\cdot)&-I\|_{\dB^{s-1}_{p,1}}\lesssim
\int_0^{|t|}\Bigl(\|\na\vv b_0\|_{L^\infty}\|\na
X(t',\cdot)-I\|_{\dB^{s-1}_{p,1}}\\
&+C\bigl(K,\|\na\vv b_0\|_{B^{1+\f3p}_{p,1}}\bigr)\bigl(1+\|\na
X(t',\cdot)-I\|_{B^{s-2}_{p,1}}\bigr)\|\na^2\vv
b_0\|_{B^{s-2}_{p,1}}\Bigr)\,dt'.
\end{split}
\eeno Whereas similar to \eqref{apa9}, one has \beno \|\na
X(t,\cdot)-I\|_{L^p}\lesssim \int_0^{|t|}\Bigl(\|\na\vv
b_0\|_{L^\infty}\|\na X(t',\cdot)-I\|_{L^p}+\|\na\vv
b_0\|_{L^p}\Bigr)\,dt'. \eeno As a consequence, we obtain \beno
\begin{split}
\|\na X(t,\cdot)&-I\|_{B^{s-1}_{p,1}}\leq C\bigl(K,\|\na\vv
b_0\|_{B^{s-1}_{p,1}}\bigr) \int_0^{|t|}\Bigl(\|\na
X(t',\cdot)-I\|_{B^{s-1}_{p,1}}+\|\na\vv
b_0\|_{B^{s-1}_{p,1}}\Bigr)\,dt'.
\end{split}
\eeno Applying Gronwall's inequality leads to \beq\label{apa12}
\max_{t\in [-2K,2K]}\|\na X(t,\cdot)-I\|_{B^{s-1}_{p,1}}\leq
C\bigl(K,\|\na \vv b_0\|_{B^{s-1}_{p,1}}\bigr)\|\na\vv
b_0\|_{B^{s-1}_{p,1}}\quad\mbox{for}\quad s>4. \eeq

Finally we deduce from \eqref{apb1} and \eqref{apa8} that \beno
\begin{split}
\|\psi\|_{\dB^s_{p,1}}\leq \int_{-2K}^{2K}C\bigl(\|\na
X(t',\cdot)-I\|_{\dB^{\f3p}_{p,1}}\bigr)\Bigl(1+\|\na
X(t',\cdot)-I\|_{B^{s-1}_{p,1}}\Bigr)\|\na f\|_{B^{s-1}_{p,1}}\,dt',
\end{split}
\eeno which together with \beno \|\psi\|_{L^p}\leq \|f\|_{L^p} \eeno and
\eqref{apa11} and \eqref{apa12} concludes the proof of \eqref{apa5}.
\end{proof}

\begin{proof}[Proof of Proposition \ref{LL1}] Under the assumption of \eqref{LL1pq}, we would first like to find a
solution $(\psi_1,\psi_2)\in B^s_{p,1}(\R^3)$ to the  following
system:
\begin{equation}\label{apa1}
 \left\{\begin{array}{l}
\displaystyle (1-\p_{x_2}\psi_2)\p_{x_3}\psi_1+\p_{x_3}\psi_2\p_{x_2}\psi_1=b_0^1,\quad\mbox{for}\quad x\in\R^3, \\
\displaystyle (1-\p_{x_1}\psi_1)\p_{x_3}\psi_2+\p_{x_3}\psi_1\p_{x_1}\psi_2=b_0^2 \\
\end{array}\right.
\end{equation}
If we use the standard iteration scheme to solve the above problem,
the iterated solutions will lose derivative on each step. However,
notice that \begin{multline*}
\p_{x_1}\bigl(\p_{x_2}\psi_1\p_{x_3}\psi_2+\p_{x_3}\psi_1(1-\p_{x_2}\psi_2)\bigr)+\p_{x_2}\bigl(
\p_{x_3}\psi_1\p_{x_1}\psi_2+\p_{x_3}\psi_2(1-\p_{x_1}\psi_1)\bigr)\\+\p_{x_3}\bigl(
(1-\p_{x_1}\psi_1)(1-\p_{x_2}\psi_2)-\p_{x_2}\psi_1\p_{x_1}\psi_2\bigr)=0.
\end{multline*}
This along with $\dive \vv b_0=0$ ensures that
 \beno \p_{x_3}\bigl(b_0^3-
(1-\p_{x_1}\psi_1)(1-\p_{x_2}\psi_2)+\p_{x_2}\psi_1\p_{x_1}\psi_2\bigr)=0.\eeno
Since $b_0^3-1\in B_{p,1}^s(\R^3),$ we conclude that \beq
\label{apa2} b_0^3=
(1-\p_{x_1}\psi_1)(1-\p_{x_2}\psi_2)-\p_{x_2}\psi_1\p_{x_1}\psi_2.
\eeq With \eqref{apa2}, we solve \eqref{apa1} for $\p_{x_3}\psi_1$ and
$\p_{x_3}\psi_2$ from \eqref{apa1} \beno
\begin{pmatrix} \p_{x_3}\psi_1 \\
\p_{x_3}\psi_2
\end{pmatrix}=\f1{b_0^3}\begin{pmatrix} 1-\p_{x_1}\psi_1& -\p_{x_2}\psi_1 \\
-\p_{x_1}\psi_2&1-\p_{x_2}\psi_2
\end{pmatrix}\begin{pmatrix} b_0^1 \\
b_0^2
\end{pmatrix}, \eeno
or equivalently \beq \label{apa3}
\begin{split}
b_0^1\p_{x_1}\psi_1+b_0^2\p_{x_2}\psi_1+b_0^3\p_{x_3}\psi_1=b_0^1,\\
b_0^1\p_{x_1}\psi_2+b_0^2\p_{x_2}\psi_2+b_0^3\p_{x_3}\psi_2=b_0^2.
\end{split}
\eeq Thanks to \eqref{LL1pq} and Lemma \ref{lemappa}, \eqref{apa3}
has a solution $(\psi_1,\psi_2)$ so that \beq \label{apa13}
\bigl\|\bigl(\psi_1,\psi_2\bigr)\bigr\|_{B^s_{p,1}}\leq
C(K,\e_0)\bigl\|\bigl(b_0^1,b_0^2\bigr)\bigr\|_{B^s_{p,1}}. \eeq
Whereas for $\vv \Psi=\bigl(\psi_1,\psi_2,\psi_3\bigr)^T,$ we deduce
from $\mbox{det}\bigl(I-\na\vv \Psi\bigr)=1$ that
\begin{multline*}
\bigl(1-\p_1\psi_1-\p_2\psi_2+\p_1\psi_1\p_2\psi_2-\p_2\psi_1\p_1\psi_2\bigr)\p_3\psi_3+
\bigl(\p_2\psi_1\p_3\psi_2+\p_3\psi_1(1-\p_2\psi_2)\bigr)\p_1\psi_3\\
+\bigl(\p_3\psi_1\p_1\psi_2+(1-\p_1\psi_1)\p_3\psi_2\bigr)\p_2\psi_3=-\p_1\psi_1-\p_2\psi_2+\p_1\psi_1\p_2\psi_2-\p_2\psi_1\p_1\psi_2,
\end{multline*}
which together with \eqref{apa1} and \eqref{apa2} yields
\beq\label{apa14}
b_0^1\p_{x_1}\psi_3+b_0^2\p_{x_2}\psi_3+b_0^3\p_{x_3}\psi_3=b_0^3-1.\eeq
Along the same line to the proof of \eqref{apa13}, \eqref{apa14} has
a solution $\psi_3$ so that
 \beno
\|\psi_3\|_{B^s_{p,1}}\leq C(K,\e_0)\|b_0^3-1\|_{B^s_{p,1}}. \eeno
This together with \eqref{apa13} leads to \eqref{B3p0}. And thus
\eqref{B3p3} follows from \eqref{B3p2}.

 Finally observing that $U_0$
defined in Proposition \ref{LL1} is in fact the adjoint matrix of
$I-\na_x\vv \Psi,$ $U_0$ automatically satisfies \eqref{B3}. This
finishes the proof of Proposition \ref{LL1}.
\end{proof}

\noindent {\bf Acknowledgments.} We would like to thank Professor
Fanghua Lin for profitable discussions on this topic. After this
work has been finished, Professor Lin told us that he and Ting Zhang
 proved a similar wellposedness result for a modified MHD system. We
 also would like to thank Dr. Zhen Lei for pointing out a mistake in
 the proof of Lemma \ref{lemappa} in the  preliminary version of
 this paper.

Part of this work was done when the second author was visiting the
laboratory  of mathematics of Universit\'e Paris-Sud  in April 2013.
 He would like to thank the hospitality of the
laboratory. Both authors are supported by innovation grant from
National Center for Mathematics and Interdisciplinary Sciences. L.
Xu is partially supported by NSF of China under Grant 11201455.
 P. Zhang is partially supported by NSF
of China under Grant 10421101 and 10931007.


\medskip

\begin{thebibliography}{99}

\bibitem{AP} H. Abidi and M. Paicu,  Global existence for the
magnetohydrodynamic system in critical spaces, {\it Proc. Roy. Soc.
Edinburgh Sect. A}, {\bf 138} (2008),  447-476.

\bibitem{bcd} H. Bahouri, J.~Y. Chemin and R. Danchin, {\it  Fourier Analysis and Nonlinear
Partial Differential Equations,}  Grundlehren der Mathematischen
Wissenschaften, Springer 2011.

\bibitem{Bo} J.~M. Bony,  Calcul symbolique et propagation des
singularit\'es pour les \'quations aux d??riv\'ees partielles non
lin\'eaires, {\it Ann. Sci. \'Ecole Norm. Sup.}, {\bf 14}(4) (1981),
209--246.

\bibitem{Ca} H. Cabannes, {\it Theoretical Magnetofludynamics}, AP, New York, 1970.

\bibitem{CW} C. Cao and J. Wu,  Global regularity for the 2D MHD
equations with mixed partial dissipation and magnetic diffusion,
{\it Adv. Math.}, {\bf 226} (2011),  1803-1822.

\bibitem{CRW} C. Cao, D. Regmi and J. Wu,  The 2D MHD
equations with horizontal dissipation and horizontal magnetic
diffusion, {\it J. Differential Equations}, {\bf  254}  (2013),
2661-2681.




\bibitem{CDGG} J.~Y. Chemin, B. Desjardins, I. Gallagher and
E. Grenier, Fluids with anisotropic viscosity, {\it M2AN Math.
Model. Numer. Anal.}, {\bf 34} (2000),  315-335.

\bibitem{CPZ1} J.~Y. Chemin,  M. Paicu and P. Zhang, Global large solutions  to 3-D inhomogeneous
 Navier-Stokes system with one slow variable, preprint 2012.

\bibitem{CZ}J. Y. Chemin and P.  Zhang, On the global wellposedness  to the 3-D incompressible anisotropic
   Navier-Stokes equations, {\it Comm. Math. Phys.}, {\bf 272} (2007),
   529--566.

\bibitem{Ch-Zh} Y. Chen and P. Zhang, The global existence of small solutions to
the incompressible viscoelastic fluid System in general space
dimensions,  {\it Comm. P.  D. E.}, {\bf 31} (2006), 1793-1810.


\bibitem{ChCa} C. Chiuderi and F. Califano, {\it Phy. Rev. E}, {\bf 60} (1999), 4701-4707.

\bibitem{Da} R. Danchin,  Global existence in critical spaces for compressible
Navier-Stokes equations, {\it Invent. Math.}, {\bf 141} (2000),
579-614.

\bibitem{CP} T.~G.~Cowling and D. Phil, {\it Magnetohydrodynnamics}, the Institute of
Physics, 1976.

\bibitem{DL} G. Duvaut and J.~L. Lions, In\'equations en thermo\'elasticit\'e et
magn\'etohydrodynamique, {\it Arch. Ration. Mech. Anal.}, {\bf 46}
(1972), 241-279.


\bibitem{GZ2} G. Gui, J. Huang and P. Zhang, Large global solutions to the
$3-$D inhomogeneous Navier-Stokes equations,  {\it J. Funct. Anal.
}, {\bf 261} (2011), 3181-3210;

\bibitem{DI} D. Iftimie, The resolution of the Navier-Stokes equations in anisotropic
spaces, {\it Rev.Mat.Iberoamericana}, {\bf 15} (1999),  1--36.

\bibitem{LL} L.~ D. Landau and E.~M. Lifshitz, {\it Electrodynamics of Continuous Media}, 2nd ed.
Pergaman, New York, 1984.

\bibitem{LLZhen} Z. Lei,  C. Liu and Y.  Zhou, Global solutions for
incompressible viscoelastic fluids, {\it Arch. Ration. Mech. Anal.},
{\bf 188} (2008),  371-398.

\bibitem{Lin} F. Lin, Some analytical issues for elastic complex
fluids,  {\it Comm. Pure Appl. Math.},   {\bf 65} (2012), 893-919.


\bibitem{LLZ} F. Lin, C. Liu and P. Zhang, On the
hydrodynamics of visco-elasicity,  {\it Comm. Pure Appl. Math.},
{\bf 58} (2005), 1437-1471.


\bibitem{LZ} F. Lin and P. Zhang, Global small solutions  to MHD type system (I): 3-D
case,  {\it Comm. Pure Appl. Math.}, (2013) in press.

\bibitem{XLZMHD1}  F. Lin, L. Xu and P. Zhang, Global small solutions  to 2-D MHD
system,  arXiv:1302.5877.



\bibitem{Majda} A. Majda, {\it Compressible fluid flow and systems of conservation
laws in several space variables,} Applied Mathematical Sciences, 53.
Springer-Verlag, New York, 1984.


\bibitem{Pa02} M. Paicu, \'Equation anisotrope
de Navier-Stokes dans des espaces  critiques, {\it  Revista
Matematica Iberoamericana}, {\bf 21} (2005), 179--235.

\bibitem{PZ1} M. Paicu and P. Zhang, Global solutions to the 3-D incompressible
 anisotropic Navier-Stokes system  in the critical spaces,  {\it Comm. Math. Phys.},  {\bf 307} (2011), 713-759.

\bibitem{ST} M. Sermange and R. Temam, Some mathematical questions related to the MHD
equations, {\it Comm. Pure Appl. Math.}, {\bf 36} (1983),  635-664.

\bibitem{XZZ}  L. Xu, P. Zhang and Z. Zhang, Global solvability to a free boundary 3-D incompressible
viscoelastic fluid system with surface tension, {\it Arch. Ration.
Mech.
 Anal.}, {\bf 208} (2013), 753-803.


\bibitem{Zhangt2} T. Zhang,  Erratum to: Global wellposed problem for the 3-D
incompressible anisotropic Navier-Stokes equations in an anisotropic
space, {\it Comm. Math. Phys.}, {\bf 295} (2010),  877-884.


\end{thebibliography}
\end{document}